\documentclass[12pt,a4paper,reqno]{amsart}

\usepackage[dvipsname,usenames]{xcolor}
\usepackage[utf8]{inputenc}
\usepackage{etoolbox}

\usepackage{subcaption}

\usepackage[T1]{fontenc}
\usepackage[american]{babel}
\usepackage{amsmath,amssymb,amsthm,amscd,bm}
\usepackage{pgf,tikz}
\usepackage{mathrsfs}
\usepackage{enumerate}
\usetikzlibrary{arrows}
\usepackage{paralist}
\usepackage{dsfont}
\usepackage{xfrac}
\usepackage{comment}
\usepackage{esint}

\usepackage{textcomp}

\let\OLDthebibliography\thebibliography
\renewcommand\thebibliography[1]{
  \OLDthebibliography{#1}
  \setlength{\parskip}{3pt}
  \setlength{\itemsep}{2.5pt plus 0.5ex}
}

\usepackage[colorlinks,linkcolor={black},citecolor={black},urlcolor={black}]{hyperref}
\usepackage[margin=2.6cm,bmargin=3.2cm,tmargin=3.2cm]{geometry}

\renewcommand{\tilde}{\widetilde}

\newcommand{\KR}{\mathrm{KR}}
\newcommand{\TV}{\mathrm{TV}}
\newcommand{\tKR}{\widetilde{\mathrm{KR}}}
\newcommand{\BV}{\mathrm{BV}}
\newcommand{\BVKR}{\BV_{\KR}(\cI; \cM_+(\Td))}
\newcommand{\WKR}{W_\KR^{1,1}(\cI; \cM_+(\Td))}

\newcommand{\sa}{\mathsf{v}}
\newcommand{\sd}{\mathsf{d}}
\newcommand{\sm}{\mathsf{m}}

\newcommand{\sz}{\mathsf{z}}

\newcommand{\sJ}{\mathsf{J}}

\renewcommand{\hom}{\mathrm{hom}}

\DeclareMathOperator{\Eff}{\mathsf{Eff}}
\newcommand{\Effg}{\mathsf{Eff}_{\mathrm{geo}}}
\DeclareMathOperator{\Rep}{{\mathsf{Rep}}}
\DeclareMathOperator*{\Conv}{Conv}

\newcommand{\cA}{\mathcal{A}}

\newcommand{\cE}{\mathcal{E}}
\newcommand{\cF}{\mathcal{F}}
\newcommand{\cH}{\mathcal{H}}
\newcommand{\cI}{\mathcal{I}}

\newcommand{\cM}{\mathcal{M}}
\newcommand{\cP}{\mathcal{P}}

\newcommand{\cS}{\mathcal{S}}
\newcommand{\cT}{\mathcal{T}}

\newcommand{\cZ}{\mathcal{Z}}

\newcommand{\cW}{\mathcal{W}}

\newcommand{\cX}{\mathcal{X}}

\newcommand{\cC}{\mathcal {C}}
\newcommand{\cMA}{\mathcal{MA}}
\newcommand{\bMA}{\mathbb{MA}}
\newcommand{\V}{V}

\newcommand{\XQ}{\cX^Q}
\newcommand{\EQ}{\cE^Q}

\newcommand{\Haus}{\mathscr{H}}

\newcommand{\bA}{\mathbb{A}}

\newcommand{\bF}{\mathbb{F}}

\newcommand{\T}{\mathbb{T}}

\newcommand{\one}{{{\bf 1}}}

\newcommand{\Lip}{\mathrm{Lip}}
\newcommand{\Leb}{\mathscr{L}}

\newcommand{\bCE}{\mathbb{CE}}
\newcommand{\cCE}{\mathcal{CE}}
\newcommand{\CE}{\mathsf{CE}}

\newcommand{\R}{\mathbb{R}}
\newcommand{\Z}{\mathbb{Z}}
\newcommand{\N}{\mathbb{N}}

\DeclareMathOperator{\id}{id}

\newcommand{\eps}{\varepsilon}
\newcommand{\conv}{\mathrm{conv}}

\newcommand{\Meps}{\R_+^{\cX_\eps}}
\newcommand{\Mdeps}{\R_a^{\cE_\eps}}

\newcommand{\MT}{\cM_+(\T^d)}
\newcommand{\MdT}{\cM^d(\T^d)}

\newcommand{\tand}{\quad\text{ and }\quad}

\newcommand{\dd}{\, \mathrm{d}}
\newcommand{\ddd}{\mathrm{d}}

\newcommand{\ip}[1]{\langle {#1}\rangle}
\newcommand{\bip}[1]{\big\langle {#1}\big\rangle}

\newcommand{\suchthat}{\ensuremath{\ : \ }} % such that inside, for example, the sets definition

\renewcommand{\hat}{\widehat}

\theoremstyle{plain}
\newtheorem{theorem}{Theorem}[section]
\newtheorem{corollary}[theorem]{Corollary}
\newtheorem{lemma}[theorem]{Lemma}
\newtheorem{proposition}[theorem]{Proposition}

\newtheorem{definition}[theorem]{Definition}
\newtheorem{assumption}[theorem]{Assumption}

\theoremstyle{remark}
\newtheorem{remark}[theorem]{Remark}
\newtheorem{example}[theorem]{Example}

\newtheorem*{claim*}{Claim}
\newtheorem*{remark*}{Remark}
\newtheorem*{example*}{Example}
\newtheorem*{notation*}{Notation}
\numberwithin{equation}{section}

\definecolor{jan}{rgb}{0.0,0.2,0.5}
\definecolor{mat}{rgb}{0.0,0.5,0.3}

\def\bfm{{\pmb{m}}}  
  \def\bfr{\pmb{r}}
  \def\bfu{\pmb{u}}

\def\bfJ{\pmb{J}}   
\def\bfM{\pmb{M}}  
  
  \def\bfU{\pmb{U}}
\def\bfV{\pmb{V}}

\def\BS{\boldsymbol}

    \def\bfmu{{\BS\mu}}
\def\bfnu{{\BS\nu}}            \def\bfxi{{\BS\xi}}
            
\def\bfsigma{{\BS\sigma}}

\DeclareMathOperator{\dive}{\mathsf{div}}
\DeclareMathOperator{\supp}{supp}
\DeclareMathOperator{\diam}{diam}

\newcommand{\bW}{\mathbb{W}}

\newcommand{\Td}{\mathbb{T}^d}

\newcommand{\fm}{\mathfrak{m}}

\makeatletter
\def\moverlay{\mathpalette\mov@rlay}
\def\mov@rlay#1#2{\leavevmode\vtop{%
		\baselineskip\z@skip \lineskiplimit-\maxdimen
		\ialign{\hfil$\m@th#1##$\hfil\cr#2\crcr}}}
\newcommand{\charfusion}[3][\mathord]{
	#1{\ifx#1\mathop\vphantom{#2}\fi
		\mathpalette\mov@rlay{#2\cr#3}
	}
	\ifx#1\mathop\expandafter\displaylimits\fi}
\makeatother

\renewcommand{\phi}{\varphi}

\newcommand{\tP}{\mathrm{P}}
\newcommand{\etP}{\mathsf{P}}
\newcommand{\cG}{\mathcal{G}}

\DeclareMathOperator{\Dom}{\mathsf{D}}

\title{Homogenisation of dynamical optimal transport on periodic graphs}
\author{Peter Gladbach}
\address{Institut f\"ur angewandte Mathematik, Universit\"at Bonn, Endenicher Allee 60, 53115 Bonn, Germany}
\email{gladbach@iam.uni-bonn.de}

\author{Eva Kopfer}
\address{Institut f\"ur angewandte Mathematik, Universit\"at Bonn, Endenicher Allee 60, 53115 Bonn, Germany}
\email{eva.kopfer@iam.uni-bonn.de}

\author{Jan Maas}
\address{Institute of Science and Technology Austria (IST Austria),
Am Campus 1, 3400 Klosterneuburg, Austria}
\email{jan.maas@ist.ac.at}

\author{Lorenzo Portinale}
\address{Institut f\"ur angewandte Mathematik, Universit\"at Bonn, Endenicher Allee 60, 53115 Bonn, Germany}
\email{portinale@iam.uni-bonn.de}

\let\oldtocsection=\tocsection

\let\oldtocsubsection=\tocsubsection

\let\oldtocsubsubsection=\tocsubsubsection

\renewcommand{\tocsection}[2]{\vspace{0.3ex}\hspace{0em}\oldtocsection{#1}{#2}}
\renewcommand{\tocsubsection}[2]{\vspace{0.3ex}\hspace{1em}\oldtocsubsection{#1}{#2}}
\renewcommand{\tocsubsubsection}[2]{\hspace{2em}\oldtocsubsubsection{#1}{#2}}

\begin{document}

\begin{abstract}
This paper deals with the large-scale behaviour of dynamical optimal transport on $\Z^d$-periodic graphs with general lower semicontinuous and convex energy densities. 
Our main contribution is a homogenisation result that describes the effective behaviour of the discrete problems in terms of a continuous optimal transport problem. The effective energy density can be explicitly expressed in terms of a cell formula, which is a finite-dimensional convex programming problem that depends non-trivially on the local geometry of the discrete graph and the discrete energy density.

Our homogenisation result is derived from a $\Gamma$-convergence result for action functionals on curves of measures, which we prove under very mild growth conditions on the energy density.  
We investigate the cell formula in several cases of interest, including finite-volume discretisations of the Wasserstein distance, where non-trivial limiting behaviour occurs.

\end{abstract}

\maketitle

\setcounter{tocdepth}{1}
\tableofcontents

\section{Introduction}

In the past decades there has been intense research activity in the field of optimal transport, both in pure mathematics and in applied areas  \cite{Villani:2003,Villani:2009,santambrogio,Peyre-Cuturi:2019}. 
In continuous settings, a central result in the field is the \emph{Benamou--Brenier formula} \cite{Benamou-Brenier:2000}, which establishes the equivalence of static and dynamical optimal transport. 
It asserts that the classical Monge--Kantorovich problem, in which a cost functional is minimised over couplings of given probability measures $\mu_0$ and $\mu_1$, is equivalent to a dynamical transport problem, in which an energy functional is minimised over all solutions to the continuity equation connecting $\mu_0$ and $\mu_1$.

In discrete settings, the equivalence between static and dynamical optimal transport breaks down, and it turns out that the dynamical formulation \cite{Maas:2011,Mielke:2011,CHLZ11} is essential in applications to evolution equations,
 discrete Ricci curvature,
 and functional inequalities \cite{Erbar-Maas:2012,Mielke:2013,Erbar-Maas:2014, Erbar-Maas-Tetali:2015, Fathi-Maas:2016,Erbar-Henderson-Menz:2017,Erbar-Fathi:2018}.
Therefore, it is an important problem to analyse the discrete-to-continuum limit of dynamical optimal transport in various setting.

This limit passage turns out to be highly nontrivial. In fact, seemingly natural discretisations of the Benamou--Brenier formula do not necessarily converge to the expected limit, 
even in one-dimensional settings \cite{gladbach2020}. 
The main result in \cite{GlKoMa18} asserts that, for a sequence of meshes on a bounded convex domain in ${\mathbb{R}}^d$, an isotropy condition on the meshes is required to obtain the convergence of the discrete dynamical transport distances to ${\mathbb{W}}_2$. 
This is in sharp contrast to the scaling behaviour of the corresponding gradient flow dynamics, for which no additional symmetry on the meshes is required to ensure the convergence of discretised evolution equations to the expected continuous limit \cite{Disser-Liero:2015,Forkert2020}.

The goal of this paper is to investigate the large-scale behaviour of dynamical optimal transport on  graphs with a $\Z^d$-periodic structure. 
Our main contribution is a homogenisation result that describes the effective behaviour of the discrete problems
in terms of a continuous optimal transport problem, in which the effective energy density depends non-trivially on the geometry of the discrete graph and the discrete transport costs.

\subsection*{Main results}

We give here an informal presentation of the main results of this paper, ignoring several technicalities for the sake of readability.
Precise formulations and a more general setting can be found from Section \ref{sec:discrete} onwards.

\subsubsection*{Dynamical optimal transport in the continuous setting}

For $1 \leq p < \infty$, 
let $\bW_p$ be the Wasserstein--Kantorovich--Rubinstein distance between probability measures on a metric space $(X,\sd)$: for $\mu^0, \mu^1 \in \cP(X)$,
\begin{align*}
	\bW_p(\mu^0, \mu^1)
	:= 
	\inf_{
		\gamma 
		 \in \Gamma(\mu^0,\mu^1)
		}
	\bigg\{ 
		\int_{X \times X}
			\sd(x,y)^p
		\dd \gamma(x,y)
	\bigg\}^{1/p},
\end{align*}
where $\Gamma(\mu^0,\mu^1)$ denotes the set of couplings of $\mu^0$ and $\mu^1$, i.e., all measures $\gamma \in \cP(X \times X)$ with marginals $\mu^0$ and $\mu^1$.
On the torus $\T^d$ (or more generally, on Riemannian manifolds),
the Benamou--Brenier formula \cite{Benamou-Brenier:2000, Ambrosio-Gigli-Savare:2008} provides an equivalent dynamical formulation for $p > 1$, namely
\begin{align}
	\label{eq: bb}
\bW_p(\mu^0, \mu^1) 
	 = \inf_{(\rho, j)
	  } 
		\bigg\{ 
			\int_0^1 \int_{\T^d} 
				\frac{|j_t(x)|^p}{\rho_t^{p-1}(x)}
			\dd x \dd t
		\bigg\}^{1/p},
\end{align}
where the infimum runs over all solutions $(\rho, j)$  to the continuity equation 
	$\partial_t \rho + \nabla \cdot j = 0$ with boundary conditions 
	$\rho_0(x) \dd x = \mu^0(\ddd x)$
and 	
	$\rho_1(x) \dd x = \mu^1(\ddd x)$.

\medskip

In this paper we consider general lower semicontinuous and convex energy densities 
 	$f : \R_+ \times \R^d \to \R \cup \{ + \infty\}$ 
under suitable (super-)linear growth conditions.
(The Benamou--Brenier formula above corresponds to the special case 
	$f(\rho,j) = \frac{|j|^p}{\rho^{p-1}}$).
For sufficiently regular curves of measures 	
$
\bfmu  : (0,1) \to \cM_+(\Td) 
$, 
we consider the continuous action 
\begin{align}	\label{eq:intro_A}
	\bA(\bfmu) 
	:= \inf_{\bfnu
	} 
        \bigg\{ 
            \int_0^1
            \int_{\T^d}
            f
                \bigg( 
                    \frac{ \ddd \mu_t}{\ddd \Leb^d}, 
                    \frac{ \ddd \nu_t}{\ddd \Leb^d}
                \bigg) 
            \dd x 	
            \dd t 
             \ : \
            (\bfmu, \bfnu) \in \bCE
        \bigg\}.
\end{align}
Here, the infimum runs over all time-dependent vector-valued measures 
$\bfnu : (0,1) \to \cM^d(\T^d)$ 
satisfying the continuity equation $(\bCE)$ 
	$\partial_t \mu_t + \nabla \cdot \nu_t = 0$ 
in the sense of distributions.

\subsubsection*{Dynamical optimal transport in the discrete setting} 

A  natural discrete counterpart to \eqref{eq:intro_A} can be defined 
on finite (undirected) graphs $(\cX, \cE)$.
For each edge $(x,y) \in \cE$ we fix a lower semicontinuous and convex energy density\footnote{
	In the sequel we consider more general discrete energy densities $F(m,J)$, not necessarily sums of edge-energies.}
	$F_{xy}: \R_+ \times \R_+ \times \R \to \R_+$.
For sufficiently regular curves
	$\bfm : (0,1) \to \cM_+(\cX)$
we then consider the discrete action
\begin{align}	\label{eq:intro_cA}
	\cA(\bfm) := \inf_{\bfJ}
		\bigg\{
			\int_0^1
				\sum_{(x,y) \in \cE}
					F_{xy}
					\big(
						m_t(x), m_t(y), J_t(x,y) 
					\big) 
						\dd t
			\suchthat
				(\bfm, \bfJ) \in \cCE		
		\bigg\}.
\end{align}
Here, the infimum runs over all time-dependent ``discrete vector fields'', i.e., all anti-symmetric functions
	$\bfJ : (0,1) \to \R^\cE$
satisfying the 
discrete continuity equation $(\cCE)$
	$\partial_t m_t(x) 
		+ \dive J_t(x) 
		= 0
	$ 
for all $x \in \cX$,
where
	$\dive J_t(x):= 
	\sum_{y: (x,y) \in \cE} J_t(x,y)$
denotes the discrete divergence.
Variational problems of the form \eqref{eq:intro_cA} arise naturally in the formulation of jump processes as generalised gradient flows \cite{Peletier-Rossi-Savare-Tse:2020}.

\subsubsection*{Dynamical optimal transport on $\Z^d$-periodic graphs} 

In this work we fix a $\Z^d$-periodic graph $(\cX,\cE)$ embedded in $\R^d$, as in Figure \ref{fig:per_graph_intro}. 
For sufficiently small $\eps > 0$ with 
	$1/\eps \in \N$,
we then consider the finite graph 
	$(\cX_\eps, \cE_\eps)$ 
obtained by scaling $(\cX, \cE)$ by a factor $\eps$, 
and wrapping the resulting graph around the torus, 
so that the resulting graph is embedded in $\Td$. 
We are interested in the behaviour of the rescaled discrete action, defined for curves 
	$\bfm : (0,1) \to \cM_+(\cX_\eps)$
by
\begin{align}
	\label{eq:Aeps-intro}
	\cA_\eps(\bfm) := \inf_{\bfJ}
		\bigg\{
			\int_0^1
				\sum_{(x,y) \in \cE_\eps}
				\eps^d 
					F_{xy}
				\bigg(
					\frac{m_t(x)}{\eps^d},
					\frac{m_t(y)}{\eps^d}, \frac{J_t(x,y)}{\eps^{d-1}}
				\bigg) 
			\dd t
				\suchthat
			(\bfm, \bfJ) \in \cCE_\eps	
		\bigg\}.	
\end{align}
As above, the infimum runs over all time-dependent ``discrete vector fields''
	$\bfJ : (0,1) \to \R^{\cE_\eps}$
satisfying the 
discrete continuity equation $(\cCE_\eps)$
on the rescaled graph $(\cX_\eps, \cE_\eps)$.
\begin{figure}[h!]
	\includegraphics[scale=0.25]{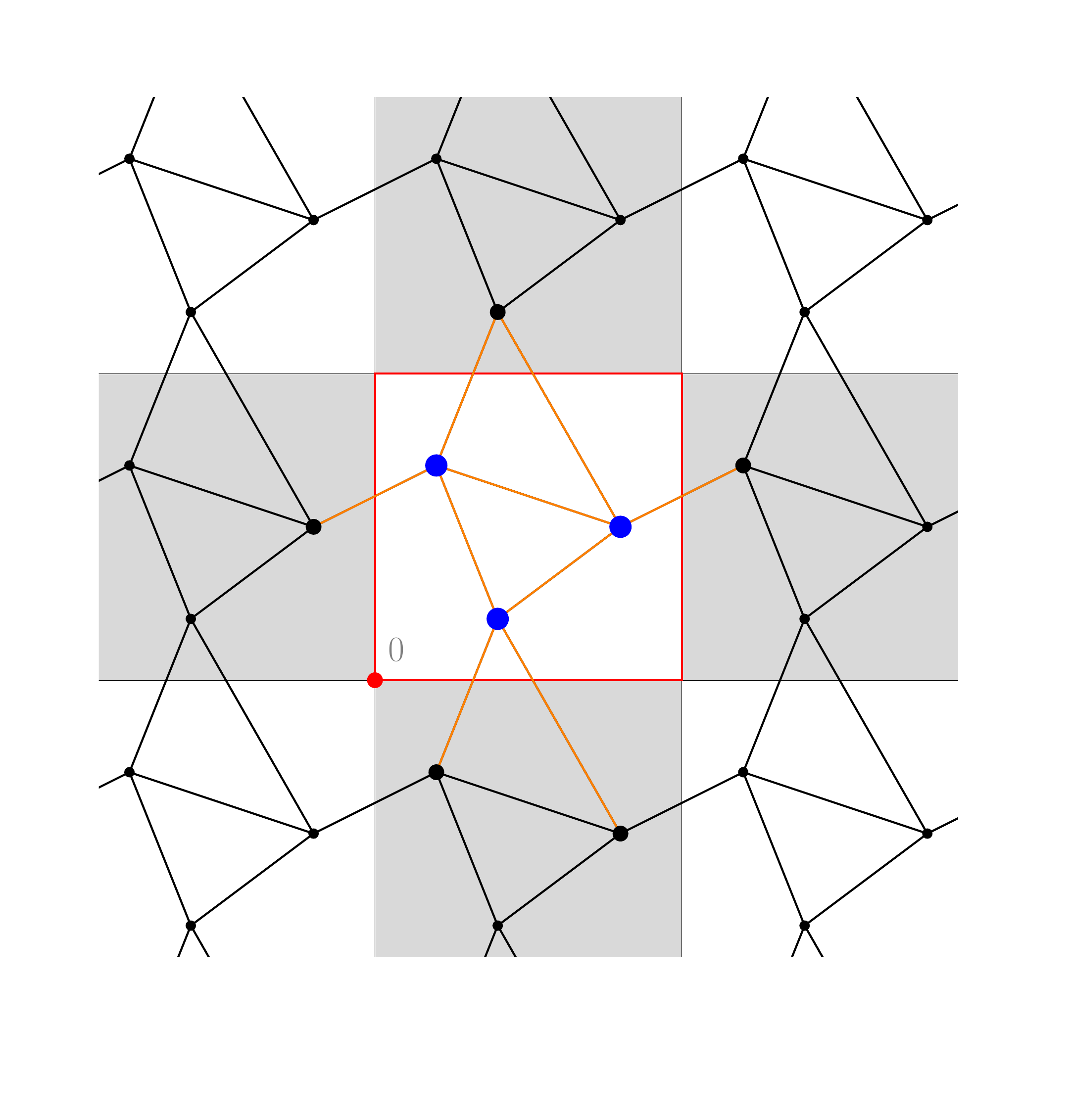}
	\vspace{-5mm}
	\caption{A fragment of a $\Z^d$-periodic graph $(\cX, \cE)$. The unit cube $Q:= [0,1)^d \subset \R^d$ is shown in red. In blue and in orange, respectively, $\cX^Q$ and $\cE^Q$.
	}
	\label{fig:per_graph_intro}
\end{figure} 

\subsubsection*{Convergence of the action}

One of our main results (Theorem \ref{thm:main}) asserts that, as $\eps \to 0$, the action functionals $\cA_\eps$ converge to a limiting functional $\bA = \bA_\hom$ of the form \eqref{eq:intro_A}, 
with an effective energy density 
	$f = f_\hom$
which depends non-trivially on the geometry of the graph $(\cX,\cE)$ and the discrete energy densities $F_{xy}$.
We only require a very mild linear growth condition on the energy densities $F_{xy}$:

\begin{center}
	\emph{As $\eps \to 0$, 
	the functionals 
	$\cA_\eps$ 
	$\Gamma$-converge to 
	$\bA_{\hom}$ \\
	in the weak (and vague) topology of 
	$\cM_+\big((0,1) \times \T^d\big)$. }
\end{center}
The precise formulation of this result involves an extension of $\bA_\hom$ to measures on 
	$(0,1) \times \Td$; 
see Section \ref{sec:continuous} below.

Let us now explain the form of the effective energy density $f_\hom$, which is given by a cell formula.
For given $\rho \geq 0$ and $j \in \R^d$, $f_\hom(\rho, j)$ is obtained by minimising 
	the discrete energy per unit cube 
	among all periodic mass distributions $m : \cX \to \R_+$
	representing $\rho$, 
	and all periodic divergence-free discrete vector fields $J : \cE \to \R$ representing $j$ in the following sense.
Set $\cX^Q := \cX\cap [0,1)^d$ and $\cE^Q:=  \big\{ (x,y) \in \cE \suchthat x \in \cX^Q \big\}$. 
Then $f_\hom:\R_+ \times \R^d \to \R_+$ is given by
\begin{align}	\label{eq:intro_fhom}
	f_\hom(\rho,j) :=
		\inf_{m,J}
			\bigg\{
				\sum_{(x,y) \in \cE^Q}
					F_{xy}
				\big(
					m(x), m(y), J(x,y) 
				\big) 
			\suchthat 
				(m,J) \in \Rep(\rho,j) 
			\bigg\},
\end{align}
where the set of representatives $\Rep(\rho,j)$ consists of all $\Z^d$-periodic functions $m: \cX \to \R_+$ and all $\Z^d$-periodic discrete divergence-free vector fields satisfying
\begin{align}	\label{eq:intro_effective}
	\sum_{x \in \cX^Q} m(x) = \rho 	
		\tand
	\Eff (J) :=  \frac12 
		\sum_{(x,y) \in \cE^Q} J(x,y) (y-x) = j. 
\end{align}

\subsubsection*{Boundary value problems}

Our second main result deals with the corresponding boundary value problems, which arise by minimising the action functional among all curves with given boundary conditions, as in the Benamou--Brenier formula \eqref{eq: bb}. 
We define
\begin{align*}
	\cMA_\eps(m^0,m^1) & :=
		\inf_{\bfm}
			\big\{
				\cA_\eps(\bfm)
					\ : \
						m_0 = m^0,
							\;
						m_1 = m^1
			\big\} &&  \text{for }  m^0, m^1 \in \cP(\cX_\eps), \\
			\bMA_\hom(\mu^0,\mu^1) &:=
			\inf_{\bfmu}
				\big\{
					\bA_\hom(\bfmu)
						\ : \
							\mu_0 = \mu^0,
								\;
								\mu_1 = \mu^1
				\big\}
				&& \text{for } \mu^0, \mu^1 \in \cP(\Td).
\end{align*}
We then obtain the following result (Theorem \ref{theorem: uniform compactness 2}):
\begin{center}
	\emph{As $\eps \to 0$, 
	the minimal actions 
	$\cMA_\eps$
	$\Gamma$-converge to 
	$\bMA_{\hom}$ \\
	in the weak topology of 
	$\cM_+(\Td) \times \cM_+(\Td)$.
	}
\end{center}
This result is proved under a superlinear growth condition on the discrete energy densities, which holds for discretisations of the Wasserstein distance $\bW_p$ for $p > 1$. 

A special case of interest is the case where $\cMA_\eps$ is a Riemannian transport distance associated to a gradient flow structure for Markov chains as in \cite{Maas:2011,Mielke:2011}.
In this situation, we show that the discrete transport distances converge to a $2$-Wasserstein distance on the torus. 
Interestingly, the underlying distance is induced by a Finsler metric, which is not necessarily Riemannian.

We also investigate transport distances with nonlinear mobility \cite{Dolbeault-Nazaret-Savare:2009}, \cite{lisiniMarigonda2010} and their finite-volume discretisations on the torus $\Td$. 
In the spirit of \cite{GlKoMa18}, we give a geometric characterisation of finite-volume meshes for which the discretised transport distances converge to the expected limit.

\subsubsection*{Compactness}

The results for boundary value problems are obtained by combining our first main result with a compactness result for sequence of measures with bounded action, which is of independent interest. 
We obtain two results of this type.

In the first compactness result (Theorem \ref{thm: compactness}) 
we assume at least linear growth of the discrete energies $F_{xy}$ at infinity. 
Under this condition we prove compactness in the space $\BV_{\KR}\big((0,1) ; \cM_+(\Td) \big)$, which consists of curves of bounded variation, with respect to the Kantorovich--Rubinstein ($\KR$) norm on the space of measures. 
The convergence holds for almost every $t\in (0,1)$. 

In the second compactness result (Theorem \ref{theorem: uniform compactness}), which is used in the analysis of the boundary value problems, we assume a stronger condition of at least superlinear growth on the energy densities $F_{xy}$. 
We then obtain compactness in the space 
	$W_\KR^{1,1}\big((0,1); \cM_+(\Td)\big)$, 
which consists of absolutely continuous curves with respect to the $\KR$-norm. 
The convergence is uniform for $t \in (0,1)$. 
We refer to the Appendix for precise definitions of these spaces.

\subsubsection*{Related works}
For a classical reference to the study of flows on networks, we refer to Ford and Fulkerson \cite{FordFulkerson}.

Many works are devoted to discretisations of continuous energy functionals in the framework of Sobolev and BV spaces, e.g., \cite{Piatniski-Remy:2001,Alicandro-Cicalese-Gloria:2011, Alicandro-Cicalese:2004, Bach-Braides-Cicalese:2020}.
Cell formulas appear in various discrete and continuous variational homogenisation problems; see, e.g., \cite{marcellini1978periodic, braides1998, Alicandro-Cicalese:2004, braides2000, gladbach2019limits}.

The large scale behaviour of optimal transport on random point clouds has been studied by Garcia--Trillos, who proved convergence to the Wasserstein distance \cite{Garcia-Trillos:2020}.

\subsection*{Organisation of the paper}
Sections \ref{sec:discrete} and \ref{sec:continuous} contain the necessary definitions as well as the assumptions we use throughout the article in the discrete and continuous settings. 
Section \ref{sec:homogenised} deals with the definition of the homogenised action functional. 
In Section \ref{sec: results} we present the rigorous statements of our main results, including the $\Gamma$-convergence of the discrete energies to the effective homogenised limit and the compactness theorems for curves of bounded discrete energies. The proof of our main results can be found in Section \ref{sec: compactness} (compactness and convergence of the boundary value problems) and Sections \ref{sec: lower} and \ref{sec: upper} ($\Gamma$-convergence of $\cA_\eps$). Finally, in Section \ref{sec:examples}, we discuss several examples and apply our results to some common finite-volume and finite-difference discretisations.

\subsection{Sketch of the proof of Theorem \ref{thm:main}}
In the last part of this section, we shortly sketch a non-rigorous proof of our main result on the convergence of $\cA_\eps$ to the homogenised limit described by $f_\hom$ (Theorem \ref{thm:main}). Crucial tools to show both the lower bound and the upper bound in Theorem \ref{thm:main} are regularisation procedures for solutions to the continuity equation, both at the discrete and at the continuous level. 

In this section, we use the informal notation $\lesssim$ and $\gtrsim$ to mean that the corresponding inequality holds up to a small error in $\eps>0$, e..g $A_\eps \lesssim B_\eps$ means that $A_\eps \leq B_\eps + o_\eps(1)$ where $o_\eps(1) \to 0$ as $\eps \to 0$.

For $x \in \cX_\eps \subset \Td$, we denote by $x_\sz$ the unique element of $\Z_\eps^d$ satisfying $x \in Q_\eps^{x_\sz} = [0,\eps)^d + \eps x_\sz$. Note that $\{Q_\eps^z \suchthat z \in \Z_\eps^d\}$ defines a partition of $\Td$.

In order to compare discrete and continuous measures, we make use of the embedding maps for $m \in \cP(\cX_\eps)$ and anti-symmetric $J:\cE_\eps \to \R$
	\begin{align*}
		\iota_\eps m 
		&	:= 
		\eps^{-d}	
		\sum_{x \in \cX_\eps}
		m(x)
		\Leb^d|_{Q_\eps^{x_\sz}} \in \cP(\Td)
			, \\
		\iota_\eps J
		&	:= 
		\eps^{-d+1}	
		\sum_{(x,y) \in \cE_\eps}
		\frac{J(x,y)}{2}
		\bigg(\int_0^1
		\Leb^d|_{ Q_\eps^{(1-s)x_\sz + s y_\sz}}
		\dd s\bigg)
		(y_\sz - x_\sz) \in \cM^d(\Td), 
	\end{align*}
as they preserve the continuity equation: if $(\bfm, \bfJ) \in \cCE_\eps$, then $(\iota_\eps \bfm, \iota_\eps \bfJ) \in \bCE$. 

We also use the notation $\cF_\eps(m,J) := \sum_{(x,y) \in \cE_\eps}
\eps^d 
F_{xy}
\big(
\frac{m(x)}{\eps^d},
\frac{m(y)}{\eps^d}, \frac{J(x,y)}{\eps^{d-1}}
\big) $.

\medskip
\noindent
\textit{Sketch of the liminf inequality}. \
Consider curves
	$(m_t^\eps)_{t \in (0,1)} \subseteq \cM_+(\cX_\eps)$	
and let
	$\bfm^\eps \in \cM_+((0,1) \times \cX_\eps)$ 
be the corresponding measure on space-time defined by 
	$\bfm^\eps(\ddd x, \ddd t) 
		= m_t^\eps(\ddd x) \dd t$.
Suppose that 
$\iota_\eps \bfm^\eps \to\bfmu$ vaguely 
 in $\cM_+((0,1) \times \Td)$ 
as $\eps \rightarrow 0$.  
The goal is to show the \textit{liminf inequality}
\begin{align}	\label{eq:sketch_lb}	
	\liminf_{\eps \to 0} 
	\cA_\eps(\bfm^\eps) \geq \bA_{\hom}(\bfmu).
\end{align}

We can assume that $\cA_\eps(\bfm^\eps) = \cA_\eps(\bfm^\eps, \bfJ^\eps) \leq C <\infty$ for every $\eps>0$, for some sequence of vector fields $\bfJ^\eps$ such that $(\bfm^\eps, \bfJ^\eps) \in \cCE_\eps$.
As we are going to see in \eqref{eq:density_embedded_flux}, the embedded solutions to the continuity equation $(\iota_\eps \bfm^\eps, \iota_\eps \bfJ^\eps) \in \bCE$ defines curves of measures with densities with respect to $\Leb^d$ on $\Td$ of the form, for every $u \in Q_\eps^{\bar z} \subset \Td$
\begin{align*}
	\rho_t(u)
	= \eps^{-d} 
	\sum_{\substack{x \in \cX_\eps\\
			x_\sz = \bar z}} 
	m_t^\eps(x)
	\tand
	j_t(u)
	=  
	\frac{1}{2\eps^{d-1}}
	\sum_{\substack{(x,y) \in \cE_\eps\\
			x_\sz = \bar z}} 
	J_{t,u}^\eps(x,y) 
	\big( 
	y_\sz - x_\sz 
	\big),
\end{align*} 
where $J_{t,u}^\eps \in \R^{\cE_\eps}$ is a convex combination of 
$
\big\{
J_t^\eps\big( \cdot -\eps z \big)
\, : \, 
z \in \Z_\eps^d, \, |z|_\infty \leq R_0 + 1	
\big\}$. 

As we will estimate the discrete energies at any time $t \in (0,1)$, for simplicity we drop the time dependence and write $\rho= \rho_t$, $j= j_t$, $m^\eps=m_t^\eps$, $J^\eps=J_t^\eps$, $J_u^\eps = J_{t,u}^\eps$. 

The main goal is to  construct, for every $u \in Q_\eps^{\bar z}$, a representative 
\begin{align}	\label{eq:sketch_competitor_lb}
	\bigg(
	\frac{\hat m_u}{\eps^d}, 
	\frac{\hat J_u}{\eps^{d-1}}
	\bigg)
	\in
	\Rep\big( 
	\rho(u), j(u)
	\big)
\end{align}
which is approximately equal to the values of $(m^\eps, J^\eps)$ close to $\cX \cap \{ x_\sz = \bar z \}$.
The lower bound \eqref{eq:sketch_lb} would then follow by integrating in time the static estimate
\begin{align}	\label{eq:sketch_lb_estimate}
	\begin{aligned}
		\cF_\eps(m, J)
		\gtrsim 
		\eps^d
		\sum_{\bar z \in \Z_\eps^d}
		\cF	\bigg(
		\frac{\hat m_{\eps \bar z}}{\eps^d}, 
		\frac{\hat J_{\eps \bar z}}{\eps^{d-1}}
		\bigg) 
		\gtrsim 
		\int_{\T^d}
		f_\hom\big( \rho(u), j(u) \big)	
		\dd u
		= \bF_\hom(\iota_\eps m, \iota_\eps J),
	\end{aligned}
\end{align}
and using the lower semicontinuity of $\bA_\hom$, where in the last inequality we used the very definition of the  homogenised density $f_\hom(\rho(u),j(u))$, which corresponds to the minimal microscopic cost with total mass $\rho(u)$ and flux $j(u)$.

In order to find the sought representatives in \eqref{eq:sketch_competitor_lb}, the natural choice is to define $\hat m_u \in \R_+^{\cX}$
and	
$\tilde J_u \in \R_a^{\cE}$ by taking the values of 
$m$ and $J_u$ 
in the $\eps$-cube at $\bar z$, 
and insert these values at every cube in $(\cX, \cE)$, so that the result is $\Z^d$-periodic. Precisely:
\begin{align*}
	\hat m_u(x) 
	:= 
	m(\eps \bar x)
		, \quad 
	\tilde J_u ( x,y )
	:= 
	J_u ( \eps \bar x, \eps (y - x_\sz + \bar z ) )
		, \quad
 		\text{for } ( x,y ) \in \cE ,	
\end{align*}
where $\bar x := x - x_\sz + \bar z$.
This would ensure that $\eps^{-d} \hat m_u \in \Rep\big( \rho(u) \big)$.
Unfortunately, this construction would produce a vector field
$\eps^{-(d-1)}\tilde J_u$ which (in general) does not belong to $\Rep\big(j(u)\big)$: 
indeed, while $\tilde J_u$ has the desired effective flux 
(i.e., $\Eff(\eps^{-(d-1)}\tilde J_u) = j(u)$, as given in  \eqref{eq:intro_effective}),
it would not be (in general) divergence-free.

In order to deal with this complication, we shall introduce a \emph{corrector field} 
$\bar J_u$, i.e., an anti-symmetric and $\Z^d$-periodic function $ \bar J_u : \cE \to \R$ satisfying
\begin{align}
	\label{eq:sketch_corrector-prop}
	\dive \bar J_u 
	= -\dive \tilde J_u, \quad
	\Eff(\bar J_u ) 
	= 0, \tand
	\big\| \bar J_u \big\|_{\ell^\infty(\EQ)} 
	\leq 
	\tfrac12 \big\| \dive \tilde J_u \big\|_{\ell^1(\XQ)} , 
\end{align}
whose existence we prove in Lemma \ref{lemma:bounds_divergence_eq 2}.

It is clear that if we set 
$\hat J_u
:= \tilde J_u 
+  \bar J_u$ 
by construction we have	
$\dive \hat J_u = 0$ 
and 
$\Eff\big(\eps^{-(d-1)}\hat J_u\big) = j(u)$, 
thus
\begin{align*}
	\frac{\hat J_u}{\eps^{d-1}}
	:= \frac{\tilde J_u 
		+  \bar J_u}
	{{\eps^{d-1}}}
	\in \Rep\big(j_u \big).
\end{align*}

To carry out this program and prove a lower bound of the form \eqref{eq:sketch_lb_estimate}, we need to quantify the error we perform passing from $(m^\eps, J^\eps)$ to $\big\{ (\hat m_u, \hat J_u) \suchthat u \in \Td \big\}$. It is evident by construction and from \eqref{eq:sketch_corrector-prop} that spatial and time regularity of $(m^\eps, J^\eps)$ are crucial to this purpose. For example, an  $\ell^\infty$-bound on the time derivative of the form $\| \partial_t m_t^\eps \|_\infty \leq C \eps^d$ (or, in other words, a Lipschitz bound in time for $\rho_t$) together with $(\bfm^\eps, \bfJ^\eps) \in \cCE_\eps$ would imply a control on $\dive J$ and thus a control of the error in \eqref{eq:sketch_corrector-prop} of the form $\| \eps^{1-d}\bar J_u \|_\infty \leq C \eps$.

This is why a key, first step in our proof is a regularisation procedure at the discrete level: for any given sequence of curves $\big\{ (\bfm^\eps, \bfJ^\eps) \in \cCE_\eps \suchthat \eps>0 \big\}$ of (uniformly) bounded action $\cA_\eps$, we can exihibit another sequence $\big\{ (\tilde \bfm{}^\eps, \tilde \bfJ{}^\eps) \in \cCE_\eps \suchthat \eps>0 \big\}$, quantitatively close as measures and in action $\cA_\eps$ to the first one, which enjoy good Lipschitz and $l^\infty$ properties and for which the above explained program can be carried out.

This result is the content of Proposition \ref{prop:regularisation_discrete} and it is based on a three-fold regularisation, that is in energy, in time, and in space (see Section \ref{sec:disc-reg}).

\medskip
\noindent
\textit{Sketch of the limsup inequality}. \
The goal is to show that, for every $(\bfmu,\bfnu) \in \bCE$, we can find $\bfm^\eps \in \cM_+((0,1) \times \cX_\eps)$ such that $\iota_\eps \bfm^\eps \to \bfmu$ weakly in $\cM_+((0,1) \times \Td)$ and 
\begin{align}	\label{eq:sketch_ub}
	\limsup_{\eps \to 0} 
	\cA_\eps(\bfm^\eps) \leq \bA_{\hom}(\bfmu, \bfnu).
\end{align}
In a similar fashion as in the the proof of the liminf inequality, the first step is a regularisation procedure, this time at the continuous level (Proposition \ref{prop:density}). Thanks to this approximation result, in the sketch we can without loss of generality assume that 
\begin{align}	\label{eq:sketch_assumptions_smooth_upperb}
	\bA_\hom(\bfmu, \bfnu) < \infty
	\tand
	\Big\{
	\big(\rho_t(x),j_t(x)\big) 
	\suchthat (t,x) \in [0,1] \times \Td 
	\Big\} 
	\Subset  
	\Dom(f_{\hom})^\circ,
\end{align}
where $(\rho_t,j_t)_t$ are the \textit{smooth} densities of $(\bfmu,\bfnu)\in \bCE$ with respect to $\Leb^{d+1}$ on $(0,1) \times \Td$.

Note that the convexity of $f_\hom$ ensures its Lipschitz-continuity on every compact set $K \Subset \Dom(f_{\hom})^\circ$, hence the assumption \eqref{eq:sketch_assumptions_smooth_upperb} allows us to assume such regularity for the rest of the proof.

The idea is to split the proof of the upper bound into several steps. In short, we first discretise the continuous measures $(\bfmu,\bfnu)$ and identify \textit{optimal discrete microstructures}, i.e. minimisers of the cell problem described by $f_\hom$, on each $\eps$-cube $Q_\eps^z$, $z \in \Z_\eps^d$. A key difficulty at this stage is that the optimal selection has the flaw of not preserving the continuity equation, hence an additional \textit{correction} is needed.  To this purpose, we first apply the discrete regularisation result Proposition \ref{prop:regularisation_discrete} to obtain regular discrete curves and then find suitable \textit{small} correctors that provide discrete competitors for $\cA_\eps$, that is solutions to $\cCE_\eps$ which are \textit{close} to the optimal selection.

Let us explain these steps in more detail.

\medskip 
\noindent
\underline{Step 1}: \ 
For every $z \in \Z_\eps^d$, $t \in (0,1)$,  and each cube $Q_\eps^z$ we consider the natural discretisation of $(\bfmu,\bfnu)$, that we denote by
$
\big(
\tP_\eps \mu_t(z),\tP_\eps \nu_t(z) 
\big)_{t ,z} 
\subset 
\R_+ \times \R^d
$, given by
\begin{align*}
	\tP_\eps \mu_t(z):=\mu_t(Q_\eps^z), 
	\quad 
	\tP_\eps \nu_t(z):=
	\left(
	\int_{\partial Q_\eps^z\cap \partial 
		Q_\eps^{z+ e_i}}j_t\cdot e_i\dd\cH^{d-1}
	\right)_{i=1}^d.
\end{align*}

An important feature of the operator $\tP_\eps$ is that it preserves the continuity equation from $\Td$ to $\Z_\eps^d$, in the sense that for $t \in (0,1)$ and $z \in \Z_\eps^d$
\begin{align*} 
	\partial_t \tP_\eps \mu_t(z)
	+ \sum_{i=1}^d 
	\big(
	\tP_\eps \nu_t(z) - 
	\tP_\eps \nu_t(z - e_i)
	\big)
	\cdot e_i
	= 0.
\end{align*}

\medskip
\noindent
\underline{Step 2}: \
We build the associated \textit{optimal discrete microstructure} for the cell problem for each cube $Q_\eps^z$, meaning we select 
$
(\bfm,\bfJ) = 
\big(
	m_t^z, J_t^z
\big)_{t\in(0,1), z \in \Z_\eps^d}
$ 
such that
\begin{align*}	
	\bigg(
	\frac{m_t^z}{\eps^d},\frac{J_t^z}{\eps^{d-1}}
	\bigg)
	\in \Rep_o \bigg( \frac{\tP_\eps \mu_t(z)}{\eps^d}
	,
	\frac{\tP_\eps \nu_t(z)}{\eps^{d-1}} \bigg) 
	,
\end{align*}
 where $\Rep_o$ denotes the set of optimal representatives in the definition of the cell-formula \eqref{eq:intro_fhom}.
Using the smoothness of $\bfmu$ and $\bfnu$, one can in particular show that
\begin{align}
	\label{eq:sketch_ub_est1}
	\sum_{z \in \Z_\eps^d} \eps^d \cF_1 \left(  \frac{ m_t^z}{\eps^d} , \frac{ J_t^z}{\eps^{d-1}} \right) 
	\lesssim \bF_{\hom}(\mu_t, \nu_t).
\end{align}

\medskip
\noindent
\underline{Step 3}: \
The next step is to glue together the microstructures 
$(\bfm, \bfJ)$ 
defined for every $z \in \Z_\eps^d$ via a \textit{gluing operator} $\cG_\eps$ (Definition \ref{def:gluing}) to produce a global one 
$(\hat \bfm{}^\eps, \hat \bfJ{}^\eps) \in \cM_+((0,1) \times \cX_\eps) \times \cM_a((0,1) \times \cE_\eps)$. 

Thanks to the fact the gluing operators are mass preserving and that $m_t^z \in \Rep(\tP_\eps \mu_t(z))$, it is not hard to see that $\iota_\eps \hat \bfm{}^\eps \to \bfmu$ weakly in $\cM_+((0,1) \times \Td)$ as $\eps \to 0$.

\medskip
\noindent
\underline{Step 4}: \
In contrast to $\tP_\eps$, the latter operation produces curves $(\hat \bfm{}^\eps, \hat \bfJ{}^\eps)$ which would (in general) \textit{not be a solution} to the discrete continuity equation $\cCE_\eps$. Therefore, we seek to find suitable \textit{corrector vector fields} in order to obtain a discrete solution, and thus a candidate for $\cA_\eps(\hat \bfm{}^\eps)$. 

To this purpose, the next step is to regularise $(\hat \bfm{}^\eps, \hat \bfJ{}^\eps)$ by applying Proposition \ref{prop:regularisation_discrete} and obtaining a regular curve which is quantitatively close as measures and in energy to the first one. Note that no discrete regularity is (in general) guaranteed to $(\hat \bfm{}^\eps, \hat \bfJ{}^\eps)$, despite the smoothness assumption on $(\bfmu,\bfnu)$, due to possible singularities of $F_{xy}$.

For the sake of the exposition, we shall discuss the last steps of the proof assuming that $(\hat \bfm{}^\eps, \hat \bfJ{}^\eps)$ already enjoy the Lipschitz and $\ell^\infty$--regularity properties ensured by Proposition \ref{prop:regularisation_discrete}.

\medskip
\noindent
\underline{Step 5}: \
For sufficiently regular $(\hat \bfm{}^\eps, \hat \bfJ{}^\eps)$, we seek a discrete competitor for $\cA_\eps(\hat \bfm{}^\eps)$ which is close to $(\hat \bfm{}^\eps, \hat \bfJ{}^\eps)$. 
As the latter does not necessary belong to $\cCE_\eps$, we find suitable correctors $\bfV^\eps$ such that the corrected curves 
$(\hat \bfm{}^\eps, \hat \bfJ{}^\eps + \bfV^\eps)$ 
belong to $\cCE_\eps$, with $\bfV^\eps$ \textit{quantitative small}, i.e. satisfying a bound of the form
\begin{align}
	\label{eq:sketch_ub_est2}
	\sup_{t \in (0,1)}
	\big\|
	\eps^{1-d} V_t^\eps 
	\big\|_{\ell_\infty(\cE_\eps )}
	\leq
	C \eps.	
\end{align}

The existence of the corrector $\bfV^\eps$, together with the quantitative bound, is quite involved and possibly the most difficult part of the proof. It is based on a localisation argument (Lemma \ref{lemma:localisation}) and the study of the divergence equation on periodic graphs (Lemma \ref{lemma:bounds_divergence_eq}), performed at the level of each cube $Q_\eps^z$, for every $z \in \Z_\eps^d$.

The regularity of $(\hat \bfm{}^\eps, \hat \bfJ{}^\eps)$ is crucial in order to obtain the estimate \eqref{eq:sketch_ub_est2}.

\medskip
\noindent 
\underline{Step 6}: \
The final step consists in estimating the action of the measures defined as $\bfm^\eps:= \hat \bfm{}^\eps \to \bfmu$ weakly as $\eps \to 0$, and the vector fields $\bfJ^\eps:= \hat \bfJ{}^\eps + \bfV^\eps$. 

Using the regularity assumption on $(\hat \bfm{}^\eps, \hat \bfJ{}^\eps)$, the smoothness \eqref{eq:sketch_assumptions_smooth_upperb} of $(\bfmu,\bfnu)$, and the convexity of $f_\hom$, together with the quantitative bound \eqref{eq:sketch_ub_est2} and \eqref{eq:sketch_ub_est1} for the corrector we obtain
\begin{align*}
	\cF_\eps(m_t^\eps, J_t^\eps) 
	\lesssim
	\cF_\eps(\hat m_t^\eps, \hat J_t^\eps)
	\lesssim 
	\sum_{z \in \Z_\eps^d} \eps^d 
 \cF_1 \left(  \frac{ m_t^z}{\eps^d} , \frac{ J_t^z}{\eps^{d-1}} \right) 
	\lesssim 
	\bF_{\hom}(\mu_t, \nu_t).
\end{align*}

Using this bound and that $(\bfm^\eps, \bfJ^\eps) \in \cCE_\eps$, we integrate in time and get
\begin{align*}
	\limsup_{\eps \to 0} \cA_\eps(\bfm^\eps) 
	\leq 
	\limsup_{\eps \to 0}  \cA_\eps(\bfm^\eps, \bfJ^\eps)
	\leq 
	\bA_\hom(\bfmu,\bfnu),
\end{align*}
which is the sought upper bound \eqref{eq:sketch_ub}.

\section{Discrete dynamical optimal transport on 
\texorpdfstring{$\Z^d$}{Z^d}-periodic graphs}
\label{sec:discrete}

This section contains the definition of the optimal transport problem in the discrete periodic setting.
In Section \ref{sec:data} we introduce the basic objects: a $\Z^d$-periodic graph $(\cX, \cE)$ and an admissible cost function $F$.
Given a triple $(\cX, \cE, F)$, we introduce a family of discrete transport actions on rescaled graphs $(\cX_\eps, \cE_\eps)$ in Section \ref{sec:rescaled}.

\subsection{Discrete $\Z^d$-periodic setting}
\label{sec:data}

Our setup consists of the following data:

\begin{assumption}
	\label{ass:XE}	
	$(\cX, \cE)$ is a locally finite and $\Z^d$-periodic connected graph of bounded degree.
\end{assumption}	
	
More precisely, we assume that 
\begin{align*}
	\cX  = \Z^d \times \V,
\end{align*}
where $\V$ is a finite set.
The coordinates of $x = (z, v) \in \cX$ will be denoted by
\begin{align*}
	x_\sz := z, \qquad 
	x_\sa := v.
\end{align*} 
The
\begin{figure}[h]
	\includegraphics[scale=0.30]{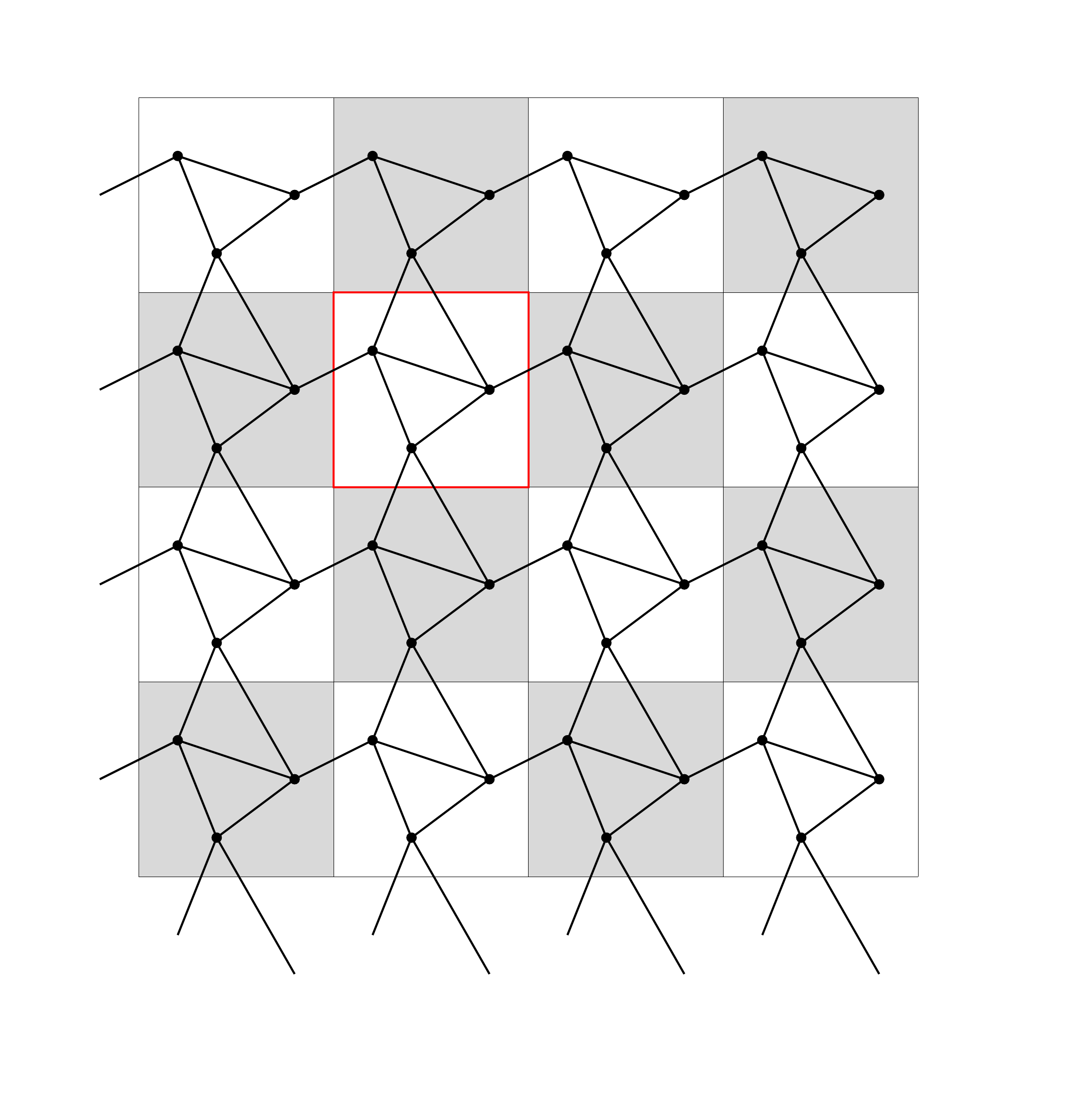}
	\vspace{-5mm}
	\caption{A fragment of a $\Z^d$-periodic graph $(\cX, \cE)$. The blue nodes represent $\cX^Q$ and the orange edges represent $\cE^Q$.   
	}
	\label{fig:per_graph}
\end{figure}
set of edges $\cE \subseteq \cX \times \cX$ is symmetric and $\Z^d$-periodic, in the sense that 
\begin{align*}
	(x, y) \in \cE 
		\quad \text{iff} \quad
	\big( S^z(x), S^z(y) \big) \in \cE
		\text{ for all } z \in \Z^d.
\end{align*}
Here, $S^{\bar z} : \cX \to \cX$ is the \emph{shift operator}
defined by
\begin{align*}
	S^{\bar z} (x) = (\bar z + z, v) \quad 
		\text{for } x = (z, v) \in \cX.
\end{align*}
We write $x \sim y$ whenever $(x,y) \in \cE$.

Let 
	$R_0 := 
		\max_{(x,y) \in \cE} 
			|  x_\sz - y_\sz|_{\ell_\infty^d}
	$ 
be the maximal edge length, 
measured with respect to the supremum norm 
	$|\cdot|_{\ell_\infty^d}$ on $\R^d$.
It will be convenient to use the notation
\begin{align*}
	\XQ :=
	\{
		x \in \cX
		\ : \ x_\sz = 0
	\} 
		\tand
	\EQ := 
		\big\{
			(x, y) \in \cE	
			\ : \
			x_\sz = 0
		\big\}.
\end{align*}

\begin{remark}[Abstract vs.~embedded  graphs]
	\label{rem:embedded_graph}
	Rather than working with abstract $\Z^d$-periodic graphs, it is possible to regard $\cX$ as a $\Z^d$-periodic subset of $\R^d$, 
	by choosing $\V$ to be a subset of $[0,1)^d$ and using the identification $(z,v) \equiv z + v$, see Figure \ref{fig:per_graph}. 
	Since the embedding plays no role in the formulation of the discrete problem, we work with the abstract setup.
\end{remark}

\begin{assumption}[Admissible cost function]
	\label{ass:F}
	The function 
		$F : \R_+^\cX \times \R_a^\cE \to \R \cup \{+\infty\}$ 
	is assumed to have the following properties:
	\begin{enumerate}[(a)]
	\item\label{item:F0}
	$F$ is convex and lower semicontinuous.
		\item\label{item:F1}
		 $F$ is \emph{local} in the sense that there exists $R_1 < \infty$ such that $F(m,J) = F(m',J')$ whenever $m, m' \in \R_+^\cX$ and $J, J' \in \R_a^\cE$ agree within a ball of radius $R_1$, i.e.,  
		 \begin{align*}
			m(x) &  = m'(x) 
				&& 
				\text{for all } x \in \cX 
				\text{ with } |x_\sz|_{\ell_\infty^d} \leq R_1, \quad and
		\\
			J(x,y) & = J'(x,y)	
				&& 
				\text{for all } (x,y) \in \cE
				\text{ with } |x_\sz|_{\ell_\infty^d},
								|y_\sz|_{\ell_\infty^d} \leq R_1.
		 \end{align*}
		\item\label{item:F2}
		 $F$ is of at least \emph{linear growth}, i.e., there exist 
			$c > 0$ and $C < \infty$ 
		such that 
		\begin{align}
		\label{eq: growth}
		F(m,J) 
			\geq 
			c \sum_{ (x, y) \in \EQ } 
					|J(x,y)| 
			- C \Bigg(
					1 + \sum_{\substack{
								x\in \cX\\ 
								|x|_{\ell_\infty^d} \leq R
							}} 
						m(x) 
				\Bigg)
		\end{align}
		for any $m \in \R_+^\cX$ and $J \in \R_a^\cE$. Here, $R := \max\{R_0, R_1\}$.
		\item\label{item:F3} 
		There exist a $\Z^d$-periodic function 
			$m^\circ\in \R_+^\cX$ 
		and a $\Z^d$-periodic and 
		divergence-free vector field
			$J^\circ\in \R_a^\cE$ 
		such that
		\begin{align}\label{eq: int dom}
			( m^\circ, J^\circ ) \in \Dom(F)^\circ.
		\end{align}
	\end{enumerate}
\end{assumption}	

\begin{remark}
	As $F$ is local, it depends on finitely many parameters. Therefore, $\Dom(F)^\circ$, the topological interior of its domain $\Dom(F)$ is defined unambiguously.
\end{remark}

\begin{remark}
	\label{rem:examples}
In many examples, the function $F$ takes one of the following forms, for suitable functions $F_x$ and $F_{xy}$:
\begin{align*}	
		F(m,J)& =
			\sum_{x\in \XQ}
		F_x\Big( m(x), \big( J(x,y) \big)_{y\sim x} \Big),
		& F(m,J) =
		 \frac12
		 \sum_{(x,y) \in \EQ}
		F_{xy}\Big(m(x),m(y),J(x,y)\Big).
	\end{align*}
We then say that $F$ is vertex-based (respectively, edge-based). 	
\end{remark}

\begin{remark}
	Of particular interest are edge-based functions of the form
	\begin{align}\label{eq: Wp}
		F(m,J) 
		= \frac12 \sum_{(x,y) \in \EQ} 
			\frac{|J(x,y)|^p}{\Lambda\big(q_{xy} m(x), q_{yx} m(y)\big)^{p-1}},
	\end{align}
	where 
	$1 \leq p < \infty$, the constants
	$q_{xy}, q_{yx} > 0$ are fixed parameters defined for $(x,y)\in \EQ$, 
	and $\Lambda$ is a suitable mean (i.e., $\Lambda : \R_+ \times \R_+ \to \R_+$ is a jointly concave and $1$-homogeneous function satisfying $\Lambda(1,1) = 1$).
	Functions of this type arise naturally in discretisations of Wasserstein gradient-flow structures \cite{Maas:2011,Mielke:2011,CHLZ11}.

	We claim that these cost function satisfy the growth condition \eqref{eq: growth}.
	Indeed, using Young's inequality 
		$|J| \leq \tfrac1p 	 \tfrac{|J|^p}{\Lambda^{p-1}}
			  + \tfrac{p-1}p \Lambda$
	we infer that			  
	\begin{align*}
		\sum_{(x,y) \in \EQ}
				|J(x,y)|
		&	\leq
		\frac{1}{p} \sum_{(x,y) \in \EQ}
			\frac{|J(x,y)|^p}{\Lambda\big(
					q_{xy} m(x), q_{yx} m(y)\big)^{p-1}}
		\\& \qquad + 
		\frac{p-1}{p} \sum_{(x,y) \in \EQ} 
		\Lambda\big(q_{xy} m(x), q_{yx} m(y)\big)
		\\&	\leq
		\frac2p F(m,J)
		+ 
			C\sum_{ x \in \cX, 
					|x|_{\ell_\infty^d} \leq R_0
			 } m(x),
	\end{align*}
	with constant $C>0$ depending on $\max_{x,y} (q_{xy} + q_{yx})$.
	This shows that \eqref{eq: growth} is satisfied.
\end{remark}

\subsection{Rescaled setting}
\label{sec:rescaled}

Let $(\cX, \cE)$ be a locally finite and $\Z^d$-periodic graph as above. 
Fix $\eps > 0$ such that $\frac1\eps \in \N$.
\emph{The assumption that $\frac1\eps \in \N$ remains in force throughout the paper.}

\smallskip
\noindent\emph{The rescaled graph.} \
Let $\T_\eps^d = (\eps \Z / \Z)^d$ be the discrete torus of mesh size $\eps$.
The corresponding equivalence classes 
are denoted by 
	$[\eps z]$ for $z \in \Z^d$. 
To improve readability, we occasionally omit the brackets.
Alternatively, we may write
	$\T_\eps^d = \eps \Z_\eps^d$
where
	$\Z_\eps^d = \big(\Z / \tfrac1\eps \Z \big)^d$. 

The rescaled graph 
	$(\cX_\eps, \cE_\eps)$ 
is constructed by rescaling the $\Z^d$-periodic graph 
	$(\cX, \cE)$ 
and wrapping it around the torus.
More formally, we consider the finite sets
\begin{align*}
	\cX_\eps := 
		\T_\eps^d \times \V
	\quad\text{and}\quad
	\cE_\eps := 
		\big\{
			\big( T_\eps^0 (x), T_\eps^0 (y) \big)
				\ : \
			(x, y) \in \cE
		\big\}
\end{align*}
where, for $\bar z \in \Z_\eps^d$,
\begin{align}
\label{eq:def-T}
	T_\eps^{\bar z} : \cX \to \cX_\eps, 
\qquad
	(z, v) 
		\mapsto 
	\big( [\eps (\bar z + z)], v \big).
\end{align}
Throughout the paper we always assume that $\eps R_0 < \frac12$, to avoid that  edges in $\cE$ ``bite themselves in the tail'' when wrapped around the torus. 
For $x = \big([\eps z], v\big) \in \cX_\eps$ we will write
\begin{align*}
	x_\sz := z \in \Z_\eps^d, \qquad 
	x_\sa := v\in \V.
\end{align*} 

\smallskip
\noindent\emph{The rescaled energies.} \
Let $F : \R_+^\cX \times \R_a^\cE \to \R \cup \{ + \infty\}$ be a cost function satisfying Assumption \ref{ass:F}. 
For $\eps > 0$ satisfying the conditions above, we shall define a corresponding energy functional $\cF_\eps$ in the rescaled periodic setting.

First we introduce some notation, which we  use to transfer functions defined on $\cX_\eps$ to $\cX$ (and from $\cE_\eps$ to $\cE$).
Let $\bar z \in \Z_\eps^d$.
Each function 
	$\psi : \cX_\eps \to \R$ 
induces a $\frac{1}{\eps}\Z^d$-periodic function 
\begin{align*}
	\tau_\eps^{\bar z} \psi 
		: \cX \to \R, 
		\qquad
	\big(\tau_\eps^{\bar z} \psi \big)(x)
		:=
	\psi\big( T_\eps^{\bar z}(x) \big)
	\quad \text{ for } 
	x \in \cX.
\end{align*}
see Figure \ref{fig:deftau}. Similarly, each function
	$J : \cE_\eps \to \R$ 
induces a $\frac{1}{\eps}\Z^d$-periodic function
\begin{align*}
	\tau_\eps^{\bar z} J : \cE \to \R, \qquad  
	\big( \tau_\eps^{\bar z} J \big)(x,y)
		:=
	J
		\big( 
			T_\eps^{\bar z}(x),
			T_\eps^{\bar z}(y)
		\big)
	\quad \text{ for } 
	(x,y) \in \cE.
\end{align*}

\begin{figure}[h]
	\begin{subfigure}{.5\textwidth}
		\centering
	\end{subfigure}%
	\begin{subfigure}{.5\textwidth}
		\centering
		\includegraphics[scale=1.2]{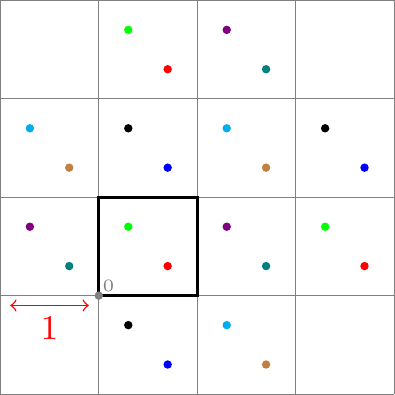}
	\end{subfigure}%
	\caption{On the left, the value of a function $\psi: \cX_\eps \to \R$ correspond to different colors over the nodes. On the right, the corresponding values of $\tau_\eps^z \psi: \cX \to \R$.}
	\label{fig:deftau}
\end{figure}

\begin{definition}[Discrete energy functional]
\label{eq:discrete-energy}
	The rescaled energy
	is defined by
	\begin{align*}
			\cF_\eps : 
				\Meps \times \Mdeps 
					&\to 
				\R \cup \{ + \infty\},	
		\qquad 
			(m, J) 
				\mapsto 
			\sum_{z \in \Z_\eps^d}
				\eps^d
		F\bigg(
			\frac{\tau_\eps^z m}
				{\eps^d}
			,
			\frac{\tau_\eps^z J}
				{\eps^{d-1}}
		\bigg).
	\end{align*}
\end{definition}

\begin{remark}
		\label{rem:well_posed_Feps}
	We note that
	$\cF_\eps(m,J)$ 
	is well-defined as an element in 
	$\R \cup \{ + \infty\}$.
	Indeed, the (at least) linear growth condition \eqref{eq: growth} yields
	\begin{align*}
		\cF_\eps( m, J ) 
		= 
		\sum_{z \in \Z_\eps^d}
		\eps^d
		F\bigg(
		\frac{\tau_\eps^z m}
		{\eps^d}
		,
		\frac{\tau_\eps^z J}
		{\eps^{d-1}}
		\bigg)
		& \geq 
		- C 
		\sum_{z \in \Z_\eps^d}
		\eps^d	
		\Bigg(
		1 + \sum_{\substack{
				x\in \cX\\ 
				|x|_{\ell_\infty^d} \leq R
		}} 
		\frac{\tau_\eps^z m(x)}{\eps^d}
		\Bigg)
		\\& \geq
		- C 
		\bigg( 
		1 + (2R + 1)^d 
		\sum_{x \in \cX_\eps} m(x)
		\bigg) 
		> - \infty.
	\end{align*} 
\end{remark}

For ${\bar z} \in \Z_\eps^d$
it will be useful to consider the \emph{shift operator} 
	$S^{\bar z}_\eps : \cX_\eps \to \cX_\eps$
and 
	$S^{\bar z}_\eps : \cE_\eps \to \cE_\eps$
defined by
\begin{align*}
	S^{\bar z}_\eps (x) & = 
		\big( [\eps (\bar z + z)], v\big) 
	& & 
	\text{for  } x = ([\eps z], v) \in \cX_\eps,
\\
	S^{\bar z}_\eps (x,y) & = 
		\big( 
			S^{\bar z}_\eps (x), S^{\bar z}_\eps (y)
		\big) 
	& &
	\text{for } (x,y) \in \cE_\eps.
\end{align*}
Moreover, for $\psi:\cX_\eps \to \R$ and $J: \cE_\eps \to \R$ we define
\begin{align}	
\begin{aligned} 
	\label{eq:def_sigma}
	&\sigma_\eps^{\bar z} \psi : \cX_\eps \to \R,		
	\qquad 
	&&\big(\sigma_\eps^{\bar z} \psi\big)(x) 
		:= \psi(S_\eps^{\bar z}(x)) 
	&& \text{for} \; x \in \cX_\eps, \\
	&\sigma_\eps^{\bar z}J : \cE_\eps \to \R, 
	\qquad
		&&\big(\sigma_\eps^{\bar z} J\big)(x,y) := J(S_\eps^{\bar z}(x,y)) 
	&& \text{for} \; (x,y) \in \cE_\eps.
\end{aligned}
\end{align}

\begin{definition}[Discrete continuity equation]	
	A pair $(\bfm, \bfJ)$ is said to be a solution to the discrete continuity equation if 
		$\bfm : \cI \to \Meps$ is continuous, 
		$\bfJ : \cI \to \Mdeps$ is Borel measurable, and 
	\begin{align}	\label{eq:def_discrete_CE}
		\partial_t m_t(x) + \sum_{y\sim x}J_t(x,y) = 0 
	\end{align}
	for all $x \in \cX_\eps$ in the sense of distributions.
	We use the notation 
	\begin{align*}
		(\bfm, \bfJ) \in \cCE_\eps^{T}.
	\end{align*}	
\end{definition}
	
\begin{remark}
	We may write \eqref{eq:def_discrete_CE} 
	as
	$
		\partial_t m_t + \dive J_t = 0
	$
	using the notation \eqref{eq:divergence}.
\end{remark}

\begin{lemma}[Mass preservation]
	\label{lem:mass-preservation}
	Let $(\bfm, \bfJ) \in\cCE_\eps^\cI$. 
	Then we have $m_s(\cX_\eps) = m_t(\cX_\eps)$ 
	for all $s, t \in \cI$.
\end{lemma}

\begin{proof}
	Without loss of generality, suppose that 
		$s, t \in \cI$ with $s < t$. 
	Approximating the characteristic function $\chi_{[s,t]}$ by smooth test functions, we obtain, for all $x \in \cX_\eps$,
	\begin{align*}
		m_t(x) - m_s(x)
			= 
		\int_s^t	
			\sum_{y \sim x}
				J_r(x,y)
		\dd r.
	\end{align*}
	Summing \eqref{eq:def_discrete_CE} over $x\in \cX_\eps$ and using the anti-symmetry of $\bfJ$, the result follows.
\end{proof}

We are now ready to define one of the main objects in this paper. 

\begin{definition}[Discrete action functional]		
\label{def:A_eps}
	For any continuous function $\bfm: \cI \to \R_+^{\cX_\eps}$ such that $t \mapsto \sum_{x \in \cX_\eps} m_t(x) \in L^1(\cI)$ and any Borel measurable function $\bfJ: \cI \to \R_a^{\cE_\eps}$, we define
	\begin{align*}
				\cA_\eps^\cI(\bfm, \bfJ) & :=	
					\int_\cI 
						\cF_\eps( m_t, J_t ) 
					\dd t
				\in \R \cup \{ + \infty\}.	
	\end{align*}
	Furthermore, we set
	\begin{align*}
		\cA_\eps^\cI(\bfm)
			& :=
			\inf \Big\{
				\cA_\eps^\cI(\bfm, \bfJ) 
				\ : \
				(\bfm, \bfJ) \in \cCE_\eps^\cI
				\Big\}.
	\end{align*}			 
\end{definition}

\begin{remark}
	We claim that 
		$\cA_\eps^\cI(\bfm, \bfJ)$ 
	is well-defined as an element in 
	 	$\R \cup \{ + \infty\}$.
	Indeed, the (at least) linear growth condition \eqref{eq: growth} yields as in Remark \ref{rem:well_posed_Feps}
	\begin{align*}
		\cF_\eps( m_t, J_t ) 
		\geq
		 - C 
			\bigg( 
				1 + (2R + 1)^d 
					\sum_{x \in \cX_\eps} m_t(x)
			\bigg).
	\end{align*} 
	for any $t \in \cI$. Since $t \mapsto \sum_{x \in \cX_\eps} m_t(x) \in L^1(\cI)$, the claim follows.

In particular, $\cA_\eps^\cI(\bfm, \bfJ)$ is well-defined whenever $(\bfm, \bfJ) \in \cCE_\eps^\cI$, 
since $t \mapsto \sum_{x \in \cX_\eps} m_t(x)$ is constant by Lemma \ref{lem:mass-preservation}.
\end{remark}

\begin{remark}
\label{rem:time-interval}
If the time interval is clear from the context, we often simply write $\cCE_\eps$ and $\cA_\eps$.
\end{remark}

The aim of this work is to study the asymptotic behaviour of the energies $\cA_\eps^\cI$ as $\eps \to 0$. 

\section{Dynamical optimal transport in the continuous setting}
\label{sec:continuous}

We shall now define a corresponding class of dynamical optimal transport problems on the continuous torus $\T^d$.
We start in Section \ref{sec:firsttry} by defining the natural continuous analogues of the discrete objects from Section \ref{sec:discrete}.
In Section \ref{sec:generalised} we define generalisations of these objects that have better compactness properties.

\subsection{Continuous continuity equation and action functional}
\label{sec:firsttry}

First we define solutions to the continuity equation on a bounded open time interval $\cI$.

\begin{definition}[Continuity equation]
	\label{def:CE0}	
		A pair $(\bfmu, \bfnu)$ is said to be a solution to the  continuity equation $\partial_t \bfmu + \nabla \cdot \bfnu = 0 $ if the following conditions holds:
		\begin{enumerate}[(i)]
			\item $\bfmu : \cI \to \cM_+(\T^d)$ is vaguely continuous;
			\item $\bfnu : \cI \to \cM^d(\T^d)$ is a Borel family satisfying $\int_\cI |\nu_t|(\T^d) \dd t < \infty$;
			\item The equation
			\begin{align}	\label{eq:def_cont_CE}
				\partial_t \mu_t(x) + \nabla \cdot \nu_t(x) = 0 
			\end{align}
		holds in the sense of distributions, i.e.,
		for all $\phi \in \cC_c\big(\cI \times \T^d \big)$,
		\begin{align*}
			\int_\cI \int_{\T^d}
				\partial_t \phi_t(x)
			\dd \mu_t(x) \dd t
			+ 
			\int_\cI \int_{\T^d}
				\nabla \phi_t(x) 
				\cdot
			\dd \nu_t(x) \dd t = 0.
		\end{align*}
\end{enumerate}
		We use the notation 
		\begin{align*}
			(\bfmu, \bfnu) \in \cCE^\cI.
		\end{align*}
	\end{definition}

We will consider the energy densities $f$ with the following properties.	

\begin{assumption}
	\label{ass:f}
Let 
	$f : \R_+ \times \R^d \to \R \cup \{+\infty\}$ 
be a lower semicontinuous and convex function, 
whose domain has nonempty interior. 
We assume that there exist constants
	$c > 0$ and $C < \infty$ 
such that the (at least) linear growth condition
\begin{align}\label{eq:growth_f}
	f(\rho, j)
		\geq c |j| - C (\rho + 1)
\end{align}	
holds for all
	$\rho \in \R_+$ and $j \in \R^d$.
\end{assumption}

The corresponding \emph{recession function} 
	$f^\infty : \R_+ \times \R^d \to \R \cup \{+\infty\}$
is defined by 
\begin{align*}
	f^\infty(\rho,j) 
		:= 
	\lim_{t \to + \infty}
	\frac{f(\rho_0 + t \rho, j_0 + t j)}{t},
\end{align*}
where $(\rho_0, j_0) \in \Dom(f)$ is arbitrary.
It is well known that the function $f^\infty$ is lower semicontinuous and convex, and it satisfies 
\begin{align}	\label{eq:growth_finfty}
	f^\infty(\rho,j) 
		\geq 
		 c|j| - C\rho.
\end{align}
We refer to \cite[Section 2.6]{amfupa} for a proof of these facts.

Let $\Leb^d$ denote the Lebesgue measure on $\T^d$.
For $\mu \in \cM_+(\T^d)$ 
and $\nu \in \cM^d(\T^d)$ 
we consider the Lebesgue decompositions given by 
\begin{align*}
	\mu = \rho \Leb^d + \mu^\perp, \qquad
	\nu = j \Leb^d + \nu^\perp
\end{align*}
for some $\rho \in L_+^1(\T^d)$ and $j \in L^1(\T^d; \R^d)$.
It is always possible to introduce a measure 
	$\sigma \in \cM_+(\T^d)$ such that 
\begin{align*}
	\mu^\perp = \rho^\perp \sigma, \qquad
	\nu^\perp = j^\perp \sigma,
\end{align*}
for some 
	$\rho^\perp \in L_+^1(\sigma)$
and 
	$j^\perp \in L^1(\sigma;\R^d)$. 
(Take, for instance, $\sigma = \mu^\perp + |\nu^\perp|$.)
Using this notation we define the continuous energy as follows.

\begin{definition}[Continuous energy functional]
\label{def:energyCont}
	Let $f : \R_+ \times \R^d \to \R \cup \{ + \infty\}$ satisfy Assumption \ref{ass:f}. 
	We define the continuous energy functional by
	\begin{align*}
		&\bF : 
			\cM_+( \Td ) 
				\times  
			\cM^d(\T^d) 
				\to 
			\R \cup \{ +\infty \}, \\
		&\bF(\mu, \nu) 
			:= 
			\int_{\Td} 
				f
					\big( 
						\rho(x), 
						j(x) 
					\big) 
			\dd x
			+
			\int_{\Td} 
				f^\infty 
					\big(
						\rho^\perp(x), j^\perp(x) 
					\big)
			\dd \sigma(x).
	\end{align*}
\end{definition}	

\begin{remark}
	By $1$-homogeneity of $f^\infty$, this definition does not depend on the choice of the measure $\sigma \in \cM_+(\T^d)$.	
\end{remark}

\begin{definition}
	[Action functional]		
	\label{def:bA1}
	For any curve 
		$\bfmu : \cI \to \cM_+(\T^d)$ 
	with $\int_\cI \mu_t(\T^d) \dd t < \infty$
	and any Borel measurable curve 
		$\bfnu : \cI \to \cM^d(\T^d)$ 
	we define
	\begin{align*}
				\bA^\cI(\bfmu, \bfnu) & :=	
					\int_\cI 
						\bF( \mu_t, \nu_t ) 
					\dd t.	
	\end{align*}
	Furthermore, we set
	\begin{align*}
		\bA^\cI(\bfmu)
			& :=
			\inf_{\bfnu} 
				\Big\{
					\bA^\cI(\bfmu, \bfnu) 
					\ : \
					(\bfmu, \bfnu) \in \bCE^\cI
				\Big\}.
	\end{align*}			 
\end{definition}

\begin{remark}
	As $f(\rho, j) \geq - C (1 + \rho)$
	by \eqref{eq:growth_f},
	the assumption $\int_\cI \mu_t(\T^d) \dd t < \infty$
	ensures that $\bA^\cI(\bfmu, \bfnu)$ is well-defined 
	in $\R \cup \{ + \infty \}$.
\end{remark}

\begin{remark}[Dependence on time intervals]
	Remark \ref{rem:time-interval} applies in the continuous setting as well. 
	If the time interval is clear from the context, we often simply write $\bCE$ and $\bA$.
	\end{remark}

Under additional assumptions on the function $f$, 
it is possible to prove compactness for families of solutions to the continuity equation with bounded action; see \cite[Corollary 4.10]{Dolbeault-Nazaret-Savare:2009}.
However, in our general setting, such a compactness result fails to hold, as the following example shows.  

\begin{example}[Lack of compactness]
	\label{ex:lack-of-compactness}
	To see this, let $y^\eps(t)$ be the position of a particle of mass $m$ that moves 
	from $0$ to 	
	$\bar y \in [0,\frac12]^d$ 
	in the time interval 
		$(a_\eps, b_\eps) 
			:= 
		\big(\tfrac{1-\eps}{2}, \tfrac{1+\eps}{2}\big)$
	with constant speed $\frac{|\bar y|}{\eps}$. 
	At all other times in the time interval $\cI = (0,1)$ the particle is at rest:
\begin{align*}
	y^\eps(t) = 
	\begin{cases}
		0, 	&	t \in [0, a_\eps],
		\\
		\big(t - \tfrac12(1 - \eps)\big) 
			\eps^{-1} \bar y,  	
			& 	t \in \big(a_\eps, b_\eps),\\
		\bar y  & 	t \in [b_\eps, 1].\\
	\end{cases}
\end{align*}
The associated solution $(\bfmu^\eps, \bfnu^\eps)$ to the continuity equation $\partial_t \bfmu^\eps + \nabla \cdot \bfnu^\eps = 0$ is given by 
\begin{align*}
	\mu_t^\eps(\ddd x) 
	:= m \delta_{y^\eps(t)}(\ddd x) , 
		\qquad
	\nu_t^\eps(\ddd x) 
	:= \frac{m|\bar y|}{\eps}  
		\chi_{(a_\eps, b_\eps)}(t)
		\delta_{y^\eps(t)}(\ddd x).		
\end{align*}

Let $f(\rho, j) = |j|$ be the total momentum, which satisfies  Assumption \ref{ass:f}.
We then have 
	$\bF(\mu_t^\eps, \nu_t^\eps) 
		= \frac{m|\bar y|}{\eps}  
	\one_{(a_\eps, b_\eps)}(t)
	$, 
hence $\bA^\cI(\bfmu^\eps, \bfnu^\eps) = m \bar y$, independently of $\eps$.

However, as $\eps \to 0$, the motion converges 
to the discontinuous curve given by 
	$\mu_t = \delta_0$ for $t \in [0,\tfrac12)$ 
and $\mu_t = \delta_{\bar y}$ for $t \in (\tfrac12, 1]$.
In particular, it does not satisfy the continuity equation in the sense above. 
\end{example}

\subsection{Generalised continuity equation and action functional}
\label{sec:generalised}
In view of this lack of compactness, we will extend the definition of the continuity equation and the action functional to more general objects.

\begin{definition}[Continuity equation]
	\label{def:CE1}	
	A pair of measures $(\bfmu, \bfnu) \in 
	\cM_+\big(\cI \times \T^d\big) 	
	\times 
	\cM^d\big(\cI \times \T^d\big)$
	is said to be a solution to the continuity equation 
	\begin{align}	\label{eq:def_CE}
	\partial_t \bfmu + \nabla \cdot \bfnu = 0 
	\end{align}
	if, for all $\phi\in \cC_c^1\big(\cI\times \Td\big)$, we have
	\begin{align*}
		\int_{\cI \times \Td} 
			\partial_t \phi
		\dd \bfmu 
		+ \int_{\cI \times \Td} 
			\nabla \phi 
		\cdot \dd \bfnu 
		= 0.
	\end{align*}
	As above, we use the notation $(\bfmu, \bfnu) \in \bCE^\cI$.
\end{definition}

Clearly, this definition is consistent with Definition \ref{def:bA1}.

Let us now extend the action functional $\bA^\cI$ as well.
For this purpose, 
let $\Leb^{d+1}$ denote the Lebesgue measure on $\cI \times \T^d$.
For $\bfmu \in \cM_+\big( \cI \times \T^d\big)$ 
and $\bfnu \in \cM^d\big( \cI \times \T^d\big)$ 
we consider the Lebesgue decompositions given by 
\begin{align*}
\bfmu = \rho \Leb^{d+1} + \bfmu^\perp, \qquad
\bfnu = j \Leb^{d+1} + \bfnu^\perp
\end{align*}
for some 
	$\rho \in L_+^1\big(  \cI \times \T^d \big)$ 
and 
	$j \in L^1\big( \cI \times \T^d; \R^d\big)$.
As above, it is always possible to introduce a measure $\bfsigma \in \cM_+( \cI \times \T^d)$ such that 
\begin{align}	\label{eq:decomp_sigma}
\bfmu^\perp = \rho^\perp \bfsigma, \qquad
\bfnu^\perp = j^\perp \bfsigma,
\end{align}
for some $\rho^\perp \in L_+^1(\bfsigma)$ and $j^\perp \in L^1(\bfsigma;\R^d)$.

\begin{definition}[Action functional]\label{def: A2}
	We define the action by
	\begin{align*}
	&\bA^\cI : 
	\cM_+\big( \cI\times \Td \big) 
	\times  
	\cM^d\big( \cI\times \T^d\big) 
	\to 
	\R \cup \{ +\infty \}, \\
	&\bA^\cI(\bfmu, \bfnu) 
	:= 
	\int_{\cI \times \Td} 
	f
	\big( 
	\rho_t(x), 
	j_t(x) 
	\big) 
	\dd x \dd t  
	+
	\int_{ \cI \times \Td} 
	f^\infty 
	\big(
	\rho^\perp_t(x), j^\perp_t(x) 
	\big)
	\dd \bfsigma(t,x).
	\end{align*}
	Furthermore, we set
	\begin{align*}
	\bA^\cI(\bfmu)
		:=
		\inf\{
				\bA^\cI(\bfmu,\bfnu) 
			: 
			(\bfmu,\bfnu) \in \bCE^\cI
		\}.
	\end{align*}
\end{definition}	

\begin{remark}
	This definition does not depend on the choice of $\bfsigma$, in view of the $1$-homogeneity of $f^\infty$. As $f(\rho,j) \geq - C(1+\rho)$ and $f_\infty(\rho,j) \geq - C \rho$ from \eqref{eq:growth_f} and \eqref{eq:growth_finfty}, the fact that $\bfmu(\cI \times \Td) <\infty$ ensures that  $\bA^\cI(\bfmu, \bfnu)$ is well-defined 
	in $\R \cup \{ + \infty \}$.
\end{remark}

\begin{example}[Lack of compactness]
	\label{ex:lack-of-compactness-cont}
	Continuing Example \ref{ex:lack-of-compactness}, 
	we can now describe the limiting jump process as a solution to the generalised continuity equation.
	Consider the measures 
		$\bfmu^\eps \in \cM_+(\cI \times \T^d)$ 
	and 
		$\bfnu^\eps \in \cM^d(\cI \times \T^d)$ 
	defined by
	\begin{align*}
		\bfmu^\eps(\ddd  x, \ddd  t)
			= \mu_t^\eps(\ddd x) \dd t, \qquad
		\bfnu^\eps(\ddd x, \ddd t)
			= \nu_t^\eps(\ddd x) \dd t.
	\end{align*}
	Then we have
		$\bfmu^\eps \to \bfmu$ 
	and 
		$\bfnu^\eps \to \bfnu$
	weakly, respectively, in $\cM_+(\cI\times \Td)$ and $\cM^d(\cI \times \Td)$,
	where $\bfmu$ represents the discontinuous curve
	\begin{align*}
		\bfmu(\ddd x, \ddd t) = \dd \mu_t(x) \dd t, \quad \text{where }
		\mu_t = 
		\begin{cases}
			\delta_0, &  t \in [0,\tfrac12), \\
			\delta_{\bar y}, &  t \in (\tfrac12, 1].
		\end{cases}
	\end{align*} 
	The measure $\bfnu$ does \emph{not} admit a disintegration with respect to the Lebesgue measure on $\cI$; in other words, it is not associated to a curve of measures on $\T^d$. We have 
	\begin{align*}
		\bfnu_t(\ddd x, \ddd t) 
			= m|\bar y|
			\Haus^1|_{[0,\bar y]}(\ddd x)
				\delta_{1/2}(\ddd t).
	\end{align*}
	Here $\Haus^1|_{[0,\bar y]}$ denotes the $1$-dimensional Hausdorff measure on the (shortest) line segment connecting $0$ and $\bar y$.

	Note that $(\bfmu, \bfnu)$ solves the continuity equation, as $\bCE^\cI$ is stable under joint weak-convergence.
	Furthermore, we have
		$\bA^\cI(\bfmu, \bfnu) = m \bar y$.
\end{example}

The next result shows that any solution to the continuity equation $(\bfmu, \bfnu) \in \bCE^\cI$ induces a (not necessarily continuous) curve of measures $(\mu_t)_t \in \cI$.
The measure $\bfnu$ is not always associated to a curve of measures on $\cI$; see Example \ref{ex:lack-of-compactness-cont}.
We refer to Appendix \ref{sec:norms} for the definition of $\BVKR$.

\begin{lemma}[Disintegration of solutions to $\bCE^\cI$]
	\label{lemma:disintegration}
	Let $(\bfmu, \bfnu) \in \bCE^{\cI}$. 
	Then 
		$\dd \bfmu(t,x) = \dd \mu_t(x) \dd t$ 
	for some measurable curve 
		$t \mapsto \mu_t \in \cM_+(\T^d)$
	with finite constant mass.
	If $\bA^{\cI}(\bfmu) < \infty$, 
	then this curve 
	belongs to $\BVKR$ and 
	\begin{align}
		\| \bfmu \|_{\BVKR}
			\leq 
		|\bfnu| \big( \cI \times \Td \big).
	\end{align}
\end{lemma}

\begin{proof}
	Let $\lambda \in \cM_+(\cI)$ be the time-marginal of $\bfmu$, i.e., 
		$\lambda := (e_1)_{\#} \bfmu$ 
	where $e_1: \cI \times \Td \to \cI$, $e_1(t,x) = t$. 
	We claim that $\lambda$ is a constant multiple of the Lebesgue measure on $\cI$.
	By the disintegration theorem (see, e.g., \cite[Theorem 5.3.1]{Ambrosio-Gigli-Savare:2008}), this implies the first part of the result.
	
	To prove the claim, note that 
	the continuity equation $\bCE^{\cI}$ yields
	\begin{align}
		\label{eq:disi}
		\int_{\cI} 
			\partial_t \phi(t) \dd \lambda(t) 	
			=
		\int_{\cI \times \T^d} 
			\partial_t \phi(t) \dd \bfmu(t,x) 
			= 0
	\end{align} 
	for all $\phi \in C_c^\infty(\cI)$.
	
	Write $\cI = (a,b)$, 
	let $\psi \in C_c^\infty(\cI)$ be arbitrary, 
	and set 
		$\bar\psi := \frac{1}{|\cI|}\int_\cI \psi(t) \dd t$.
	We define
		$\phi(t) = \int_a^t \psi(s) \dd s 
					- (t-a) \bar \psi$. 
	Then $\phi \in C_c^\infty(\cI)$ and $\partial_t \phi = \psi - \bar \psi$.	
	Applying \eqref{eq:disi} we obtain 
		$\int_\cI (\psi - \bar \psi)\dd \lambda = 0$,
	which implies the claim, and hence the first part of the result.
	
	\smallskip

	To prove the second part, 
	suppose that 
		$\bfmu\in \cM_+(\cI\times \Td)$ 
	has finite action, 
	and let 
		$\bfnu\in \cM^d\big( \cI \times \Td \big)$ 
	be a solution to the continuity equation 
	\eqref{eq:def_CE}. 
	Applying \eqref{eq:def_CE} to a test function 
	$\phi \in 
	\cC_c^1\big( \cI; \cC^1(\Td) \big) 	
	\subseteq 
	\cC_c^1\big( \cI \times \Td \big)$ 
	such that $\max_{t\in \cI}\|\phi_t\|_{\cC^1(\Td)} \leq 1$, we obtain
	\begin{align}	\label{eq:BV_C1star}
		\int_{\cI \times \T^d} 
			\partial_t \phi_t 
		\dd \mu_t \dd t 
		= 	- \int_{\cI \times \T^d} 
					\nabla \phi \cdot 
			\dd \bfnu
		\leq 
			|\bfnu| \big( \cI \times \Td \big) 
		< \infty,
	\end{align}
	which implies the desired bound in view of \eqref{eq:BV-def}.
\end{proof}

The next lemma deals with regularity properties for curves of measures with finite action and fine properties for the functionals $\bA$ defined in Definition \ref{def: A2} with $f=f_\hom$.

\begin{lemma}
	[Properties of $\bA^\cI$]	
	\label{lemma: action lsc}
	Let $\cI \subset \R$ be a bounded open interval. 
	The following statements hold:
	\begin{enumerate}[(i)]
		\item  
		The functionals 
		$(\bfmu, \bfnu) \mapsto \bA^\cI(\bfmu, \bfnu)$ 
		and 
		$\bfmu \mapsto \bA^\cI(\bfmu)$
		are convex.
		\item 
		Let $\bfmu \in \cM_+(\cI \times \Td)$.
		Let $\{\cI_n\}_n$ be a sequence of bounded open intervals such that 
			$\cI_n \subseteq \cI$ and 
			$|\cI \setminus \cI_n| \to 0$ as $n \to \infty$. 
		Let $\bfmu^n \in \cM_+(\cI_n \times \Td)$
		be such that 
		\footnote{We regard measures on $\cI_n \times \Td$ as measures on the bigger set $\cI \times \Td$ by the canonical inclusion.} 
		\begin{align*}
			\bfmu^n \to \bfmu
				\ \text{vaguely in} \ \cM_+(\cI \times \Td)
			 \ \text{ and } \ 
			\bfmu^n(\cI_n \times \Td) \to \bfmu(\cI \times \Td).
		\end{align*}
		as $n \to \infty$.
		Then:
		\begin{align}	\label{eq:lsc_varying_intervals}
			\liminf_{n \to \infty} 
			\bA^{\cI_n}(\bfmu^n)
		\geq 
			\bA^{\cI}(\bfmu).
		\end{align}
		If, additionally, 
		$\bfnu \in \cM^d(\cI \times \Td)$ 
		and 
		$\bfnu^n \in \cM^d(\cI_n \times \Td)$
		satisfy 
		$\bfnu^n \to \bfnu$ vaguely in $\cM^d(\cI \times \Td)$,
		then we have
		\begin{align}\label{eq:lsc_varying_intervals_2}
			\liminf_{n \to \infty} 
				\bA^{\cI_n}(\bfmu^n, \bfnu^n) 
			\geq 
				\bA^{\cI}(\bfmu, \bfnu).
		\end{align}
		In particular,
		the functionals 
		$(\bfmu, \bfnu) \mapsto \bA^\cI(\bfmu, \bfnu)$ 
		and 
		$\bfmu \mapsto \bA^\cI(\bfmu)$
		are lower semicontinuous with respect to (joint) vague convergence.
	\end{enumerate}
\end{lemma}

\begin{proof}		
	\emph{(i)}: \
	Convexity of $\bA^\cI$ follows from convexity of 
	$f$, $f^\infty$,
	and the linearity of the constraint \eqref{eq:def_CE}.
	
	\smallskip
	
	\emph{(ii)}: \
	First we show \eqref{eq:lsc_varying_intervals_2}.
	Consider the convex energy density 
	$g(\rho,j) := f(\rho,j) + C(\rho + 1)$,
	which is nonnegative by \eqref{eq: growth}. 
	Let $\bA_g$ be the corresponding action functional defined using $g$ instead of $f$.
	Using the nonnegativity of $g$, 
	the fact that $|\cI \setminus \cI_n| \to 0$,
	and the lower semicontinuity result from \cite[Theorem 2.34]{amfupa},
	we obtain
	\begin{align*}
		\liminf_{n\to \infty} 
		\bA_g^{\cI_n}(\bfmu^n,\bfnu^n )
		\geq 
		\liminf_{n\to \infty} 
		\bA_g^{\tilde \cI}(\bfmu^n,\bfnu^n )
		\geq
		\bA_g^{\tilde \cI}(\bfmu, \bfnu).
	\end{align*}
	for every open interval $\tilde \cI \Subset \cI$. 
	Taking the supremum over $\tilde \cI$, we obtain
	\begin{align}	\label{eq:lsc_Ag}
		\liminf_{n\to \infty} 
		\bA_g^{\cI_n}(\bfmu^n,\bfnu^n )
		\geq
		\bA_g^{\cI}(\bfmu, \bfnu).
	\end{align}
	Since we have
	$\bfmu^n\big( \cI_n \times \Td \big) \to
	\bfmu\big( \cI \times \Td \big)$ by assumption, 
	the desired result
	\eqref{eq:lsc_varying_intervals_2} 
	follows from  \eqref{eq:lsc_Ag} 
	and the identity
	\begin{align*}
		\begin{aligned}
			\bA_g^{\cI_n}(\bfmu^n, \bfnu^n)
			= 
			\bA^{\cI_n}(\bfmu^n, \bfnu^n) 
			+ C \Big( \bfmu^n(\cI_n\times \Td) + 1 \Big).
		\end{aligned}
	\end{align*}
	
	\smallskip 
	
	Let us now show \eqref{eq:lsc_varying_intervals}.
	Let
	$\{\bfmu^n\}_n 
		\subseteq
	\cM_+\big( \cI_n \times \Td \big)
	$ 
	be such that 
		$\sup_n \bA^{\cI_n}(\bfmu^n) < \infty$ 
	and 
		$\bfmu^n \to \bfmu$ vaguely in $\cM_+(\cI \times \Td)$. 
	Let
		$\bfnu^n \in \cM^d(\cI_n \times \Td)$ 
	be such that 
		$(\bfmu^n, \bfnu^n) \in \bCE^{\cI_n}$ and 
	\begin{align*}
		 \bA^{\cI_n}(\bfmu^n, \bfnu^n) 
		 \leq \bA^{\cI_n}(\bfmu^n) + \frac1n.
	\end{align*}
	From Lemma \ref{lemma:disintegration}, we infer that 
		$\dd \bfmu^n(t,x) = \dd \mu_t^n(x) \dd t$ 
	where $(\mu_t^n)_{t \in \cI_n}$ 
	is a curve of constant total mass 
	$c_n := \mu_t^n(\T^d)$.
	Moreover, $M := \sup_n c_n <+\infty$, since $\bfmu^n \to \bfmu$ vaguely. 
	The growth condition \eqref{eq:growth_f} implies that
	\begin{align*}
		\sup_n 
		|\bfnu^n|\big( \cI^n \times \Td\big ) 
		\leq 
		\frac1c \sup_n \bA_\hom^{\cI_n}(\bfmu^n)
		+ \frac{C|\cI|}{c} 
			\big( M + 1 \big)
		< \infty.
	\end{align*}
	Hence, by the Banach--Alaoglu theorem, 
	there exists a subsequence of $\{\bfnu^n\}_n$ 
	(still indexed by $n$) 
	such that 
		$\bfnu^n \to \bfnu$ 
			vaguely in $\cM^d(\cI \times \Td)$
	and 
		$(\bfmu, \bfnu) \in \bCE^{\cI}$. 
	Another application of Lemma \ref{lemma:disintegration} ensures that 
		$\dd \bfmu(t,x) = \dd \mu_t(x) \dd t$ 
	where $(\mu_t)_{t \in \cI}$ is of constant mass 
		$c := \mu_t(\T^d) 
			= \lim_{n \to \infty} c_n
		$.
	
	We can thus apply the first part of \emph{(ii)} to obtain
	\begin{align*}
		\bA^{\cI}(\bfmu)
			\leq 
		\bA^{\cI}(\bfmu, \bfnu) 
			\leq 
		\liminf_{n \to \infty} 
			\bA^{\cI_n}( \bfmu^n, \bfnu^n ) 
		= 
		\liminf_{n \to \infty} 
			\bA^{\cI_n}( \bfmu^n ),
	\end{align*}
	which ends the proof.
\end{proof}

\section{The homogenised transport problem}
\label{sec:homogenised}

Throughout this section we assume that 
$(\cX,\cE)$ safisfies Assumption \ref{ass:XE} 
and 
$F$ safisfies Assumption \ref{ass:F}.

\subsection{Discrete representation of continuous measures and vector fields}

To define $f_\hom$, the following definition turns out to be natural.  

\begin{definition}[Representation]\label{def: representation}
	\mbox{}
		\begin{enumerate}[(i)]
			\item We say that 
				$m\in \R_+^{\cX}$ represents 
				$\rho \in \R_+$ 
			if $m$ is $\Z^d$-periodic and 
			\begin{align*} 
				\sum_{x \in \XQ} 
					m(x) 
				= \rho.
			\end{align*}	
			\item We say that $J\in \R_a^{\cE}$ represents a vector $j \in \R^d$ if 
			\begin{enumerate}
				\item $J$ is $\Z^d$-periodic;
				\item $J$ is divergence-free (i.e., 
					$\dive J(x) = 0$ 
					  for all $x \in \cX$);
				\item The effective flux of $J$ equals $j$; i.e., $\Eff(J) = j$, where
				\begin{align}
				\label{eq:Eff_J}
					\Eff(J) 
						:=
						\frac12
						\sum_{(x,y) \in \EQ} 
							J(x,y) 
							\big( 
								y_\sz - x_\sz 
							\big).
					\end{align}
			\end{enumerate}
		\end{enumerate}
		We use the (slightly abusive) notation 
			$m \in \Rep(\rho)$ and $J \in \Rep(j)$.
		We will also write 
			$\Rep(\rho,j) = \Rep(\rho) \times \Rep(j)$.
	\end{definition}
	
	\begin{remark}
	Let us remark that $x_\sz = 0$ in the formula for $\Eff(J)$, since $x_\sz \in \XQ$.
	\end{remark}
	
	\begin{remark}
		The definition of the effective flux $\Eff(J)$ is natural in view of Lemmas \ref{lem:conteq-embed} 
		and
		\ref{lemma:density_embedded_flux} below. 
		These results show that a solution to the continuous continuity equation can be constructed
		starting from a solutions to the discrete continuity equation, with a vector field of the form \eqref{eq:Eff_J}.
	\end{remark}

	Clearly, 
		$\Rep(\rho) \neq \emptyset$ 
	for every $\rho \in \R_+$. 
	It is also true, though less obvious, that
		$\Rep(j) \neq \emptyset$
	for every $j \in \R^d$.
	We will show this in Lemma \ref{lem:J_Pq} 
	using the $\Z^d$-periodicity 
	and the connectivity of $(\cX, \cE)$.
	
	To prove the result, we will first introduce a
	natural vector field associated to each
	simple directed path $P$ on $(\cX,\cE)$,
	For an edge 
		$e =(x,y) \in \cE$, 
	the corresponding reversed edge will be denoted by 
		$\overline e = (y,x) \in \cE$. 
	
	\begin{definition}[Unit flux through a path]	\label{def:vectoralongpath}
		Let $P := \{x_i\}_{i=0}^m$ be a simple path in $(\cX, \cE)$, 
		thus $e_i = (x_{i-1}, x_i) \in \cE$ 
		for $i = 1, \ldots, m$,
		and $x_i \neq x_k$ for $i \neq k$. 
		The \emph{unit flux through $P$}
		is the discrete field 
			$J_{P} \in \R^{\cE}_a$ 
		given by
		\begin{align}	\label{eq:divergence_J_Pq}
			J_P(e) = 
		\begin{cases}
		1  & \text{if } e = e_i \text{ for some } i, \\
		-1 & \text{if } e = \overline e_i  \text{ for some } i,\\
		0  & \text{otherwise}
		\end{cases}
		\end{align}
		The \emph{periodic unit flux through $P$} is the vector field 
			$\tilde J_P \in\R^{\cE}_a$
		defined by
		\begin{align}
		\label{eq:periodic-unit-flux}
			\tilde J_P(e) 
				= \sum_{z \in \Z^d} 
					J_P(T_z e)
				\quad 
			\text{for } e \in \cE.
		\end{align}	
	\end{definition}
	
	In the next lemma we collect some key properties of these vector fields.
	Recall the definition of the discrete divergence in \eqref{eq:divergence}.

	\begin{lemma}[Properties of $J_P$]	\label{lem:J_Pq}
		Let $P := \{x_i\}_{i=0}^m$ be a simple path in $(\cX, \cE)$.
		\begin{enumerate}[(i)]
			\item 
				The discrete divergence of the associated unit flux 
					$J_P : \cE \to \R$  
				is given by
				\begin{align}
					\dive J_P
					= 
					\one_{\{x_0\}} - \one_{\{x_m\}}.
				\end{align} 
			\item 
				The discrete divergence of the periodic unit flux 
					$\tilde J_P : \cE \to \R$  
				is given by
				\begin{align}
					\dive \tilde J_P(x)
					= 
					\one_{\{(x_0)_\sa\}}(x_\sa)
						-
					\one_{\{(x_m)_\sa\}}(x_\sa),
					\quad
					x \in \cX.
				\end{align} 
				In particular, $\dive \tilde J_P \equiv 0$ 
				iff $(x_0)_\sa = (x_m)_\sa$.
			\item 
			The periodic unit flux 
				$\tilde J_P : \cE \to \R$ 
			satisfies $\Eff(\tilde J_P) = (x_m)_\sz - (x_0)_\sz$. 
			\item For every $j \in \R^d$ we have $\Rep(j) \neq \emptyset$.
		\end{enumerate}
	\end{lemma}
	\begin{figure}[h]
		\begin{subfigure}{.5\textwidth}
			\centering
			\includegraphics[scale=0.45]{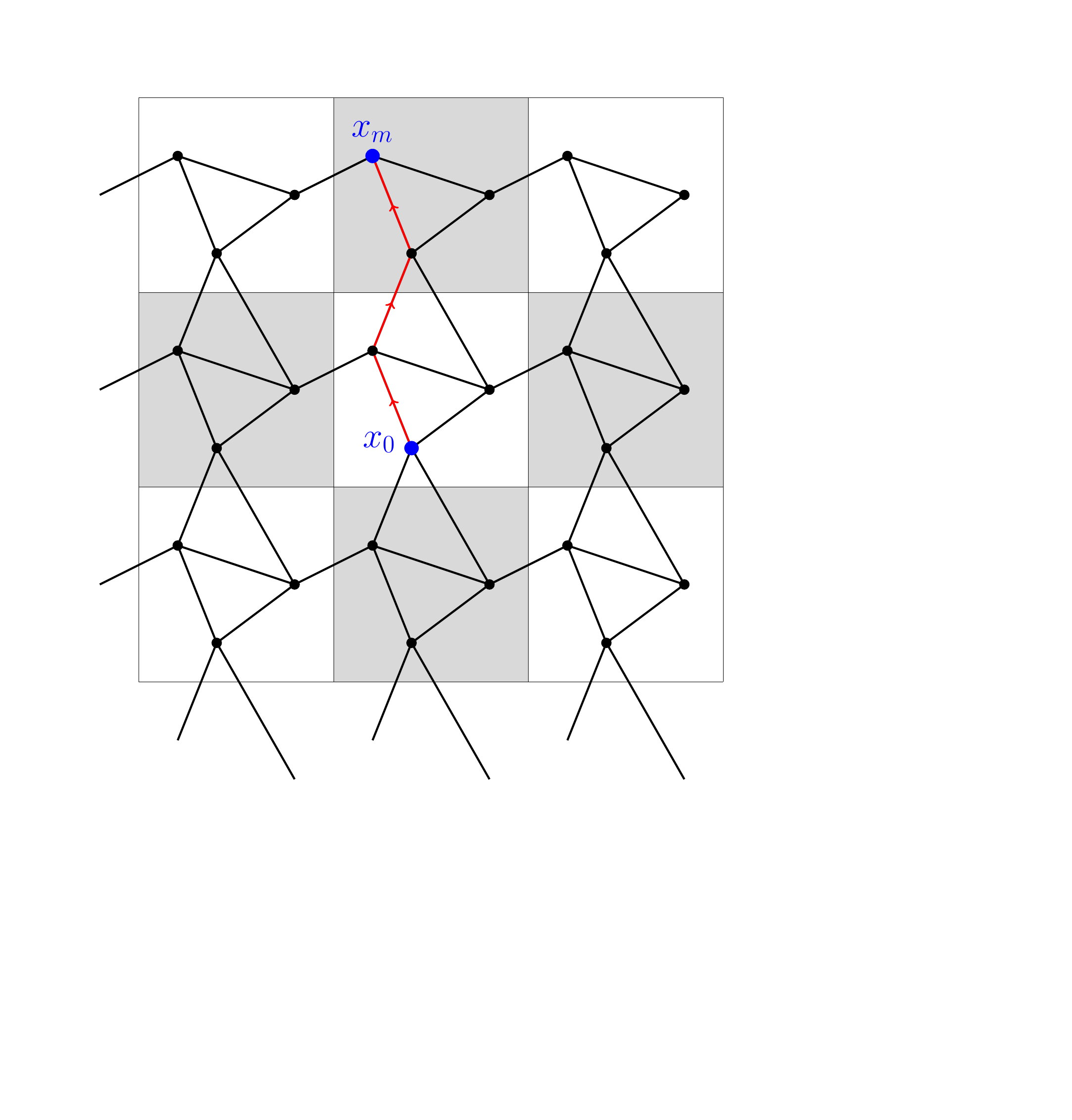}
		\end{subfigure}%
	\begin{subfigure}{.5\textwidth}
		\centering
		\includegraphics[scale=0.45]{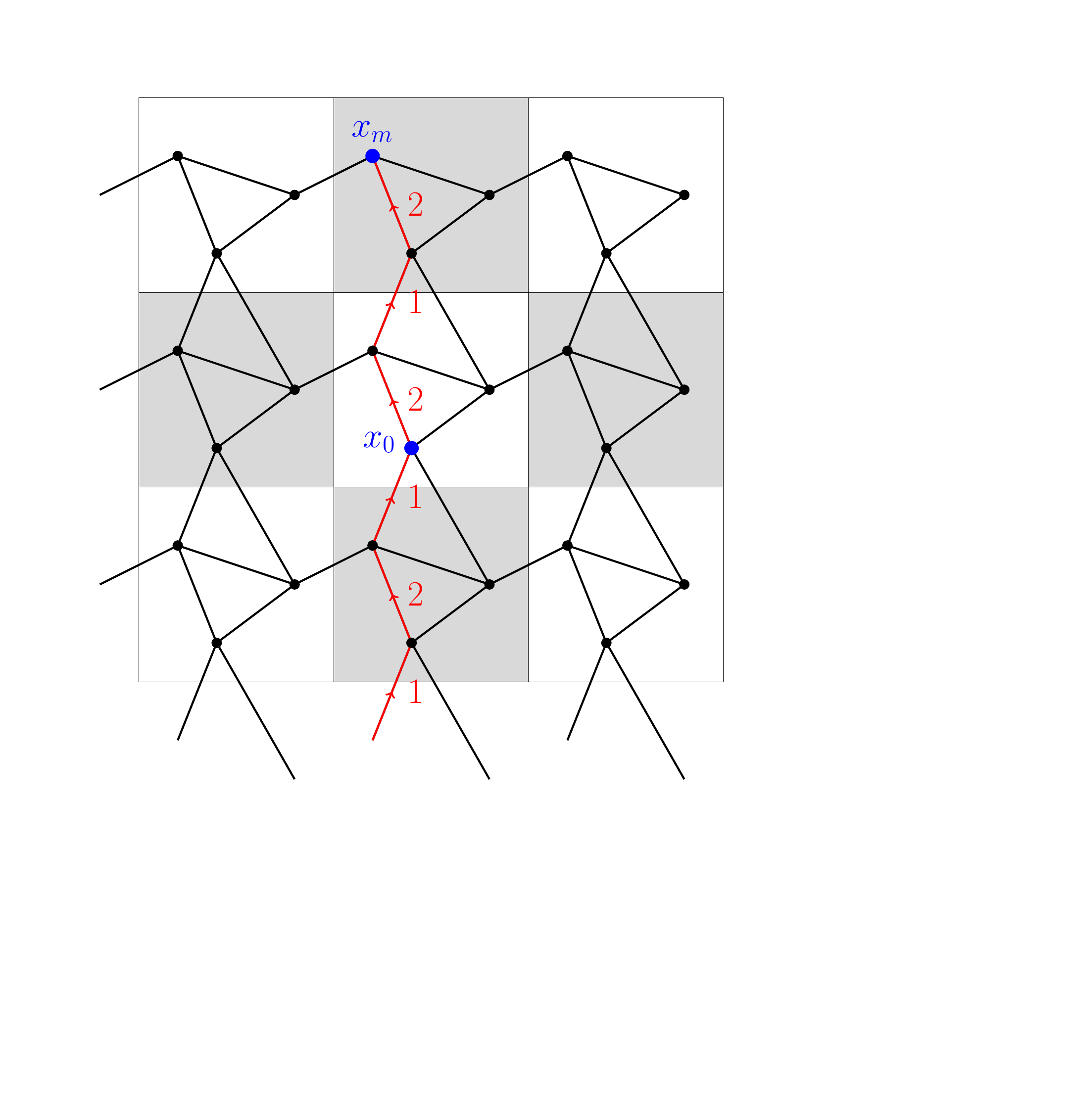}
	\end{subfigure}%
\vspace{-2.5cm}
		\caption{In the first figure, in red, the (directed) path $P$ from $x_0$ to $x_m$, support of the vector field $J_P$. In the second one, in red, the support of the vector field $\tilde J_P$ and its values.} 
		\label{fig:JPq2}
	\end{figure}
	
	\begin{proof}
		\emph{(i)} is straightforward to check, and \emph{(ii)} is a direct consequence. 
	
		To prove \emph{(iii)}, we use the definition of $\tilde J_{P,q}$ to obtain
		\begin{align*}
				\sum_{ (x,y) \in \EQ} 
						\tilde J_P(x,y) 
						\big( y_\sz - x_\sz \big) 
			& = \sum_{ (x,y) \in \EQ}  
					\sum_{z \in \Z^d} 
						J_P( T_z x, T_z y ) 
						\big(y_\sz - x_\sz \big) 
			\\& 
			= 
					 \sum_{ (x,y) \in \cE } 
						 J_P(x,y) 
						 \big(y_\sz - x_\sz \big).
		\end{align*}
	By construction, we have
		\begin{align*}
			\frac12 \sum_{(x,y) \in \cE} 
				J_P(x,y)
					\big(y_\sz - x_\sz \big)
				= \sum_{j=1}^m 
					(x_j)_\sz 
					- (x_{j-1})_\sz 
				= (x_m)_\sz - (x_0)_\sz,
		\end{align*}
	which yields the result.
	
	For \emph{(iv)}, taking $j = e_i$, we use the connectivity and nonemptyness of $(\cX,\cE)$ to find a simple path connecting some $(v,z) \in \cX$ to $(v,z+e_i) \in \cX$. The resulting $\tilde J_P \in \R_a^\cE$ is divergence-free by \emph{(ii)} and $\Eff(\tilde J_P)= e_i$ by \emph{(iii)}, so that $\tilde J_P \in \Rep(e_i)$. For a general $j=\sum_{i=1}^d j_i e_i$ we have $\Rep(j) \supseteq \sum_{i=1}^d j_i \Rep(e_i) \neq \emptyset$.
	\end{proof}

\subsection{The homogenised action}	

We are now in a position to define the homogenised energy density.

\begin{definition}[Homogenised energy density]
	The \emph{homogenised energy density}
	$f_{\hom} : \R_+ \times \R^d \to \R \cup \{ +\infty \}$ 
	is defined by the cell formula
	\begin{equation}
	\begin{aligned}	\label{eq: fhom}
		f_{\hom}(\rho,j) 
		& := 
		\inf\big\{ 
					F(m,J) 
				 \ : \ 
					(m,J) \in \Rep(\rho,j) 
			\big\}.
	\end{aligned}
	\end{equation}
	For $(\rho, j)\in \R_+ \times \R^d$, 
	we say that $(m,J) \in \Rep(\rho,j)$ 
	is an \emph{optimal representative} if 
		$F(m,J) = f_{\hom}(\rho,j)$.
	The set of optimal representatives is denoted by 
	\begin{align*}
		\Rep_o(\rho, j).	
	\end{align*}
\end{definition}	

In view of Lemma \ref{lem:J_Pq}, the set of representatives $\Rep(\rho, j)$ is nonempty for every 
	$(\rho, j)\in \R_+ \times \R^d$.
The next result shows that $\Rep_o(\rho, j)$ is nonempty as well.

\begin{lemma}[Properties of the cell formula]
	\label{lem:cell-formula}
	Let $(\rho, j)\in \R_+ \times \R^d$.
	If $f_\hom(\rho,j) < + \infty$, then the set of optimal representatives $\Rep_o(\rho, j)$ is nonempty, closed, and convex. 
\end{lemma}

\begin{proof}
This follows from the coercivity of $F$ and the direct method of the calculus of variations.
\end{proof}

\begin{lemma}[Properties of $f_\hom$ and $f_\hom^\infty$]	
	\label{lemma:properties_fhom}
	The following properties hold:
\begin{enumerate}[(i)]
	\item 
	The functions $f_\hom$ and $f_\hom^\infty$ are lower semicontinuous and convex. 
	\item 
	There exist constants $c > 0$ and $C < \infty$ such that, for all $\rho \geq 0$ and $j \in \R^d$,
	\begin{align}\label{eq: growth hom}
	f_\hom(\rho,j)
		\geq c|j| - C (\rho + 1),
		\qquad 
	f_\hom^\infty(\rho, j) 
		\geq c|j| - C \rho.
	\end{align}
	
	\item
	The domain 
	$\Dom(f_\hom) \subseteq \R_+ \times \R^d$ 
	%is convex and 
	has nonempty interior. 
	In particular, for any pair 
		$(m^\circ, J^\circ)$
	satisfying \eqref{eq: int dom}, 
	the element 
		$(\rho^\circ, j^\circ) 
			\in (0,\infty) \times \R^d$
	defined by
	\begin{align}	\label{eq:def_r0j0}
		(\rho^\circ, j^\circ) 	 
			:= 
		\bigg(
			\sum_{x \in \XQ} 
				m^\circ(x)
		,
		\frac12\sum_{(x,y) \in \EQ} 
				J^\circ(x,y) \big( y_\sz - x_\sz \big)
		\bigg)			
	\end{align}
	belongs to $\Dom(f_\hom)^\circ$.
\end{enumerate}
\end{lemma}

\begin{proof} 

\emph{(i)}:
	The convexity of $f_\hom$ follows from the convexity of $F$ and the affinity of the constraints. 
	Let us now prove lower semicontinuity of $f_\hom$. 
	
	Take $(\rho, j) \in \R_+ \times \R^d$ and sequences
	$\{\rho_n\}_n \subseteq \R_+$ and $\{j_n\}_n \subseteq \R^d$ converging to $\rho$ and $j$ respectively. 
	Without loss of generality we may assume that 
		$L := \sup_{n \to \infty} 
				f_\hom(\rho_n, j_n) 
			< \infty$. 
	By definition of $f_\hom$, there exist 
		$(m_n, J_n) \in \Rep(\rho_n, j_n)$ 
	such that 
		$F(m_n, J_n) \leq f_\hom(\rho_n, j_n) + \frac1n$.
	From the growth condition \eqref{eq: growth} we deduce that
	\begin{align*}
		\sup_n 
			\sum_{ x \in \XQ } 
				m_n(x) 
		= \sup_n \rho_n < \infty
	\tand	
		\sup_n
			\sum_{(x,y) \in \EQ}
				|J_n(x,y)| 
		\lesssim 1 + L + \sup_n r_n
		< \infty.
	\end{align*}
	From the Bolzano--Weierstrass theorem 
	we infer subsequential convergence 
	of 
		$\{(m_n, J_n)\}_n$ 
	to some $\Z^d$-periodic pair
		$(m, J) \in \R_+^\cX \times \R^\cE$.
	Therefore, by lower semicontinuity of $F$, it follows that
	\begin{align}
		F(m, J)
		\leq \liminf_{n \to \infty}
				F(m_n, J_n) 
		\leq \liminf_{n \to \infty}
				f_\hom(\rho_n, j_n)
	\end{align}
	Since 
		$(m,J) \in \Rep(\rho,j)$, 
	we have
		$f_\hom(\rho, j) \leq F(m, J)$,
	which yields the desired result.  
	Convexity and lower semicontinuity of $f_\hom^\infty$ follow from the definition, see \cite[Section 2.6]{amfupa}.
	\smallskip

	\emph{(ii)}
	Take $\rho \in \R_+$ and $j \in \R^d$.
	If $f_\hom(\rho, j) = + \infty$, the assertion is trivial, so we assume that 
		$f_\hom(\rho, j) < + \infty$.
	Then there exists a competitor 
		$(m, J) \in \Rep(\rho,j)$ 
	such that 
		$F(m,J) \leq f_\hom(\rho,j) + 1$.
	The growth condition \eqref{eq: growth} asserts that
	\begin{align*}
		F(m,J)
		& \geq c 
			\sum_{(x,y) \in \EQ} 
					|J(x,y)| 
			- C \sum_{x \in \XQ} 
				m(x) 
			- C 
	\end{align*}
	Therefore, the claim follows from the fact that 
	\begin{align*}
		R_0 \sum_{(x,y) \in \EQ} |J(x,y)| 
				\gtrsim |j|
		\tand
		\sum_{x \in \XQ} m(x) = r,
	\end{align*}		
	where $R_0 = 
		\max_{(x,y) \in \cE} 
			|  x_\sz - y_\sz|_{\ell_\infty^d}
	$ .

	\emph{(iii)}: 
	Let $( m^\circ, J^\circ ) \in \Dom(F)^\circ$ satisfy Assumption \ref{ass:F}, and define 
		$(\rho^\circ, j^\circ) 
			\in (0,\infty) \times \R^d$ 
	by \eqref{eq:def_r0j0}.
	For $i = 1, \ldots, d$, 
	let $e_i$ be the coordinate unit vector.
	Using Lemma \ref{lem:J_Pq} \emph(iv) we take 
		$J^i \in \Rep(e_i)$. 
	For 
		$\alpha \in \R$ 
			with $|\alpha|$ sufficiently small, and 
		$\beta =  \sum_{i=1}^d \beta_i e_i \in \R^d$
	we define 
	\begin{align*}
		m_\alpha(x) 
			&:= m^\circ(x) + \frac{\alpha }{\#(\XQ)}
			&& 	x \in \cX, \\ 
		J_\beta(x,y) 
			&:= J^\circ(x,y) + \sum_{i=1}^d \beta_i J^i(x,y)
			&& 	(x,y) \in \cE. 
	\end{align*}
	It follows that 
		$(m_\alpha, J_\beta) \in 
			\Rep(\rho^\circ + \alpha, j^\circ + \beta)$, 
	and therefore, 
		$f_\hom(\rho^\circ + \alpha, j^\circ + \beta) 
			\leq F(m_\alpha, J_\beta)$.
	By Assumption \ref{ass:F}, the right-hand side is finite for $|\alpha| + |\beta|$ sufficiently small. 
	This yields the result.
	\end{proof}

The homogenised action $\bA^\cI_\hom$ can now be defined by taking $f = f_\hom$ in Definition \ref{def: A2}.

\subsection{Embedding of solutions to the discrete continuity equation} \label{sec: embedding}

For $\eps > 0$ and $z \in \Z$ (or more generally, for $z \in \R$) let $Q_\eps^z := \eps z + [0, \eps)^d \subseteq \T^d$ 
denote the cube of side-length $\eps$ based at $\eps z$.
For $m \in \Meps$ and $J \in \Mdeps$ we define 
	$\iota_\eps m
		 \in \MT$ 
and 
	$\iota_\eps J 
		\in \MdT$
by 	
\begin{subequations}
	\label{eq:embeddings}
\begin{align}
		\label{eq: iota m}
	\iota_\eps m 
	&	:= 
		\eps^{-d}	
		\sum_{x \in \cX_\eps}
			m(x)
			\Leb^d|_{Q_\eps^{x_\sz}}, \\
				\label{eq: iota J}
	\iota_\eps J
		&	:= 
		\eps^{-d+1}	
		\sum_{(x,y) \in \cE_\eps}
			\frac{J(x,y)}{2}
			\bigg(\int_0^1
				\Leb^d|_{ Q_\eps^{(1-s)x_\sz + s y_\sz}}
			\dd s\bigg)
			(y_\sz - x_\sz), 
\end{align}
\end{subequations}

The embeddings \eqref{eq:embeddings} are chosen to ensure that solutions to the discrete continuity equation are mapped to solutions to the continuous continuity equation, as the following result shows.

\begin{lemma}
	\label{lem:conteq-embed}
	Let $(\bfm, \bfJ) \in\cCE_\eps^\cI$ solve the discrete continuity equation and 
	define 
		$\mu_t = \iota_\eps m_t$
	 and 
		$\nu_t = \iota_\eps J_t$.
	Then $(\bfmu, \bfnu)$ solves the continuity equation
	(i.e., $(\bfmu, \bfnu) \in \bCE^\cI$).
\end{lemma}

\begin{proof}
	Let $\phi : \cI \times \T^d \to \R$ 
	be smooth with compact support. Then:
\begin{align*}
	&\int_\cI
		\int_{\T^d} 
			\nabla \phi \cdot
		\dd \nu_t
	\dd t
	\\& = 	\frac{1}{2  \eps^d}
		\sum_{(x,y) \in \cE_\eps}
			\int_\cI
				J_t(x,y)
				\int_0^1 
					\int_{Q_\eps^{(1-s) x_\sz + s y_\sz}}
						\nabla \phi(t,x) \cdot \eps (y_\sz - x_\sz)
					\dd \Leb^d
				\dd s
			\dd t
	\\&	= \frac{1}{2 \eps^d}
			\sum_{(x,y) \in \cE_\eps}	
			\int_\cI	
				J_t(x,y)	
				\int_0^1 
					\partial_s 
					\bigg(
						\int_{Q_\eps^{(1-s) x_\sz + s y_\sz}}
							\phi
						\dd \Leb^d
					\bigg)
				\dd s
			\dd t
	\\&	= \frac{1}{2 \eps^d}
			\sum_{(x,y) \in \cE_\eps}
				\int_\cI
				J_t(x,y)
				\bigg(
						\int_{Q_\eps^{y_\sz}}
							\phi 
							\dd \Leb^d
					-
						\int_{Q_\eps^{x_\sz}}
							\phi 
							\dd \Leb^d
					\bigg)
			\dd t.
	\end{align*}

On the other hand, the discrete continuity equation yields
\begin{align*}
	\int_\cI
		\int_{\T^d} 
			\partial_t \phi
		\dd \mu_t
	\dd t
	& = \frac{1}{\eps^{d}}	 
	\sum_{x \in \cX_\eps}
	\int_\cI
		m_t(x)
		\partial_t \bigg(\int_{Q_\eps^{x_\sz}} 
		\phi \dd \Leb^d \bigg)
	\dd t		
	\\ & = \frac{1}{2\eps^d}	 
	\sum_{(x,y) \in \cE_\eps}
	\int_\cI
	J_t(x,y) \bigg(\int_{Q_\eps^{x_\sz}} 
			\phi \dd \Leb^d 
			-
			\int_{Q_\eps^{y_\sz}} 
			\phi \dd \Leb^d \bigg)
	\dd t.
\end{align*}		
Comparing both expressions, we obtain the desired identity $\partial_t \bfmu + \nabla \cdot \bfnu = 0$ in the sense of distributions.
\end{proof}

The following result provides a useful bound for the norm of the embedded flux.

\begin{lemma}
	\label{eq:norm-embedded-flux}
	For $J \in \Mdeps$ we have
	\begin{align*}
		|\iota_\eps J|(\T^d)
			\leq 
		\frac{\eps R_0 \sqrt{d}}{2}
		\sum_{(x,y) \in \cE_\eps}
			|J(x,y)|.
	\end{align*}
\end{lemma}

\begin{proof}
	This follows immediately from 
		\eqref{eq:density_embedded_flux},
	since 
		$\Leb^d \big(Q_\eps^{(1-s)x_\sz + s y_\sz}\big) 
			= \eps^d
		$
	and
		$|y_\sz - x_\sz| \leq R_0 \sqrt{d}$ 
	for $(x,y) \in \cE_\eps$.
\end{proof}

Note that both measures in \eqref{eq:embeddings} 
are absolutely continuous 
with respect to the Lebesgue measure. 
The next result provides an explicit expression 
for the density of the momentum field. Recall the definition of the shifting operators $\sigma_\eps^{\bar z}$ in \eqref{eq:def_sigma}.

\begin{lemma}[Density of the embedded flux]
\label{lemma:density_embedded_flux}
	Fix $\eps <\frac1{2R_0}$. 
	For $J \in \R_a^{\cE_\eps}$ we have 
	$\iota_\eps J = j_\eps \Leb^d$ 
	where $j_\eps : \Td \to \R^d$ is given by
\begin{align}	\label{eq:density_embedded_flux}
	j_\eps(u)
		& = 
			\eps^{-d+1}
				\sum_{z \in \Z_\eps^d}
					\chi_{Q_\eps^z}(u)
					\Bigg(\frac12	
						\sum_{
							\substack{(x, y) \in \cE_\eps \\
									x_\sz = z}
								}
					J_u(x, y) 
		\big( y_\sz - x_\sz \big)
	\Bigg) \quad \text{for } u \in \Td.
\end{align}	
Here, $J_u(x, y)$ is a convex combination of $\big\{  \sigma_\eps^{\bar z} J  (x,y )\big\}_{{\bar z}\in \Z_\eps^d}$, i.e., 	
\begin{align*}
	J_u(x, y) 
		& = \sum_{\bar z \in \Z_\eps^d}
				\lambda_u^{\eps,\bar z}(x,y)
				 \sigma_\eps^{\bar z} J  (x,y ),
\end{align*}
where 
	$\lambda_u^{\eps, {\bar z}}(x,y) \geq 0$ 
and 
	$\sum_{{\bar z} \in \Z_\eps^d} 
		\lambda_u^{\eps,{\bar z}}(x,y) 
	= 1$.
Moreover, 
\begin{align}	\label{eq:prop_support_lambda}
	\lambda_u^{\eps, {\bar z}}(x,y) = 0
	 \quad \text{whenever} \ 
	u \in Q_\eps^{x_\sz}, \, |{\bar z}|_\infty > R_0 + 1.	
\end{align}

\end{lemma}

\begin{proof}
	Fix $\eps < \frac{1}{2R_0}$, 
	let $z \in \Z_\eps^d$ 
	and $u \in Q_\eps^z$.
	We have 
	\begin{align*}
		j_\eps(u)
		& =
		\eps^{-d+1}	
		\sum_{(x,y) \in \cE_\eps}
			\frac{J(x,y)}{2}
			\bigg(\int_0^1
				\chi_{Q_\eps^{(1-s)x_\sz + s y_\sz}}(u)
			\dd s\bigg)
			(y_\sz - x_\sz)
	\\&	=
		\eps^{-d+1}	
		\sum_{	
				\substack{(x, y) \in \cE_\eps \\
						x_\sz = z}
			}
		\sum_{\bar z \in \Z_\eps^d}	
			\frac{\sigma_\eps^{\bar z} J  (x,y)}{2}
			\bigg(\int_0^1
				\chi_{Q_\eps^{\bar z + (1-s)x_\sz + s y_\sz}}(u)
			\dd s\bigg)
			(y_\sz - x_\sz),
	\end{align*}
	which is the desired form \eqref{eq:density_embedded_flux} with 
	\begin{align*}
		\lambda_u^{\eps, {\bar z}}(x,y)
		=
		\bigg(\int_0^1
				\chi_{Q_\eps^{\bar z + (1-s)x_\sz + s y_\sz}}(u)
			\dd s\bigg)
	\end{align*}
	for $(x, y) \in \cE_\eps$ 
	with
		$x_\sz = z$.
	Since the family of cubes
	$
		\big\{
			Q_\eps^{\bar z + sy_\sz+(1-s)x_\sz }
		\big\}_{\bar z \in \Z_\eps^d} 			
	$
	is a partition of $\Td$, it follows that 
	$\sum_{{\bar z} \in \Z_\eps^d} 
			\lambda_u^{\eps,{\bar z}}(x,y) 
		= 1$.

To prove the final claim, 
let $(x, y) \in \cE_\eps$ 
with
	$x_\sz = z$ as above and
take $\bar z \in \Z_\eps^d$ with $|{\bar z}|_\infty > R_0 + 1$.
Since
	$|x_\sz - y_\sz| \leq R_0$,
% Finally, suppose that $u \in Q_\eps^z$ and $\| \bar z \|_\infty > R_0 + 1$. 
the triangle inequality yields
\begin{align*}
	\big\| 
		\big( 
			\bar z 
				+ 
			s y_\sz + (1-s) x_\sz
		\big)
			- x_\sz 
	\big\|_\infty 
		\geq 
		\| \bar z \|_\infty - (1-s) 
			\| y_\sz - x_\sz \|_\infty 
		> 
			1,
\end{align*}
for $s \in [0,1]$. 
Therefore, $u \in Q_\eps^z$ implies
	$\chi_{Q_\eps^{\bar z + (1-s)x_\sz + s y_\sz}}(u) = 0$, 
hence 
	$\lambda_u^{\eps, {\bar z}}(x,y) = 0$ 
as desired.
\end{proof}

\section{Main Results}
\label{sec: results}

In this section we present the main result of this paper, which asserts that the discrete action functionals $\cA_\eps$ converge to a continuous action functional $\bA = \bA_\hom$ with the nontrivial homogenised action density function $f = f_\hom$ defined in Section \ref{sec:homogenised}.

\subsection{Main convergence result}

We are now ready to state our main result. 
We use the embedding 	
	$\iota_\eps : \Meps \to \cM_+(\Td)$ 
defined in \eqref{eq: iota m}. 
The proof of this result is given in Section \ref{sec: lower} and \ref{sec: upper}.

\begin{theorem}[$\Gamma$-convergence] \label{thm:main}
	Let $(\cX, \cE)$ be a locally finite and $\Z^d$-periodic connected graph of bounded degree (see Assumption \ref{ass:XE}).
	Let 
		$F : \R_+^\cX \times \R_a^\cE \to \R \cup \{+\infty\}$ 
	be a cost function satisfying Assumption \ref{ass:F}.
	Then the functionals 
		$\cA^\cI_\eps$ 
	$\Gamma$-converge to 
		$\bA^\cI_{\hom}$ as $\eps \to 0$ 
	with respect to the weak (and vague) topology. 
	More precisely:
	\begin{enumerate}[(i)]
\item (\textbf{liminf inequality})
Let $\bfmu \in \cM_+(\cI \times \T^d)$.
For any sequence of curves 
	$\{\bfm^\eps\}_{\eps}$ 
with 
	$\bfm^\eps = (m_t^\eps)_{t \in \cI} 
		\subseteq \R_+^{\cX_\eps}$ 
such that 
	$\iota_\eps \bfm^\eps 	
		  \to \bfmu$ vaguely 
 in $\cM_+(\cI \times \Td)$ 
as $\eps \rightarrow 0$, 
we have the lower bound
\begin{align}	\label{eq:liminf}
	\liminf_{\eps \to 0} 
			\cA^\cI_\eps(\bfm^\eps) \geq \bA^\cI_{\hom}(\bfmu).
\end{align}
\item (\textbf{limsup inequality})
For any 
	$\bfmu \in \cM_+(\cI \times \T^d)$ 
there exists a sequence of curves 
	$\{\bfm^\eps\}_{\eps}$ 
with 
	$\bfm^\eps = (m_t^\eps)_{t \in \cI} 
		\subseteq \R_+^{\cX_\eps}$ 
such that 
	$\iota_\eps \bfm^\eps \to  \bfmu$ 
	 weakly in $\cM_+(\cI \times \T^d)$
as $\eps \to 0$, and we have the upper bound
\begin{align}	\label{eq:limsup}
	\limsup_{\eps \to 0} 
		\cA^\cI_\eps(\bfm^\eps) \leq \bA^\cI_{\hom}(\bfmu).
\end{align}
\end{enumerate}
\end{theorem}

\begin{remark}[Necessity of the interior domain condition]
	Assumption \ref{ass:F} is crucial in order to obtain the $\Gamma$-convergence of the discrete energies. Too see this, let us consider the one-dimensional graph $\cX=\Z$ and the edge-based cost associated with 
	\begin{align*}
		F_{xy}(m(x),m(y),J(x,y)) :=
		\begin{cases}
			\frac
				{J(x,y)^2}{m(x)}
				&
					\text{if } m(x)=m(y) \neq 0 ,\\
			0	
				& \text{if } J(x,y)=m(x)=m(y)=0 ,\\
			\infty
				& \text{otherwise}.
		\end{cases}	
	\end{align*}
	Clearly $F$ satisfies conditions $(a)-(c)$ from Assumption \ref{ass:F}, but $(d)$ fails to hold. The constraint $m(x)=m(y)$ on neighbouring $x,y \in \cX$ forces every $\bfm: \cI \to \R_+^{\cX_\eps}$ with $\cA_\eps(\bfm)<\infty$ to be constant in space (and hence in time, by mass preservation). Therefore, the $\Gamma$-limit of the $\cA_\eps$ is finite only on constant measures $\bfmu= \alpha\Leb^{d+1}$, with $\alpha \in \R_+$. On the other hand, we have\footnote{See also Section \ref{sec:NNI}.} that $f_\hom(\rho,j)=\frac{|j|^2}{\rho}$, which corresponds to the $\bW_2$ action on the line.
	
	It is interesting to note that if the constraint is replaced by something of the form $|m(x)-m(y)|\leq \delta$, for some $\delta >0$, then all the assumptions are satisfied and our Theorem can be applied. In this case, the limit coincides with the $\bW_2$ action.
	
See also Section \ref{sec:NNI} for a general treatment of the cell formula on the integer lattice $\cX=\Z^d$.
\end{remark}

\subsection{Scaling limits of Wasserstein transport problems}

For $1 \leq p < \infty$, recall that the energy density associated to the Wasserstein metric $\bW_p$ on $\R^d$ is given by $f(\rho, j) = \frac{|j|^p}{\rho^{p-1}}$. 
This function satisfies the scaling relations
$f(\lambda \rho, \lambda j) 
= \lambda f(\rho, j)$
and 
$f(\rho, \lambda j) 
= |\lambda|^p f(\rho, j)$
for $\lambda \in \R$.
  
In discrete approximations of $\bW_p$ on a periodic graph $(\cX,\cE)$, it is reasonable to assume analogous scaling relations for the function $F$, namely
$F(\lambda m, \lambda J) = \lambda F(m, J)$ 
and 
$F(m, \lambda J) = |\lambda|^p F(m, J)$.
The next result shows that if such scaling relations are imposed, we always obtain convergence to $\bW_p$ with respect to some norm on $\R^d$.
This norm does not have to be Hilbertian (even in the case $p=2$) and is characterised by the cell problem \eqref{eq: fhom}.

\begin{corollary}
Let $1 \leq p < \infty$, and suppose that $F$ has the following scaling properties
for  $m \in \R_+^\cX$ and $j \in \R_a^\cE$:
\begin{enumerate}[(i)]
	\item $F(\lambda m, \lambda J) = \lambda F(m, J)$ 
			for all $\lambda \geq 0$;
	\item $F(m, \lambda J) = |\lambda|^p F(m, J)$ 
			for all $\lambda \in \R$.
\end{enumerate}
Then $f_\hom(\rho, j) 
		= \frac{\| j \|^p}{\rho^{p-1}}$ 
for some norm $\|\cdot\|$ on $\R^d$.
\end{corollary}

\begin{proof}
Fix $\rho > 0$ and $j \in \R^d$.
The scaling assumptions imply that 
\begin{align}
\label{eq:f_hom_hom}	
	f_\hom(\lambda \rho, \lambda j) 
	= \lambda f_\hom(\rho, j)
\tand
	f_\hom(\rho, \lambda j) 
	= |\lambda|^p f_\hom(\rho, j).
\end{align}
Consequently,
\begin{align*}
	f_\hom(\rho, j) 
		= \rho f_\hom(1, j/\rho) 
		= \frac{f_\hom(1, j)}{\rho^{p-1}}.
\end{align*}

We claim that 
	$f_\hom(1, j) > 0$ whenever $j \neq 0$.
Indeed, it follows from \eqref{eq: growth hom}
that $f_\hom(1, j) > 0$
whenever $|j|$ is sufficiently large. 
By homogeneity \eqref{eq:f_hom_hom}, 
the same holds for every $j \neq 0$.
It also follows from \eqref{eq:f_hom_hom} that
	$f_\hom(1, 0) = 0$.

We can thus define
	$\|j\| := f_\hom(1,j)^{1/p} \in [0,\infty)$.
In view of the previous comments, 
we have 
	$\|0\| = 0$
and
	$\|j\| > 0$ for all $j \in \R^d \setminus \{0\}$.
The homogeneity \eqref{eq:f_hom_hom} implies that 
	$\| \lambda j\| = |\lambda|\, \|j \|$ 
for $j \in \R^d$ and $\lambda \in \R$.

It remains to show the triangle inequality
	$\| j_1 + j_2 \| \leq \|j_1\| + \|j_2\|$
for	$j_1, j_2 \in \R^d$.
Without loss of generality we assume that 
	$\|j_1\| + \|j_2\| > 0$.
For $\lambda \in (0,1)$,
the convexity of $f_\hom$ (see Lemma \ref{lemma:properties_fhom}) and the homogeneity \eqref{eq:f_hom_hom}
yield 
\begin{align*}
	f_\hom(1,j_1 + j_2) 
	\leq (1-\lambda) f_\hom 
			\Big(1, \frac{j_1}{1-\lambda} \Big) 
		+ \lambda f_\hom
			\Big(1, \frac{j_2}{\lambda} \Big) 
	= 	\frac{f_\hom(1, j_1)}{(1 - \lambda)^{p-1}} 
		+ \frac{f_\hom(1, j_2)}{\lambda^{p-1}}.
\end{align*}
Substitution of 
	$\lambda = \frac{\|j_2\|}{\|j_1\| + \|j_2\|}$ 
yields the triangle inequality.
\end{proof}

\subsection{Compactness results}

As we frequently need to compare measures with unequal mass in this paper, it is natural to work with the 
the \emph{Kantorovich--Rubinstein norm}. 
This metric is closely related to the transport distance $\bW_1$; see Appendix \ref{sec:KR}.

The following compactness result holds for solutions to the continuity equation with bounded action. 
As usual, we use the notation 
	$\bfmu(\ddd x, \ddd t) = \mu_t(\ddd x) \dd t$.

\begin{theorem}[Compactness under linear growth] \label{thm: compactness}
	Let $\bfm^\eps:\cI \to  \R_+^{\cX_\eps}$ be such that
	\begin{align*}
		\sup_{\eps > 0} \cA^\cI_\eps(\bfm^\eps) 
		< \infty
			\tand
		\sup_{\eps > 0} 
		\bfm^\eps(\cI \times \cX_\eps) 
		< \infty.	
	\end{align*} 
	Then there exists a curve 
		$(\mu_t)_{t \in \cI} \in \BVKR$ 
	such that, up to extracting a subsequence,
	\begin{enumerate}[(i)]
		\item 
			$\iota_\eps \bfm^\eps \to \bfmu$ weakly in $\cM_+(\cI \times \Td)$;
		\item 
			$\iota_\eps m_t^\eps \to \mu_t$ weakly in $\cM_+(\Td)$
			for almost every $t\in \cI$;
		\item $t \mapsto \mu_t(\T^d)$ is constant.
	\end{enumerate}
\end{theorem}

The proof of this result is given in Section \ref{sec: compactness}.

Under a superlinear growth condition on the cost function $F$, the following stronger compactness result holds. 

\begin{assumption}[Superlinear growth]
	\label{ass:stronger}
	We say that $F$ is of \emph{superlinear growth}
	if 
	there exists a function 	
	$\theta:[0,\infty) \to [0,\infty)$ 
	with $\lim_{t\to \infty} \frac{\theta(t)}{t} = \infty$ 
	and a constant $C \in \R$ such that
	\begin{equation}
		\label{eq: superlinear}
		F(m,J) 
		\geq 
		(m_0 + 1)
		\theta\bigg( \frac{J_0}{m_0+1} \bigg)
		- C (m_0 + 1)
	\end{equation}
	for all $m\in  \R_+^\cX$ and all $J\in \R^\cE_a$, where
	\begin{equation}
		m_0 
		= \sum_{\substack{x \in \cX
				\\
				|x|_{\ell_\infty^d} \leq R
		}}
		m(x)
		\tand
		J_0 
		= \sum_{(x,y) \in \EQ} |J(x,y)|,
	\end{equation}
	with $R = \max\{R_0, R_1\}$ as in Assumption \ref{ass:F}.
\end{assumption}

\begin{remark}
	The superlinear growth condition \eqref{eq: superlinear} implies the linear growth condition \eqref{eq: growth}.
	To see this, suppose that $F$ has superlinear growth.
	Let $v_0 > 0$ be such that 
	$\theta(v) \geq v$ for $v\geq v_0$. 
	If $\frac{J_0}{m_0 + 1} \geq v_0$, we have
	\begin{equation}\label{eq: sll1}
		F(m,J) 
		\geq 
		(m_0 + 1)
		\theta\bigg( \frac{J_0}{m_0+1} \bigg)
		- C (m_0 + 1)
		\geq  J_0 - C (m_0 + 1).
	\end{equation}
	On the other hand, if
	$\frac{J_0}{m_0+1} < v_0$, 
	the nonnegativity of $\theta$ implies that 
	\begin{equation}\label{eq: sll2}
		F(m,J) 
		\geq 
		- C  (m_0 + 1)
		\geq 
		\frac{C}{v_0} J_0 
		- 2C (m_0 + 1).
	\end{equation}
	Combining \eqref{eq: sll1} and \eqref{eq: sll2}, we have
	\begin{equation*}
		F(m,J) \geq 
		\min\bigg\{ 1, \frac{C}{v_0} \bigg\} J_0 
		- 2C (m_0 + 1),
	\end{equation*}
	which is of the desired form \eqref{eq: growth}.
\end{remark}

\begin{example}
The edge-based costs
\[
	F(m,J) = \frac12\sum_{(x,y)\in \EQ}| J(x,y)|^p
\]
have superlinear growth if and only if $1<p<\infty$ (with $\theta(t) = c t^p$ and  $c = |\cE^Q|^{1-p}$).
Indeed, 
\[
	2F(m,J) 
	= 		\sum_{(x,y) \in \EQ} |J(x,y)|^p 
	\geq 	c J_0^p 
	\geq 	c \frac{J_0^p}{(m_0+1)^{p-1}} 
	= 		c (m_0+1) \theta\left(\frac{J_0}{m_0+1}\right).
\]
\end{example}

\begin{example}
	The functions \eqref{eq: Wp} arising in discretisation of $p$-Wasserstein distances have superlinear growth if and only if $p > 1$ (with $\theta(t) = t^p$).
	
	To see this, consider the function $G(\alpha,\beta,\gamma) := \frac12 \frac{|\gamma|^p}{\Lambda(\alpha,\beta)^{p-1}}$. 
	Since $G$ is convex, non increasing in $(\alpha,\beta)$, and positively one-homogeneous,
	we obtain 
	\begin{align*}
		F(m,J) 
			& =  \sum_{(x,y) \in \cE^Q} 
					G\big(q_{xy}m(x), q_{yx}m(y), J(x,y)\big)
			\\ & 
			\geq  
				G\left( 
					\sum_{(x,y)\in \cE^Q} q_{xy}m(x), 
					\sum_{(x,y)\in \cE^Q} q_{yx}m(y), 
					\sum_{(x,y)\in \cE^Q} |J(x,y)| 
				\right)
			\\ & 	
				\geq    c G(m_0,m_0,J_0) 
			\geq 	\frac{c}{2} \frac{J_0^p}{(m_0+1)^{p-1}}
			= 
		 	\frac{c}{2} 
				(m_0+1)
			  \theta\left(\frac{J_0}{m_0+1}\right),
	\end{align*}
	where $c>0$ depends on $R$, the maximum degree and the weights $q_{xy}$.

\end{example}

\begin{theorem}[Compactness under superlinear growth]
    \label{theorem: uniform compactness}
Suppose that Assumption \ref{ass:stronger} holds. 
Let $\bfm^\eps : \cI \to \R_+^{\cX_\eps}$ be such that
\begin{align*}
	\sup_{\eps > 0} \cA^\cI_\eps(\bfm^\eps) 
	< \infty
	\tand
	\sup_{\eps > 0} 
	\bfm^\eps(\cI \times \cX_\eps) 
	< \infty.	
\end{align*} 
Then there exists a curve 
$(\mu_t)_{t \in \cI} \in \WKR$ 
such that, up to extracting a subsequence,
\begin{enumerate}[(i)]
\item 
	$\iota_\eps \bfm^\eps \to \bfmu$ weakly in $\cM_+(\cI \times \Td)$;
\item 
$\| \iota_\eps m_t^\eps -  \mu_t\|_{\KR(\T^d)} \to 0$
uniformly for $t \in \cI$;
\item $t \mapsto \mu_t(\T^d)$ is constant.
\end{enumerate}	
\end{theorem}

This is proven in Section \ref{subsec: compactness}. 

Note that curve $t \mapsto \mu_t \in \WKR$ 
can be continuously extended to $\overline{\cI}$.
Therefore, it is meaningful to assign boundary values to these curves.

\subsection{Result with boundary conditions}

Under Assumption \ref{ass:stronger}, we are able to obtain the following result on the convergence of dynamical optimal transport problems. Fix $\cI=(a,b)\subset \R $ an open interval. Define for $m^a,m^b\in \R_+^{\cX_\eps}$ with $m^a(\cX_\eps) = m^b(\cX_\eps)$ the minimal action as
\begin{equation}
\cMA_\eps^\cI(m^a,m^b) := \inf\left\{ \cA_\eps^\cI(m)\,:\,m_a = m^a, m_b = m^b) \right\}.
\end{equation}

Similarly, define the minimal homogenised action for $\mu^a,\mu^b\in \cM_+(\Td)$ with $\mu^a(\Td) = \mu^b(\Td)$ as
\begin{equation}
\bMA_\hom^\cI(\mu^a,\mu^b) := \inf\left\{ \bA_\hom^\cI(\mu)\,:\,\mu_a = \mu^a, \mu_b = \mu^b)\right\}.
\end{equation}
Note that in general, both $\bMA_\hom^\cI$ and $\cMA_\eps^\cI$ may be infinite even if the two measures have equal mass. Here, the values $\mu_a$ and $\mu_b$ are well-defined under Assumption \ref{ass:stronger} by Theorem \ref{theorem: uniform compactness}. Under linear growth, $\mu_a$ and $\mu_b$ can still be defined using the trace theorem in $\BV$, but we cannot prove the following statement in that case (see also Remark \ref{rem:linear_minimal}). We prove this in Section \ref{subsec: compactness2}.

\begin{theorem}[$\Gamma$-convergence of the minimal actions]
		\label{theorem: uniform compactness 2}
Assume that Assumption \ref{ass:stronger} holds. Then the minimal actions $\cMA_\eps^\cI$ $\Gamma$-converge to $\bMA_{\hom}^\cI$ in the weak topology of $\cM_+(\Td) \times \cM_+(\Td)$. Precisely:
\begin{enumerate}[(i)]
	\item For any sequences $m_\eps^a$, $m_\eps^b \in \R_+^{\cX_\eps}$ such that $\iota_\eps m_\eps^i \to \mu^i$ weakly in $\cM_+(\Td)$ as $\eps \to 0$ for $i=a,b$ , we have
\begin{align}	\label{eq:liminf_boundary}
	\liminf_{\eps \to 0} \cMA_\eps^\cI(m^a_\eps, m^b_\eps)
		\geq \bMA_{\hom}^\cI(\mu^a,\mu^b).
\end{align}
	\item For any $(\mu^a,\mu^b) \in \cM_+(\Td) \times \cM_+(\Td)$, there exist two sequences $m_\eps^a, m_\eps^b \in \R_+^{\cX_\eps}$ such that $\iota_\eps m_\eps^i \to \mu^i$ weakly in $\cM_+(\Td)$ as $\eps \to 0$ for $i=a,b$ and 
\begin{align}	\label{eq:limsup_boundary}
	\limsup_{\eps \to 0}  \cMA_\eps^\cI(m^a_\eps, m^b_\eps) 
		\leq \bMA_{\hom}^\cI(\mu^a,\mu^b).
\end{align}
\end{enumerate}
\end{theorem}

\section{Proof of compactness and convergence of minimal actions}\label{sec: compactness}
This section is divided into three sub-parts: in the first one, we prove the general compactness result Theorem \ref{thm: compactness}, which is valid under the linear growth assumption \eqref{ass:F}.

In the second and third part, we assume the stronger superlinear growth condition \eqref{ass:stronger} and prove the improved compactness result Theorem \ref{theorem: uniform compactness} and the convergence results for the problems with boundary data, i.e. Theorem \ref{theorem: uniform compactness 2}.

\subsection{Compactness under linear growth}
The only assumption here is the linear growth condition \eqref{ass:F}.

\begin{proof}[Proof of Theorem \ref{thm: compactness}]
	For $\eps > 0$, let 
		$\bfm^\eps: \cI \to  \R_+^{\cX_\eps}$ 
	be a curve such that
	\begin{align}
		\label{eq:bounds-mass-action}
		\sup_{\eps > 0} \cA^\cI_\eps(\bfm^\eps) 
		< \infty
			\tand
		\sup_{\eps > 0} 
		\bfm^\eps(\cI \times \cX_\eps) 
		< \infty.	
	\end{align} 
	We can find a solution to the discrete continuity equation
		$(\bfm^\eps, \bfJ^\eps) \in \cCE_\eps^\cI$, 
	such that
	\begin{align*}
		\sup_{\eps > 0} 
			\cA_\eps^\cI(\bfm^\eps, \bfJ^\eps) 
		< \infty.
	\end{align*}
	Set
		$(\mu_t^\eps, \nu_t^\eps) 	
			:= 
		(\iota_\eps m_t^\eps, \iota_\eps J_t^\eps)$, 
	where $\iota_\eps$ is defined in \eqref{eq:embeddings}.
	Lemma \ref{lem:conteq-embed} implies that $(\bfmu^\eps, \bfnu^\eps) \in \bCE^\cI$ for every $\eps>0$.

	Using Lemma \ref{eq:norm-embedded-flux}, 
	the growth condition \eqref{eq: growth}, 
	and the bounds \eqref{eq:bounds-mass-action} on the masses and the action, 
	we infer that
	\begin{align}\label{eq: total variation}
		\sup_{\eps > 0 } 
			| \bfnu^\eps | 
			\big( \cI \times \Td \big) 
		\leq \frac{R_0 \sqrt{d}}{2} 
			\sup_{\eps > 0} 
				\eps 
				\int_\cI 
					\sum_{(x,y)\in \cE_\eps} 
						|J_t^\eps(x,y)| 
				\dd t 
		< \infty.
	\end{align}
	Up to extraction of a subsequence, the Banach--Alaoglu Theorem yields existence of a measure 
		$\bar \bfnu \in \cM^d(\overline\cI \times \Td)$ 
	such that
		$\bfnu^\eps \to \bar \bfnu$ weakly in $\cM^d(\overline\cI \times \Td)$. 
	It also follows that 
	$|\bar \bfnu|(\overline\cI \times \Td) 
	\leq
		\liminf_{\eps \to 0}
			|\bfnu^\eps|(\cI \times \Td) 
		 < \infty$;
	see, e.g., 
		\cite[Theorem 8.4.7]{Bogachev}.

	Furthermore, \eqref{eq:BV-bound} and 
	\eqref{eq: total variation} imply that
	the $\BV$-seminorms of $\bfmu^\eps$ are bounded:
	\begin{align}
		\label{eq:BV-bound}
		\sup_{\eps > 0} 
			\| \bfmu^\eps \|_{\BVKR}
		 \leq \sup_{\eps > 0} 
		 	|\bfnu^\eps|(\cI \times \Td) 
		< \infty,
	\end{align}
	In particular, 
		$\sup_{\eps > 0} 
		\bfmu^\eps (\cI \times \T^d) < \infty$.
	Thus, by another application of the Banach--Alaoglu Theorem,
	there exists a measure 
		$\bfmu \in \cM_+(\overline\cI \times \Td)$ 
	and a subsequence (not relabeled) 
	such that
		$\bfmu^\eps \to \bfmu$ weakly in $\cM_+(\overline\cI \times \Td)$.

	We claim that $\bfmu$ does not charge the boundary 
		$(\overline\cI \setminus \cI ) \times \Td$
	and that
		$\bfmu(\dd x, \dd t) = \mu_t(\ddd x) \dd t$
	for a curve $(\mu_t)_{t \in \cI}$	
	of constant total mass in time.		 
	To prove the claim, write $e_1(t,x) := t$, 
	and note that each curve 
		$t \mapsto \mu_t^\eps$ 
	is of constant mass. 
	Therefore, the time-marginals 
		$(e_1)_{\#} \bfmu^\eps \in \cM_+(\cI)$ 
	are constant multiples of the Lebesgue measure. 
	Since these measures
	are weakly-convergent 
	to the time-marginal $(e_1)_{\#} \bfmu$, 
	it follows that the latter is also a
	constant multiple of the Lebesgue measure,
	which implies the claim.
	
	By what we just proved, $\bfmu$ can be identified with 
	a measure on the open set 
		$\cM_+(\cI \times \Td)$.
	Let $\bfnu$ be the restriction of $\bar \bfnu$ to $\cI \times \T^d$.
	Since $\bfmu^\eps$ (resp. $\bfnu^\eps$) converges vaguely to $\bfmu$ (resp. $\bfnu$), 
	it follows that 
		$(\bfmu, \bfnu)$ 
	belongs to $\bCE^\cI$.
	
	In view of \eqref{eq:BV-bound}, we can apply the $\BV$-compactness theorem (see, e.g., \cite[Theorem B.5.10]{mielkeroubi}) 
	to obtain a further subsequence such that 	
		$\| \mu_t^\eps -  \mu_t\|_{\KR(\T^d)} \to 0$ 
	for almost every $t \in \cI$, and the limiting curve $\bfmu$ belongs to $\BVKR$. 
	Proposition \ref{prop:KR-weakstar} yields
		$\mu_t^\eps \to \mu_t$ weakly in $\cM_+(\Td)$
	for almost every $t \in \cI$.
\end{proof}

\subsection{Uniform compactness under superlinear growth}\label{subsec: compactness}
In the last two sections, we shall work with the stronger growth condition from Assumption \ref{ass:stronger}.

\begin{remark}[Property of $f_\hom$, superlinear case]
	\label{rem:prop_fhom_superlinear}
	Let us first observe that under Assumption \ref{ass:stronger}, one has superlinear growth of $f_\hom$:
	\begin{align*}
		f_\hom(\rho, j) \geq 
		\theta\Big( \frac{|j|}{\rho+1} \Big)(\rho+1)
		- C (\rho + 1)
		, \quad
		\forall \rho \geq 0, \; j \in \R^d,
	\end{align*}
	where we recall $\theta:[0,\infty) \to [0,\infty)$ is such that $\lim_{t \to \infty} \frac{\theta(t)}{t} = +\infty$.
	
	In addition for all $j\neq 0$ we have
	\begin{align}\label{eq: fhom0j}
		f_\hom^\infty(0,j) = \lim_{t\to \infty} \frac1t f_\hom(\rho_0, j_0 + tj) \geq \lim_{t\to \infty} \frac{\theta\left(\frac{|j_0 + t j|}{\rho_0+1} \right)(\rho_0+1)}{t} = \infty. 
	\end{align}
	
	In particular, if $\bA_\hom^\cI(\bfmu, \bfnu)<\infty$, then $\bfnu \ll \bfmu + \Leb^{d+1}$. Indeed, fix $\bfsigma \in \cM_+(\cI \times \Td)$ as in \eqref{eq:decomp_sigma} and suppose that $(\bfmu + \Leb^{d+1})(A) = 0$ for some $A \subset \cI \times \Td$. By positivity of the measures, this implies that $\bfmu(A) = \Leb^{d+1}(A) =0$, thus by construction 
	\begin{align*}
		\bfmu^\perp(A) = 0 	
			\tand
		\bfnu(A) = \bfnu^\perp(A).
	\end{align*}
	From the first condition and $\bfmu^\perp = \rho^\perp \bfsigma$, we deduce that $\rho^\perp(t,x) = 0 $ for $\bfsigma$-a.e. $(t,x) \in A$. From the assumption of finite energy and \eqref{eq: fhom0j}, writing $\bfnu^\perp = j^\perp \bfsigma$, we infer that $j^\perp(t,x) = 0$ for $\bfsigma$-a.e. $(t,x) \in A$ as well. It follows that $\bfnu(A) = \bfnu^\perp(A) = 0$, which proves the claim.
\end{remark}
We are ready to prove Theorem \ref{theorem: uniform compactness}.
\begin{proof}[Proof of Theorem \ref{theorem: uniform compactness}]
	Let 
		$\{\bfm^\eps\}_\eps$ 
	be a sequence of measures such that
	\begin{align}	
		\label{eq:uniform_bounds_strong_comp}
		M := 
			\sup_\eps 
				\bfm^\eps (\cI \times \cX_\eps) +1 
			< \infty
			\tand
		A := \sup_\eps 
				\cA_\eps^\cI(\bfm^\eps) 
			< \infty.		
	\end{align} 
	Thanks to Remark \eqref{rem:prop_fhom_superlinear}, we have that $\bfnu \ll \bfmu + \Leb^{d+1}$ for all solutions $(\bfmu,\bfnu)\in \bCE^\cI$ with $\bA_{\hom}^\cI(\bfmu)<\infty$. Applying Lemma \ref{lemma:disintegration} we can write $\bfmu = \dd t\otimes \mu_t$ and because $\Leb^{d+1} = \dd t\otimes \Leb^d $, we also have disintegration $\bfnu = \dd t\otimes \nu_t$ with $\nu_t \ll \mu_t + \Leb^d$ for almost every $t\in \cI$.
	
	Moreover, it follows from the definition of $\bCE^\cI$ that, for any test function $\phi\in \cC_c^1 (\cI; \cC^1(\Td))$ we have 
	\[
	\langle \bfmu, \partial_t \phi \rangle = -\langle \bfnu, \nabla \phi \rangle = -\int_\cI \Big\langle \frac{\dd \nu_t}{\dd ( \mu_t + \Leb^{d})} (\mu_t + \Leb^{d} ), \nabla \phi \Big\rangle \dd t.
	\]
	This shows that 
		$\dd t \otimes \mu_t \in \WKR$, 
	with weak derivative
	\begin{align*}
		\partial_t\mu_t = \nabla \cdot 
		\Big(
		\frac{\dd \nu_t}{\dd ( \mu_t + \Leb^{d})} (\mu_t + \Leb^{d} ) 
		\Big) 
		\in \KR(\Td)  \quad \text{for a.e.} \; t \in \cI.
	\end{align*}

	We are left with showing uniform convergence of $\iota_\eps m^\eps_t \to \mu_t$ in $\KR(\Td)$. 
	We claim that the curves
		$\{ t \mapsto \iota_\eps m_t^\eps  \}_\eps$
		are equicontinuous with respect to
		the Kantorovich--Rubinstein norm $\|\cdot \|_{\KR(\Td)}$.
	
	To show the claimed equicontinuity, take 
		$\phi\in \cC^1(\Td)$
	and $s, t \in \cI$ with	$s<t$. 
	Since 
		$(\iota_\eps m_t^\eps, \iota_\eps J_t^\eps) 
			\in \bCE^\cI
		$ 
	we obtain using 
		Lemma \ref{eq:norm-embedded-flux},
	\begin{align}
		\label{eq: uniform continuity}
		\begin{aligned}
			\bigg| 
				\int_{\T^d} 
					\phi \dd (\iota_\eps m^\eps_t)
			-
				\int_{\T^d} 
					\phi \dd (\iota_\eps m^\eps_s)
			\bigg|
			& =
				\bigg|
				\int_s^t 
					\int_{\T^d}
						\nabla \phi \cdot 
					\dd (\iota_\eps J^\eps_r)
				\dd r
				\bigg|
			\\ &\leq
			   \|\nabla \phi\|_{\cC(\Td)} \int_s^t |\iota_\eps J_r^\eps|(\Td) \dd r 	
			\\ &\leq
			\frac{R_0 \sqrt{d}}{2} \|\nabla \phi\|_{\cC(\Td)} \int_s^t \sum_{(x,y)\in \cE_\eps} \eps|J^\eps_r(x,y)|\dd r, 
		\end{aligned}
	\end{align}
	To estimate the latter integral, we consider
	for $z\in \Z_\eps^d$ the quantities
	\begin{align*}
		\sm^\eps_r(z) 
			:= 
		\sum_{\substack{
				x\in \cX_\eps 
				\\ 
				|x_\sz - z|_{\ell_\infty^d} \leq R}
			}
			m^\eps_r(x)
	\tand
		\sJ^\eps_r(z) 
			:= 
		\sum_{\substack{
				(x,y)\in \cE_\eps
				\\
				x_z = z}
			}
			|J^\eps_r(x,y)|.
	\end{align*}
	We fix a ``velocity threshold'' $v_0 > 0$,
	and split $\Z_\eps^d$ into the low velocity region
	$\cZ_- := \{ z \in \Z_\eps^d \ : \ \frac{\eps |\sJ^\eps_r(z)|}
	{\sm^\eps_r(z) + \eps^d} 
	\leq v_0 \}$
	and its complement $\cZ_+ := \Z_\eps^d \setminus \cZ_-$.
	Then:
	\begin{align}	
	\label{eq:estimate_TV_1}
		\sum_{z\in \cZ_-} 
			\eps \sJ^\eps_r(z) 
	\leq 
		v_0 \sum_{z \in \cZ_-} \big( \sm^\eps_r(z) + \eps^d \big)
	\leq 
		C_R \big( m^\eps_r(\cX_\eps) + 1 \big) v_0,
	\end{align}
	where $C_R := (2R + 1)^d$.
	For $z \in \cZ_+$
	we use the growth condition \eqref{eq: superlinear} 
	to estimate
	\begin{align*}
		\begin{aligned}
			\eps \sJ^\eps_r(z) 
			&
			\leq 
			\big( \sm^\eps_r(z)  + \eps^d \big)
			\theta\Big(
					\frac{\eps \sJ^\eps_r(z)}
						 {\sm^\eps_r(z) + \eps^d}
				\Big) 
				\sup_{v > v_0} \frac{v}{\theta(v)} 
			\\
			&\leq
			\eps^d 
				\bigg( 
				F\bigg(\frac{\tau_\eps^z m}{\eps^d}, 
					   \frac{\tau_\eps^zJ}{\eps^{d-1}}
				\bigg) 
				+ C \Big(\frac{\sm^\eps_r(z)}{\eps^d} 
						+ 1 
					\Big) 
				\bigg)
			\sup_{v > v_0} \frac{v}{\theta(v)}.
		\end{aligned}
	\end{align*}
	Since \eqref{eq: superlinear} implies non-negativity of the term in brackets, we obtain
	\begin{align}	
	\begin{aligned}
		\label{eq:estimate_TV_3}
			\sum_{z\in \cZ_+} 
				\eps \sJ^\eps_r(z) 
	&	\leq 
			\sum_{z\in \T^d} \eps^d 
		\bigg( 
		F\bigg(\frac{\tau_\eps^z m}{\eps^d}, 
			   \frac{\tau_\eps^z J}{\eps^{d-1}}
		\bigg) 
		+ C \Big(\frac{\sm^\eps_r(z)}{\eps^d} 
				+ 1 
			\Big) 
		\bigg)
	\sup_{v > v_0} \frac{v}{\theta(v)}
	\\& \leq 
				\cF_\eps(m^\eps_r, J^\eps_r) 
				+ C \big( m^\eps_r(\cX_\eps) + 1 \big)
			\sup_{v > v_0} \frac{v}{\theta(v)}.
	\end{aligned}
	\end{align}
	Integrating in time, 
	we 
	combine \eqref{eq:estimate_TV_1} and 
	\eqref{eq:estimate_TV_3} with 
	\eqref{eq:uniform_bounds_strong_comp}
	to obtain
	\begin{align}
		\label{eq:eps-J}
		\begin{aligned}
		&	\int_s^t 
				\sum_{(x,y)\in \cE_\eps} 
					\eps|J^\eps_r(x,y)|
			\dd r 
			 = 
			\int_s^t 
				\sum_{z \in \Z_\eps^d} 
					\eps \sJ^\eps_r(z)
			\dd r 
		\leq g(t-s), 
		\\& \text{where} \quad
			g(r) := 
			\inf_{v_0 > 0} 
			\bigg\{
			r C_R M v_0 
			+
			\Big( 
				A + C ( M + |\cI| )
			\Big) 	
				\sup_{v > v_0} \frac{v}{\theta(v)}	
			\bigg\}.
	\end{aligned}\end{align}
	Combining \eqref{eq: uniform continuity} and 
	\eqref{eq:eps-J}  
	we conclude that
	\begin{align*}
		\sup_{\eps > 0} 
			\|\iota_\eps m^\eps_t 
			- \iota_\eps m^\eps_s\|_{\KR(\Td)} 
		& \leq     
		\sup_{\eps >0} 
			\sup_{\|\phi\|_{\cC^1(\Td)}\leq 1}
			\bigg| 
				\int_{\T^d} 
					\phi \dd (\iota_\eps m^\eps_t)
			-
				\int_{\T^d} 
					\phi \dd (\iota_\eps m^\eps_s)
			\bigg|
		\\& \leq 		
		\frac{R_0 \sqrt{d}}{2} 
				g(t-s).
	\end{align*}
	To prove the claimed equicontinuity, 
	it suffices to show that $g(r) \to 0$ as $r \to 0$. 
	But this follows from the growth properties of $\theta$ 
	by picking, e.g., $v_0:= r^{-1/2}$.
	
	Of course the masses are uniformly bounded in $\eps$ and $t$. 
	The Arzela-Ascoli theorem implies that every subsequence has a  subsequence converging uniformly in  $\big(\cM_+(\T^d), \|\cdot\|_\KR\big)$.
\end{proof}

\subsection{The boundary value problems under superlinear growth}\label{subsec: compactness2}
The last part of this section is devoted to the proof of the convergence of the minimal actions, under the assumption of superlinear growth, i.e. Theorem \ref{theorem: uniform compactness 2}. The proof is a straightforward consequence of the stronger compactness result Theorem \ref{theorem: uniform compactness} (and the general convergence result Theorem \ref{thm:main}) proved in the previous section, which ensures the stability of the boundary conditions as well. We fix $\cI = (a,b)$.

\begin{proof}[Proof of Theorem \ref{theorem: uniform compactness 2}]
	We shall prove the upper and the lower bound.
	
	\smallskip
	\noindent
	\textit{Liminf inequality}. \
	Pick any $\iota_\eps m_a^\eps \to \mu^a$, $\iota_\eps m_b^\eps \to \mu^b$ weakly in $\cM_+(\Td)$, and let $(\bfm^\eps, \bfJ^\eps) \in \cCE_\eps^\cI$ with the same boundary data such that
	\begin{align*}
		\lim_{\eps \to 0} \cA_\eps^\cI(\mathbf{m}^\eps, \mathbf{J}^\eps) = \lim_{\eps \to 0} \cMA_\eps^\cI(m_a^\eps, m_b^\eps) < \infty.
	\end{align*}
	By Theorem \ref{theorem: uniform compactness}, 
	there exists a (non-relabeled) subsequence of 	
		$\bfm^\eps$ 
	such that 
		$\| \iota_\eps m_t^\eps - \mu_t \|_\KR \to 0$,
	uniformly for $t\in \overline\cI$. 
	In particular, 
		$\mu_a =  \mu^a$, $\mu_b = \mu^b$.  
	We can then apply the lower bound of Theorem \ref{thm:main}, and conclude
	\begin{align*}
		\bMA_{\hom}^\cI(\mu^a,\mu^b) \leq \bA_{\hom}^\cI(\bfmu) \leq \liminf_{\eps \to \infty} \cMA_\eps^\cI(m_a^\eps, m_b^\eps).
	\end{align*} 
	
	\smallskip
	\noindent
	\textit{Limsup inequality}. \
	Fix $\mu^a, \mu^b \in \cM_+(\Td)$ such that $\bMA_{\hom}^\cI(\mu^a,\mu^b)<\infty$.  
	By the definition of $\bMA_{\hom}^\cI$ and the lower semicontinuity of $\bA_\hom$ (Lemma \ref{lemma: action lsc}), there exists $\bfmu \in \cM_+(\cI \times \Td)$ with $\bA_\hom^\cI(\mathbf{\mu}) = \bMA_{\hom}^\cI(\mu^a,\mu^b)$ and $\mu_a = \mu^a, \mu_b= \mu^b$.
	
	We can then apply Theorem \ref{thm:main} and find a recovery sequence $(\mathbf{m}^\eps, \mathbf{J}^\eps) \in \cCE_\eps^\cI$ such that $\iota_\eps \bfm^\eps \to \bfmu$ weakly and
	\begin{align*}
		\limsup_{\eps \to 0} \cA_\eps^\cI(\mathbf{m}^\eps, \mathbf{J}^\eps) \leq \bA_\hom^\cI(\bfmu) = \bMA_{\hom}^\cI(\mu^a,\mu^b).
	\end{align*}
	
	By the improved compactness result Theorem \ref{theorem: uniform compactness}, $\iota_\eps m^\eps_t \to \mu_t$ in $\KR(\Td)$ for every $t\in \overline{\cI}$, in particular for $t=a,b$. This allows us to conclude
	\begin{align*}
		\limsup_{\eps \to 0} \cMA_\eps^\cI(m_a^\eps, m_b^\eps) \leq \bMA_{\hom}^\cI(\mu^a,\mu^b) 
		, \tand
		\iota_\eps m_i^\eps \to \mu^i \, \text{weakly}
	\end{align*}
	for $i=a,b$, which is sought recovery sequence for $\bMA_{\hom}^\cI(\mu^a,\mu^b)$.
\end{proof}

\begin{remark}
		\label{rem:linear_minimal}
	It is instructive to see that under the simple linear growth condition \eqref{ass:F}, the above written proof cannot be carried out. Indeed, by the lack of compactness in $W^{1,1}(\cI; \cM_+(\Td))$ (but rather only in $\BV$ by Theorem \ref{thm: compactness}), we are not able to ensure stability at the level of the initial data, i.e. in general, $\mu_a \neq \mu^a$ (and similarly for $t =b$).
\end{remark}

\section{Proof of the lower bound}\label{sec: lower}

In this section we present the proof of the lower bound in our main result, Theorem \ref{thm:main}.  
The proof relies on two key ingredients.
The first one is a partial regularisation result for discrete measures of bounded action, which is stated in Proposition \ref{prop:regularisation_discrete} and proved in Section \ref{sec:disc-reg} below.
The second ingredient is a lower bound of the energy under partial regularity conditions on the involved measures (Proposition \ref{prop:energy_reg}). 
The proof of the lower bound in Theorem \ref{thm:main}, which combines both ingredients, is given right before Section \ref{sec:disc-reg}.

First we state the regularisation result.
Recall the Kantorovich--Rubinstein norm $\|\cdot\|_{\KR}$. see Appendix \ref{sec:KR}.

\begin{proposition}[Discrete Regularisation]
	\label{prop:regularisation_discrete}
Fix $\eps < \frac{1}{2R_0}$ 
and let	
	$(\bfm, \bfJ) \in \cCE_\eps^\cI$ 
be a solution to the discrete continuity equation
satisfying
\begin{align*}
	M := m_0(\cX_\eps) < \infty 
		\tand
	A := \cA_\eps^\cI(\bfm, \bfJ) < \infty.
\end{align*}
Then, for any $\eta > 0$ 
there exists an interval $\cI^\eta \subset \cI:=(0,T)$ with $|\cI \setminus \cI^\eta| \leq \eta$ and a solution 
	$(\tilde \bfm, \tilde \bfJ) \in \cCE_\eps^{\cI^\eta}$ 
such that:
\begin{enumerate}[(i)]
\item the following approximation properties hold:
\begin{subequations}
	\label{eq:approx-bounds}
\begin{align}
	& \textrm{(measure approximation)} & 
	\| \iota_\eps (\tilde \bfm - \bfm) 
		  \|_{\KR(\overline{\cI^\eta} \times \Td)} 
		& 	\leq \eta,
	  \\
%%%%%%%%%%%%%%%%%%%%%%%%%%%%%%%%%%%%%	  
	  & \textrm{(action approximation)} &
	  \label{eq:reg-energy}
	   \cA_\eps^{\cI^\eta}(\tilde \bfm, \tilde \bfJ)
		 \leq 
	  \cA_\eps^\cI(\bfm, \bfJ)
	  & + \eta.
\end{align}
\end{subequations}
\item 
	the following regularity properties hold, uniformly for any $t\in \cI^\eta$ and any $z \in \T_\eps^d$:
\begin{subequations}
	\label{eq:reg-bounds}
\begin{align}
	\label{eq:reg-bound}
	& \textrm{(boundedness)} &
			\big\| 
				 \tilde m_t
			\big\|_{\ell^\infty(\cX_\eps)}
		+	
			\eps
			\big\| 
				 \tilde J_t
			\big\|_{\ell^\infty(\cE_\eps)}
		 & \leq C_B \eps^{d}, 
		  \\
%%%%%%%%%%%%%%%%%%%%%%%%%%%%%%%%%%%%%		  
	\label{eq:reg-time}
	& \textrm{(time-reg.)} &
			\big\| 
				\dive \tilde J_t
			\big\|_{\ell^\infty(\cX_\eps)} 
		 & \leq C_T \eps^{d}, 
\\
%%%%%%%%%%%%%%%%%%%%%%%%%%%%%%%%%%%%%
\label{eq:reg-space}
	& \textrm{(space-reg.)} &
				\big\|  \sigma_\eps^z \tilde m_t
					  - \tilde m_t 
				\big\|_{\ell^\infty(\cX_\eps)}
			+ \eps
				\big\|  \sigma_\eps^z  \tilde J_t		
		  			 - \tilde J_t 
	\big\|_{\ell^\infty(\cE_\eps)}
	& \leq C_S |z| \eps^{d+1},
	\\
%%%%%%%%%%%%%%%%%%%%%%%%%%%%%%%%%%%%%	
	\label{eq:interior-domain}
	& \textrm{(domain reg.)} &	
	\bigg(\frac{\tau_\eps^z \tilde m_t}
		{\eps^d}
	,
	\frac{\tau_\eps^z \tilde J_t}
		   {\eps^{d-1}}
	\bigg)
	 & \in K. 
\end{align}
\end{subequations}
\noindent The constants $C_B, C_T, C_S < \infty$
	and the compact set $K \subseteq \Dom(F)^\circ$ 
	depend on $\eta$, $M$ and $A$, but not on $\eps$.
\end{enumerate}
\end{proposition}

\begin{remark}
	The $\ell^\infty$-bounds in \eqref{eq:reg-bound} are explicitly stated for the sake of clarity, although they are implied by the compactness of the set $K$ in \eqref{eq:interior-domain}.
	
	Since $(\tilde \bfm, \tilde \bfJ)\in \cCE_\eps^{\cI^\eta}$, inequality \eqref{eq:reg-time} in effect bounds $\big\| 
\partial_t \tilde m_t
	\big\|_{\ell^\infty(\cX_\eps)} 
\leq C_T \eps^{d}$.
\end{remark}

In the next result, we start by showing how to construct $\Z^d$-periodic solutions to the static continuity equation by superposition of unit fluxes.  
Additionally, we can build these solutions with vanishing effective flux and ensure good $\ell^\infty$-bounds.

\begin{lemma}
	[Periodic solutions to the divergence equation]
	\label{lemma:bounds_divergence_eq 2}
	Let $g : \cX \to \R$ be a $\Z^d$-periodic function with
	$\sum_{x \in \XQ} g(x) = 0$. 
	There exists a $\Z^d$-periodic discrete vector field $J : \cE \to \R$ satisfying
	\begin{align*}
		\dive J = g, \quad
		\Eff(J) = 0, \tand
		\| J \|_{\ell_\infty(\EQ)} 
		\leq 
		\tfrac12 \| g \|_{\ell_1(\XQ)}.
	\end{align*}
	
\end{lemma}

\begin{proof} 
	For any $v, w \in \V$, fix a simple path $P^{vw}$ in $(\cX, \cE)$ connecting $(0,v)$ and $(0,w)$.
	Let 
	$\tilde J_{vw} := \tilde J_{P^{vw}}$ 
	be the associated periodic unit flux defined in \eqref{eq:periodic-unit-flux}.
	Since $\sum_{v \in \V} g(0,v) = 0$, we can pick a coupling $\Gamma$ between the negative part and the positive part of $g$. 
	More precisely, we may pick a function $\Gamma : \V \times \V \to \R_+$ with 
	$\sum_{v,w \in \V} 
	\Gamma(v,w) 
	= \|g\|_{\ell_1(\XQ)}
	$
	such that 
	\begin{align*}
		\sum_{w \in \V} \Gamma_{vw} = g_-(0,v)	
		\quad 
		\text{for } v \in \V,
		\tand
		\sum_{v \in \V} \Gamma_{vw} = g_+(0,w)	
		\quad 
		\text{for } w \in \V.
	\end{align*}
	We then define
	\begin{align*}
		J := \sum_{v,w \in \V}
		\Gamma_{vw} \tilde J_{vw}.
	\end{align*}
	It is straightforward to verify using Lemma \ref{lem:J_Pq} that $J$ has the three desired properties.
\end{proof}

The following result states the desired relation between the functionals $\cF_\eps$ and $\bF_{\hom}$ under suitable regularity conditions for the measures involved.
These regularity conditions are consistent with the regularity properties obtained in Proposition \ref{prop:regularisation_discrete}.  

\begin{proposition}[Energy lower bound for regular measures]
\label{prop:energy_reg}	
	Let $C_B, C_T, C_S < \infty$ and let $K \subseteq \Dom(F)^\circ$ be a compact set.
	There exists a threshold $\eps_0 > 0$ 
	and a constant $C < \infty$ such that the following implication holds
	for any $\eps < \eps_0$:
	if 
	$m \in \Meps$ 
	and $J \in \Mdeps$
	satisfy the regularity properties \eqref{eq:reg-bound}-\eqref{eq:interior-domain}
	then we have the energy bound
	\begin{align*}
		\bF_\hom(\iota_\eps m, \iota_\eps J)
		\leq
			\cF_\eps(m, J)
			+ C \eps.
	\end{align*}
\end{proposition}

\begin{proof}	
Recall from \eqref{eq:density_embedded_flux} that 
	$\iota_\eps m = \rho \Leb^d$ and
	$\iota_\eps J = j \Leb^d$,
where, for $\bar z \in \Z_\eps^d$ 
and $u \in Q_\eps^{\bar z}$,
\begin{align*}
	\rho(u)
	:= \eps^{-d} 
		\sum_{\substack{x \in \cX_\eps\\
		x_\sz = \bar z}} 
			m(x)
\tand
	j(u)
	:=  
	\frac{1}{2\eps^{d-1}}
		\sum_{\substack{(x,y) \in \cE_\eps\\
			  x_\sz = \bar z}} 
			J_u(x,y) 
			\big( 
				y_\sz - x_\sz 
			\big),
\end{align*} 
where $J_u(x, y)$ is a convex combination of $\big\{
		J\big( T_\eps^z x, T_\eps^z y \big)
\big\}_{z \in \Z_\eps^d}$, 
i.e., 	
\begin{align*}
	J_u(x, y) 
	& = \sum_{z \in \Z_\eps^d}
		\lambda_u^{\eps,z}(x,y)
			J\big( T_\eps^z x, T_\eps^z y \big),
\end{align*}
where $\lambda_u^{\eps, {\bar z}}(x,y) \geq 0$,
	$\sum_{z \in \Z_\eps^d} 
		\lambda_u^{\eps,z}(x,y) = 1$,
	and $\lambda_u^{\eps, z}(x,y) = 0$ whenever
	$|z| > R_0$.

\smallskip
\emph{Step 1. Construction of a representative.} \ 
Fix $\bar z \in \Z_\eps^d$ 
and $u \in Q_\eps^{\bar z}$. 
Our first goal is to construct a representative
\begin{align*}
	\bigg(
		\frac{\hat m_u}{\eps^d}, 
		\frac{\hat J_u}{\eps^{d-1}}
	\bigg)
		\in
	\Rep\big( 
		\rho(u), j(u)
	\big).
\end{align*}
For this purpose we define candidates
	$\hat m_u \in \R_+^{\cX}$
and	
	$\tilde J_u \in \R_a^{\cE}$
as follows. 
We take the values of 
	$m$ and $J_u$ 
in the $\eps$-cube at $\bar z$, 
and insert these values at every cube in $(\cX, \cE)$, so that the result is $\Z^d$-periodic.
In formulae:
\begin{align*}
	\hat m_u(z,v) 
		&:= 
		 m(\eps \bar z, v)
		&& \text{for } (z, v) \in \cX
\\
	\tilde J_u\Big( (z,v), (z',v') \Big)
		&:= 
		J_u\Big( (\eps \bar z,v), (\eps (\bar z+ z' - z),v') \Big)
		&& \text{for } \Big( (z,v), (z',v') \Big) \in \cE.	
\end{align*}

see Figure \ref{fig:cand_lb}.

\begin{figure}[h]
	\begin{subfigure}{.5\textwidth}
		\centering
		\includegraphics[scale=1.5]{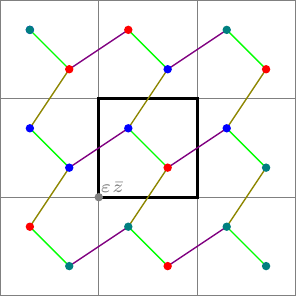}
	\end{subfigure}%
	\begin{subfigure}{.5\textwidth}
		\centering
		\includegraphics[scale=1.5]{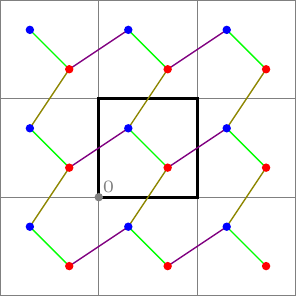}
	\end{subfigure}%
	\caption{On the left, using different colors for different values, the measures $m$ and $J_u$. On the right, the corresponding $\hat m_u$ and $\tilde J_u$, for $u \in Q_\eps^{\bar z}$.}
	\label{fig:cand_lb}
\end{figure}

We emphasise that the right-hand side 
does not depend on $z$, 
hence $m_u$ and $\tilde J_u$
are $\Z^d$-periodic.
Our construction also ensures that
\begin{align*}
	\eps^{-d}
	\sum_{x \in \XQ} 
		\hat m_u(x)
	= \rho(u),
\end{align*}
hence $\eps^{-d} \hat m_u \in \Rep\big( \rho(u) \big)$.
However, the vector field
$\eps^{-(d-1)}\tilde J_u$ does (in general) not belong to $\Rep\big(j(u)\big)$: 
indeed, while $\tilde J_u$ has the desired effective flux 
	(i.e., $\Eff\big(\eps^{-(d-1)}\tilde  J_u\big) 
		= j(u)$),
$\tilde J_u$ is not (in general) divergence-free.

To remedy this issue, we introduce a \emph{corrector field} 
$\bar J_u$, i.e., an anti-symmetric and $\Z^d$-periodic function $ \bar J_u : \cE \to \R$ satisfying
\begin{align}
\label{eq:corrector-prop}
	 \dive \bar J_u 
	 	= -\dive \tilde J_u, \quad
	 \Eff(\bar J_u ) 
	 	= 0, \tand
	 \big\| \bar J_u \big\|_{\ell^\infty(\EQ)} 
		 \leq 
	\tfrac12 \big\| \dive \tilde J_u \big\|_{\ell^1(\XQ)}.
\end{align}
The existence of such a vector field 
is guaranteed 
by Lemma \ref{lemma:bounds_divergence_eq 2}.
It immediately follows that 
	$\hat J_u
	:= \tilde J_u 
	+  \bar J_u$ 
satisfies	
$\dive \hat J_u = 0$ 
and 
$\Eff\big(\eps^{-(d-1)}\hat J_u\big) = j(u)$, 
thus
\begin{align*}
	\frac{\hat J_u}{\eps^{d-1}}
	:= \frac{\tilde J_u 
	+  \bar J_u}
	{{\eps^{d-1}}}
	\in \Rep\big(j_u \big).
\end{align*}

\smallskip
\emph{Step 2. Density comparison.} \ 
We will now use the regularity assumptions \eqref{eq:reg-bound}-\eqref{eq:interior-domain} to show that the representative 
	$(\hat m_u, \hat J_u)$ 
is not too different from the shifted density
	$(\tau_{\bar z} m,
	\tau_{\bar z} J)$.
Indeed, for $x = (z, v) \in \cX$ 
with $|z| \leq R_0$ 
we obtain using \eqref{eq:reg-space},
\begin{align}
\label{eq:m-compare}
	| \tau_\eps^{\bar z} m(x)
		- \hat m_u(x) |
	= \big| m\big(\eps(\bar z + z), v \big)
	 	 - m\big(\eps \bar z, v \big)
	  \big|
	\leq C_S \eps^{d+1} |z|.
\end{align}

Let us now turn to the momentum field.
For 
	$(x, y) = \big( (z,v), (z',v') \big) \in \cE$ 
with 
	$|z|, |z'| \leq R_1$, 
we have, using \eqref{eq:reg-space},
\begin{align*}
	& 	\big| 
			\tau_\eps^{\bar z} J(x,y)
			- \tilde J_u(x,y) 
		\big|
	\\& = \Big| J\Big( 
			\big(\eps (\bar z + z),v\big),
			\big(\eps (\bar z + z'),v'\big) \Big)
	 	 - J_u\Big( 
			\big(\eps \bar z,v\big), 
			\big(\eps (\bar z + z' - z),v'\big) 
			\Big)
	  \Big|
		\\& = \bigg| 
		\sum_{\tilde z \in \Z_\eps^d} 
		\lambda_u^{\eps,\tilde z}(x,y)
		\bigg\{ 
			J\Big( 
				\big(\eps (\bar z + z)    , v  \big),
				\big(\eps (\bar z + z'), v' \big) 
			\Big)
			\\&  \qquad\qquad\qquad\qquad
		- 	J\Big( 
				\big( \eps (\bar z + \tilde z    ), v  \big), 
				\big( \eps (\bar z + \tilde z + z' - z), v' \big) 
			\Big)
		\bigg\}
  \bigg|
	\\& \leq C_S \eps^{d} |z - \tilde z|
		\leq R_0 C_S \eps^{d} .
\end{align*}
Moreover, using \eqref{eq:corrector-prop}, \eqref{eq:reg-space}, and \eqref{eq:reg-time}, we obtain
\begin{align*}
|\bar J_u(x,y) |	
	&	\leq  \tfrac12 
		\| \dive 
			\tilde J_u 
		\|_{\ell^1(\XQ)}
	\leq C_T 
		\Big(
			\| \dive J \|_{\ell^\infty(\cE_\eps)}	 + \eps^d	
		\Big)
	 \leq C \eps^d,
\end{align*}
for some $C < \infty$ not depending on $\eps$.
Combining these bounds we obtain
\begin{align}
\label{eq:J-compare}	
	| \tau_\eps^{\bar z} J(x,y)
	- \hat J_u(x,y) |
& \leq | \tau_\eps^{\bar z} J(x,y)
	- \tilde J_u(x,y) |
+ |\bar J_u(x,y) |
	\leq C \eps^d.
\end{align}

\smallskip
\emph{Step 3. Energy comparison.} \ 
Since
$\displaystyle 
\bigg(\frac{\tau_\eps^{\bar z} m}
{\eps^d}
,
\frac{\tau_\eps^{\bar z} J}
   {\eps^{d-1}}
\bigg)
\in K$
by assumption, it follows from \eqref{eq:m-compare} and \eqref{eq:J-compare} that 
$\displaystyle 
\bigg(\frac{\hat m_u}
{\eps^d}
,
\frac{\hat J_u}
   {\eps^{d-1}}
\bigg)
\in K'$
for $\eps > 0$ sufficiently small. 
Here $K$ is a compact set, possibly slightly larger than $K$, contained in $\Dom(\cF)^\circ$.

Since $F$ is convex, it is Lipschitz continuous on compact subsets in the interior of its domain. In particular, it is Lipschitz continuous on $K'$.
Therefore, there exists a constant $C_L < \infty$ depending on $\cF$ and $K'$ such that
\begin{align*}
	\cF\bigg(\frac{\tau_\eps^{\bar z} m}
	{\eps^d}
	,
	\frac{\tau_\eps^{\bar z} J}
	   {\eps^{d-1}}
	\bigg)
	& \geq  
	\cF\bigg(\frac{\hat m_u}
	{\eps^d}
	,
	\frac{\hat J_u}
	   {\eps^{d-1}}
	\bigg)
	- C_L 
		\bigg(
			\Big\| \frac{\tau_\eps^{\bar z} m 
			- \hat m_u} 
			{\eps^d} \Big\|_{\ell^\infty_{R_0}(\cX)}
		+ 	\Big\| \frac{\tau_\eps^{\bar z} J 
				- \hat J_u 		}
				{\eps^{d-1}}\Big\|_{\ell^\infty_{R_0}(\cE)} 
		\bigg)
		\\ & \geq 
		\cF\bigg(\frac{\hat m_u}
		{\eps^d}
		,
		\frac{\hat J_u}
		   {\eps^{d-1}}
		\bigg)
		- C \eps
		\\ & \geq 
		f_\hom\big( \rho(u), j(u) \big)
		- C \eps,
\end{align*}
with $C < \infty$ depending on $C_L$, $C_S$, $C_T$, and $R_1$, but not on $\eps$.
Here, the subscript $R_0$ in $\ell^\infty_{R_0}(\cE)$ and $\ell^\infty_{R_0}(\cX)$ indicates that only elements with $|x_\sz| \leq R_1$ are considered.

Integration over $Q_\eps^{\bar z}$ followed by summation over $\bar z \in \Z_\eps^d$ yields
\begin{align*}
	\cF_\eps(m, J)
	& = \eps^d
		\sum_{\bar z \in \Z_\eps^d}
		\cF\bigg(\frac{\tau_\eps^{\bar z} m}
		{\eps^d}
		,
		\frac{\tau_\eps^{\bar z} J}
		   {\eps^{d-1}}
		\bigg)	
		\geq 
		\sum_{\bar z \in \Z_\eps^d}
			\int_{Q_\eps^{\bar z}}
				\Big(
					f_\hom\big( \rho(u), j(u) \big)
						- C \eps
				\Big)	
			\dd u
		\\& = 	
			\int_{\T^d}
					f_\hom\big( \rho(u), j(u) \big)	
			\dd u
			- C \eps
		= \bF_\hom(\iota_\eps m, \iota_\eps J) 
			- C \eps,
\end{align*}
which is the desired result.
\end{proof}

We are now ready to give the proof of the lower bound in our main result, Theorem \ref{thm:main}.

\begin{proof}[Proof of Theorem \ref{thm:main} (lower bound)]
Let $\bfmu \in \cM_+\big(\cI \times \T^d\big)$ and let 
	$(m_t^\eps)_{t \in \cI} 
		\subseteq 
	\R_+^{\cX_\eps}$
be such that the induced measures 
	$\bfm^\eps \in  \cM_+\big(\cI \times \cX_\eps \big)$
defined by
	$\ddd \bfm^\eps(t,x) = \ddd m^\eps_t(x) \dd t$
satisfy 	
	$\iota_\eps \bfm^\eps \to \bfmu$ vaguely in $\cM_+(\cI \times \Td)$
as $\eps \to 0$.
Observe that
\begin{align*}
	M := \sup_{\eps > 0}
		\bfm^\eps
		\big(
			\cI \times \cX_\eps
		\big)
	< \infty.
\end{align*}
Without loss of generality, we may assume that 
\begin{align*}
	A := 
		\sup_{\eps > 0}
			\cA_\eps(\bfm^\eps)
		 <
		\infty.
\end{align*}

\smallskip
\emph{Step 1 (Regularisation)}:
Fix $\eta > 0$. 
Let $(J_t^\eps)_{t \in \cI} \subseteq \Mdeps$ be an approximately optimal discrete vector field, i.e., 
\begin{align}\label{eq:almost-opt}
	(\bfm^\eps, \bfJ^\eps) \in\cCE_\eps^\cI
		\tand
	\cA_\eps(\bfm^\eps, \bfJ^\eps)
		\leq
	\cA_\eps(\bfm^\eps)  
		+ \eta.
\end{align}
Using Proposition \ref{prop:regularisation_discrete} we take an interval $\cI^\eta \subset \cI:=(0,T)$, $|\cI\setminus \cI^\eta|\leq \eta$ and an approximating pair $(\tilde \bfm{}^\eps, \tilde \bfJ{}^\eps) \in\cCE_\eps^{I_\eta}$ satisfying
\begin{align}\label{eq:almost-opt2}
	\| \iota_\eps (\tilde \bfm{}^\eps - \bfm^\eps) 
	  \|_{\KR(\overline{\cI^\eta} \times \Td)} \leq \eta 
	  \tand
	  \cA_\eps^{\cI^\eta}(\tilde \bfm{}^\eps, \tilde \bfJ{}^\eps)
	  \leq 
	  \cA_\eps(\bfm^\eps, \bfJ^\eps)
  + \eta,
\end{align}
together with the regularity properties \eqref{eq:reg-bounds}
for some constants $C_B, C_T, C_S < \infty$
and a compact set $K \subseteq \Dom(F)^\circ$ 
depending on $\eta$, but not on $\eps$.
By virtue of these regularity properties, we may apply Proposition \ref{prop:energy_reg} to $(\tilde \bfm{}^\eps, \tilde \bfJ{}^\eps)$. 
This yields
\begin{align}
\label{eq:reg-energ}
	\bA_\hom^{\cI^\eta}(\iota_\eps \tilde \bfm{}^\eps, 
			 \iota_\eps \tilde \bfJ{}^\eps)
	= 
	\int_{\cI^\eta}
		\bF_\hom(\iota_\eps \tilde m_t^\eps, 
			     \iota_\eps \tilde J_t^\eps)
	\dd t 
	\leq
	\int_{\cI^\eta}
		\cF_\eps(\tilde m_t^\eps, \tilde J_t^\eps)
	\dd t 
		+ C \eps,
\end{align}
with $C < \infty$ depending on $\eta$, but not on $\eps$.

\smallskip
\emph{Step 2 (Limit passage $\eps \to 0$)}:
It follows by definition of the Kantorovich--Rubinstein norm that 
\begin{align*}
	\sup_\eps 
		\iota_\eps \tilde\bfm^\eps
		\big( 
			\overline{\cI^\eta} \times \T^d
		\big)
	&\leq 
	\sup_\eps 
		\bigg(
				\iota_\eps \bfm^\eps
			\big( 
				\cI \times \T^d
			\big)
			+
			\| \iota_\eps 
					(\tilde \bfm{}^\eps
					 - 	\bfm^\eps) 
			\|_{\KR(\overline{\cI^\eta} \times \Td)}
		\bigg) \\
	&\leq 
	M + \eta.
\end{align*}
It follows from the growth condition \eqref{eq: growth} and \eqref{eq:almost-opt2} that
\begin{equation}
	\begin{aligned}		
	\label{eq:nu-bound}
	\sup_\eps 
		\big|\iota_\eps \tilde \bfJ{}^\eps\big|
		\big( 
			\overline{\cI^\eta} \times \T^d
		\big)
& \lesssim  \sup_\eps  
		\int_{\cI^\eta}
			\eps 	
			\|
				\tilde J_t^\eps 
			\|_{\ell^1(\cE_\eps)}
		\dd t
\\&	\lesssim 
	\sup_\eps  
	\int_{\cI^\eta}
		\bigg( 
			1 
			+ \| \tilde m_t^\eps \|_{\ell^1(\cX_\eps)} 
			+ \cF_\eps(\tilde m_t^\eps,
					\tilde J_t^\eps)
		\bigg)
	\dd t
	\\ & 
	\leq 
	\sup_\eps 
	\bigg(
	T 
	+ 	
	\iota_\eps \tilde\bfm^\eps
	\big( 
		\cI^\eta \times \T^d
	\big)
	+ 
	\cA_\eps^{\cI^\eta}(\tilde \bfm{}^\eps, 
			 \tilde \bfJ{}^\eps
			)
	\bigg)
	\\ & 
	\leq T + (M + \eta) + (A + 2 \eta).
	\end{aligned}
\end{equation}
Therefore, there exist measures
	$\bfmu_\eta \in \cM_+\big(\overline{\cI^\eta} \times \T^d\big)$
and 
	$\bfnu_\eta \in \cM^d\big(\overline{\cI^\eta} \times \T^d\big)$
and convergent subsequences satisfying 
\begin{align}
	\label{eq:weaks}
	\iota_\eps \tilde \bfm{}^\eps
		\to
	\bfmu_\eta  
	\ \text{and} \
		\iota_\eps \tilde \bfJ{}^\eps
		\to
	\bfnu_\eta 
	\ \text{weakly in } \cM_+(\overline{\cI^\eta} \times \Td) \ \text{and} \   \cM^d(\overline{\cI^\eta} \times \Td)  \ \text{as } \eps \to 0.
\end{align}
The vague lower semicontinuity of the limiting functional (see Lemma \ref{lemma: action lsc}), combined with \eqref{eq:almost-opt}, \eqref{eq:almost-opt2}, and \eqref{eq:reg-energ} thus yields
\begin{align}
	\label{eq:Ahom-bound}
	\bA_\hom^{\cI^\eta}
		(\bfmu_\eta, \bfnu_\eta)
	&\leq 
		\liminf_{\eps \to 0}
			\bA_\hom^{\cI^\eta}(\iota_\eps \tilde \bfm{}^\eps, 
					\iota_\eps \tilde \bfJ{}^\eps)
	\leq \liminf_{\eps \to 0} 
			\cA_\eps(\bfm^\eps)
		+ 2\eta.
\end{align}

\smallskip
\emph{Step 3 (Limit passage $\eta \to 0$)}:
Let 
	$\phi \in \Lip_1  
		\big(
			\overline{\cI^\eta} \times \Td
		\big)
	$, $\| \phi \|_\infty \leq 1 $.
For brevity, write $\ip{\phi,\bfmu} = \int_{\cI^\eta \times \T^d} \phi \dd \bfmu$.
Since from \eqref{eq:weaks}
	$\iota_\eps \bfm^\eps
	\to
\bfmu$ and
$\iota_\eps \tilde \bfm{}^\eps
	\to
\bfmu_\eta$ weakly, 
and
$\| \iota_\eps (\tilde \bfm{}^\eps - \bfm^\eps) 
\|_{\KR(\overline{\cI^\eta} \times \Td)} \leq \eta 
$ 
we obtain
\begin{align*}
	\ip{ \phi,  \bfmu_\eta - \bfmu }
	&\leq 
	\limsup_{\eps \to 0}
	\Big(
		\big|\bip{ \phi,  \bfmu_\eta - \iota_\eps \tilde \bfm{}^\eps }
		\big|
		+
		\big|\bip{ \phi,  \iota_\eps 
			(\tilde \bfm{}^\eps -\bfm^\eps) }
			\big|
		+
		\big|\bip{ \phi,  \iota_\eps \bfm^\eps- \bfmu }\big|
	\Big)
	\\& \leq
	0 + \eta + 0.
\end{align*}
It follows that $\|  \bfmu_\eta - \bfmu\|_{\KR(\overline{\cI^\eta} \times \Td)} \leq 2\eta$, which together with $|\cI \setminus \cI^\eta|\leq \eta$ implies $\bfmu_\eta \to \bfmu \in \cM_+(\cI \times \Td)$ vaguely as $\eta \to 0$.

Furthermore, \eqref{eq:nu-bound} implies that
$\sup_{\eta}\big|\bfnu^\eta\big|
\big( 
	\cI^\eta \times \T^d
\big) < \infty$.
Therefore, we may extract a subsequence so that
$\bfnu_\eta \to \bfnu$ vaguely in $\cM^d(\cI \times \Td)$ as $\eta \to 0$.
It thus follows from \eqref{eq:Ahom-bound} and the joint vague-lower semicontinuity of $\bA_\hom$ (see Lemma \ref{lemma: action lsc}) that
\begin{align*}
	\bA_\hom
	(\bfmu, \bfnu)
\leq \liminf_{\eps \to 0} 
		\cA_\eps(\bfm^\eps).
\end{align*}

To conclude the desired estimate 
$
	\bA_\hom(\bfmu)
		\leq 
	\liminf_{\eps \to 0} 
		\cA_\eps(\bfm^\eps)$,
it remains to show that $(\bfmu, \bfnu)$ solves the continuity equation. 
To show this, we first note that 
	$(\iota_\eps \tilde \bfm{}^\eps,\iota_\eps \tilde \bfJ{}^\eps) \in \bCE^{\cI^\eta}$ 
in view of Lemma \ref{lem:conteq-embed}.
It then follows from the weak convergence in \eqref{eq:weaks} that $(\bfmu_\eta, \bfnu_\eta) \in \bCE^{\cI^\eta}$.
Since $\bfmu_\eta \to \bfmu$,  $\bfnu_\eta \to \bfnu$ vaguely, and $|\cI-\cI^\eta|\leq \eta$ it holds $(\bfmu, \bfnu) \in \bCE^\cI$, which completes the proof.
\end{proof}

\subsection{Proof of the discrete regularisation result}

\label{sec:disc-reg}

This section is devoted to the proof of main discrete regularisation result, Proposition \ref{prop:regularisation_discrete}.

The regularised approximations are constructed by a three-fold regularisation: in time, space, and energy. 
Let us now describe the relevant operators.

\subsubsection{Energy regularisation}
First we embed 
$m^\circ$ and $J^\circ$ 
into the graph 
$(\cX_\eps, \cE_\eps)$.
We thus define
$m^\circ_\eps \in \Meps$
and
$J^\circ_\eps 
\in 
\Mdeps$
by
\begin{align*}
	m^\circ_\eps(\eps z, v) 
	:=
	\eps^d m^\circ(0,v)
	\qquad
	J^\circ_\eps(\eps z, v) 
	:=
	\eps^{d-1} J^\circ(0,v).
\end{align*}
It follows that 
$(m^\circ_\eps , J^\circ_\eps ) 
\in 
\Dom(\cF_\eps)^\circ$ (by continuity of $\tau_\eps^z$, $z \in \Z_\eps^d$) and
\begin{align*}
	\cF_\eps(m^\circ_\eps , J^\circ_\eps ) 
	= 
	F(m^\circ , J^\circ).
\end{align*}	
We then consider the energy regularisation operators defined by
\begin{align*}
	R_\delta : 
	\Meps & \to \Meps,
	&
	R_\delta m 
	&:= (1-\delta) m + \delta m_\eps^0,
	\\
	R_\delta : 
	\Mdeps & \to \Mdeps,
	&
	R_\delta J 
	&:= (1-\delta) J + \delta J_\eps^0.
\end{align*}

\begin{lemma}[Energy regularisation]
	\label{lem:reg-energy}	
	Let $\delta \in (0,1)$.
	The following inequalities hold
	for any 
	$\eps < \frac1{2R_0}$,
	$m \in \R_+^{\cX_\eps}$, and
	$J \in \R_a^{\cE_\eps}$:
	\begin{align*}
		\cF_\eps(
		R_\delta m, 
		R_\delta J)
		&\leq 
		(1 - \delta)
		\cF_\eps(m, J)
		+ \delta 
		\cF_\eps(m_\eps^\circ, J_\eps^\circ),\\
		%%%%%%%%%%%%%%%%			
		\| R_\delta m \|_{\ell^\infty(\cX_\eps)}
		&\leq	
		(1-\delta) 
		\| m \|_{\ell^\infty(\cX_\eps)}
		+ \delta \eps^d  	
		\| m^\circ \|_{\ell^\infty(\cX)},\\
		%%%%%%%%%%%%%%%%			
		\| R_\delta J \|_{\ell^\infty(\cE_\eps)}
		&\leq	
		(1-\delta) 
		\| J \|_{\ell^\infty(\cE_\eps)}
		+ \delta \eps^{d-1}  
		\| J^\circ \|_{\ell^\infty(\cE)}.
	\end{align*}	
\end{lemma}

\begin{proof}
	The proof is straightforward consequence of the convexity of $F$ and the periodicity of $m^\circ$ and $J^\circ$.
\end{proof}

\subsubsection{Space regularisation}

Our space regularisation is a convolution in the $z$-variable with the discretised heat kernel.
It is of crucial importance that the regularisation is performed in the $\sz$-variable only. 
Smoothness in the $\sa$-variable is not expected.

For $\lambda > 0$ and $x \in \Td$, 
let $h_\lambda(x)$ be the heat kernel on $\Td$.
We consider the discrete version
\begin{align*}
	H_\lambda^\eps : \Z_\eps^d \to \R,  
	\qquad
	H_\lambda^\eps\big([z]\big)
	:= 
	\int_{Q_\eps^z} h_\lambda(x) \dd x,
\end{align*}
where the integration ranges over the cube 
$Q_\eps^z := \eps z + [0,\eps)^d \subseteq \Td$. 
Using the boundedness and Lipschitz properties of $h_\delta$, we infer that for $z \in \Z_\eps^d$,
\begin{align}	\label{eq:discrete_kernel_bounds}
	\inf_{\Z_\eps^d} H_\lambda^\eps
	& \geq c_\lambda \eps^d,
	&
	\| 
	H_\lambda^\eps 
	\|_{\ell^\infty(\Z_\eps^d)} 
	& \leq C_\lambda \eps^d, 
	\\
	\label{eq:discrete_kernel_bounds_2}
	\| H_\lambda^\eps \|_{\ell^1(\Z_\eps^d)}
	& = 1,
	&
	\big\| 	  H_\lambda^\eps(\cdot + \eps z) 
	- H_\lambda^\eps 
	\big\|_{\ell^\infty(\Z_\eps^d)} 
	& \leq C_\lambda \eps^{d+1} |z|  
\end{align}
for some non-negative constant $ C_\lambda < \infty$ 
depending only on $\lambda > 0$. 
We then define
\begin{align*}
	S_\lambda : 
		\R_+^{\cX_\eps} & \to \R_+^{\cX_\eps},
	&
	S_\lambda m 
	&:= \sum_{z \in \Z^d_\eps} 
	H_\lambda^\eps(z) \sigma_\eps^z m,
	\\
	S_\lambda : 
	\R_a^{\cE_\eps} & \to \R_a^{\cE_\eps},
	&
	S_\lambda J 
	&:= \sum_{z \in \Z^d_\eps} 
	H_\lambda^\eps(z) \sigma_\eps^z J,
\end{align*} 
where $\sigma_\eps^z$ is defined in \eqref{eq:def_sigma}.

\begin{lemma}[Regularisation in space]
	\label{lem:reg-space}	
	Let $\lambda > 0$. 
	There exist constants
	$c_\lambda > 0$ and $C_\lambda < \infty$
	such that the following estimates hold, 
	for any 
	$\eps < \frac1{2R_0}$,
	$m \in \R_+^{\cX_\eps}$,
	$J \in \cM^d(\cE_\eps)$, and
	$z \in \Z_\eps^d$: 
	\begin{enumerate}[(i)]
		\item Energy bound: $\displaystyle
		\cF_\eps(
		S_\lambda m, 
		S_\lambda J)
		\leq 
		\cF_\eps(m, J).$
		\item Gain of integrability:
		\begin{align*}
			\| S_\lambda m \|_{\ell^\infty(\cX_\eps)}
			\leq C_\lambda \eps^d
			\| m \|_{\ell^1(\cX_\eps)}  
			\tand
			\| S_\lambda J \|_{\ell^\infty(\cE_\eps)}
			\leq C_\lambda  \eps^d
			\| J \|_{\ell^1(\cE_\eps)}.
		\end{align*}
		\item Density lower bound: 
		$\displaystyle
		\inf_{x \in \cX_\eps}
		S_\lambda m(x) 
		\geq c_\lambda \eps^d \| m\|_{\ell^1(\cX)}.
		$
		\item  Spatial regularisation: 
		\begin{align*}
			% \label{eq:L1-to-Lip} 
			\big\| 
			\tau_\eps^z 
			S_\lambda m 
			- 
			S_\lambda m
			\big\|_{\ell^\infty(\cX_\eps)}
			& \leq 
			C_\lambda \eps^{d+1} |z| 
			\| m \|_{\ell^1(\cX_\eps)}
			\tand \\
			%%%%%%%%%%%%%%	
			\big\| 
			\tau_\eps^z 
			S_\lambda J 
			- 
			S_\lambda J
			\big\|_{\ell^\infty(\cE_\eps)}
			&\leq 
			C_\lambda \eps^{d+1} |z| 
			\| J \|_{\ell^1(\cE_\eps)}.
		\end{align*}
	\end{enumerate}
\end{lemma}

\begin{proof}
	Using the convexity of $F$ and the identity  
	$\sum_z H_\lambda^\eps(z) = 1$ 
	we obtain
	\begin{align*}
		\cF_\eps(
		S_\lambda m, 
		S_\lambda J)
		&= 	\sum_{z \in \Z_\eps^d} 
		\eps^d 
		F\bigg(
			\frac
				{\tau_\eps^z S_\lambda m}	{\eps^d}, 
			\frac
				{\tau_\eps^z S_\lambda J}	{\eps^{d-1}}
		\bigg) 
		\\	& \leq 
		\sum_{z \in \Z_\eps^d} 
		\sum_{z' \in \Z_\eps^d} 
		\eps^dH_\lambda^\eps(z') 
		F\bigg(
			\frac
				{\tau_\eps^{z+z'}  m}	{\eps^d}, 
			\frac
				{\tau_\eps^{z+z'}  J}	{\eps^{d-1}}
		\bigg) 
		\\	& = 
		\sum_{z \in \Z_\eps^d} 
		\Big( 
		\sum_{z' \in \Z_\eps^d} 	
		H_\lambda^\eps(z-z')
		\Big) 
		\eps^d 
		F\bigg(
			\frac
				{\tau_\eps^{z}  m}	{\eps^d}, 
			\frac
				{\tau_\eps^{z}  J}	{\eps^{d-1}}
		\bigg)
		= F(M,J),
	\end{align*}
	where in the last equality we used \eqref{eq:discrete_kernel_bounds_2}. This shows $(i)$. Properties $(ii)$, $(iii)$, and $(iv)$ are straightforward consequence of the uniform bounds \eqref{eq:discrete_kernel_bounds}, \eqref{eq:discrete_kernel_bounds_2} for the discrete kernels $H_\lambda^\eps$. 
\end{proof}

\subsubsection{Time regularisation} 

Fix an interval $\cI =(a,b) \subset \R$ and a regularisation parameter $\tau >0$.
For $(\bfm,\bfJ) \in\cCE_\eps^\cI$, we define for $t \in \cI_\tau:= (
a+\tau, b-\tau
)$
\begin{align*}
	(T_\tau \bfm)_t 
	:= \fint_{t-\tau}^{t + \tau} m_s \dd s,
	\qquad
	(T_\tau \bfJ)_t 
	:= 
	\fint_{t-\tau}^{t + \tau} J_s \dd s. 
\end{align*}

Note that, thanks to the linearity of the continuity equation we get $(T_\tau \bfm, T_\tau \bfJ) \in \cCE_\eps^{\cI_\tau}$.

We have the following regularisation properties for the operator $T_\tau$.

\begin{lemma} [Regularisation in time]
	\label{lem:reg-time}	
	Let $\tau \in (0, \frac{b-a}{2})$.
	The following estimates hold
	for all 
	$\eps < \frac1{2R_0}$
	and 	
	all Borel curves 
	$\bfm = (m_t)_{t \in \cI} 
	\subseteq \R_+^{\cX_\eps}$
	and 		
	$\bfJ = (J_t)_{t \in \cI} 
	\subseteq \cM^d(\cE_\eps)$:
	\begin{enumerate}[(i)]
		\item Energy estimate: for some $0 \leq C<\infty$ depending only on \eqref{eq: growth} we have	\[
		\cA_\eps^{\cI_\tau}(
		T_\tau \bfm, 
		T_\tau \bfJ)
		\leq 
		\cA_\eps(\bfm, \bfJ)
		+
		C \tau \big( \bfm(\cI \times \cX_\eps) + 1 \big). 
		\]
		\vspace{2mm}
		\item Mass estimate:
		$ \displaystyle \sup_{t \in \cI_\tau}
		\| (T_\tau m)_t \|_{\ell^p(\cX_\eps)} 
		\leq 
		\sup_{t \in \cI}
		\| m_t \|_{\ell^p(\cX_\eps)}$.
		\item Momentum estimate:
		$\displaystyle 
		\sup_{t \in \cI_\tau}
		\| (T_\tau J)_t \|_{\ell^p(\cX_\eps)} 
		\leq 
		\frac{1}{\tau}
		\int_\cI
		\| J_t \|_{\ell^p(\cX_\eps)} 
		\dd t.$
		\item Time regularity:
		$\displaystyle
		\sup_{t \in \cI_\tau}
		\big\| \partial_t (T_\tau m)_t 
		\big\|_{\ell^p(\cX_\eps)} 
		\leq \frac{1}{\tau}
		\sup_{t \in \cI}
		\| m_t \|_{\ell^p(\cX_\eps)}$.	
	\end{enumerate}
	
\end{lemma}

\begin{proof}
	Set $w_\tau(s) := (2\tau)^{-1} \big|
	[(s-\tau) \vee a ,(s+\tau)\wedge b] 
	\big|$ for $s \in \cI$. Then we have
	\begin{align}	\label{eq:estimate_action_timereg}
		\cA_\eps^{\cI_\tau}(
		T_\tau \bfm, 
		T_\tau \bfJ)
		\leq	
		\int_{\cI_\tau} 
		\fint_{t-\tau}^{t+\tau}
		\cF_\eps(m_s, J_s)
		\dd s
		\dd t
		= 
		\int_{\cI} 
		w(s)
		\cF_\eps(m_s, J_s)
		\dd s	,
	\end{align}	
	as a consequence of Jensen's inequality and Fubini's theorem. Using that $0 \leq w_\tau \leq 1$, $\int_{\cI} (1-w_\tau(s)) \dd s = 2\tau$, and the growth condition \eqref{eq: growth} we infer
	\begin{align*}
		\int_\cI (1-w_\tau(s)) \cF_\eps(m_s,J_s) \dd s 
		\geq
		-C  \tau 
		\big(
		\bfm(\cI \times \cX_\eps) + 1
		\big),
	\end{align*}
	which together with \eqref{eq:estimate_action_timereg} shows $(i)$.
	
	Properties $(ii)$, $(iii)$ follow directly from the convexity of the $\ell_p$-norms and the subadditivity of the integral.
	
	Finally, $(iv)$ follows from the direct computation
	$
	\partial_t (T_\tau m)_t  = \frac1{2\tau}
	( m_{t+\tau} - m_{t-\tau} )
	$.
\end{proof}

\subsubsection{Effects of the three regularisations}
\label{subsec:proof}

We start with a lemma that shows that the effect of the three regularising operators is small if the parameters are small. 

Recall the definition of the Kantorovich-Rubinstein norm as given in Appendix \ref{sec:KR}.
\begin{lemma}[Bounds in $\KR$-norm]
	\label{lem:W1-bounds}
	Let $\cI \subset \R$ and interval and $(m_t)_{t \in \cI} \subseteq \R_+^{\cX_\eps}$ 
	be a Borel measurable curve of constant total mass 
	(i.e., $t \mapsto m_t(\cX_\eps)$ is constant), 
	and let $\bfm \in \cM_+(\cI \times \cX_\eps)$ be
	the associated measure on space-time
	defined by
	$ \bfm := \dd t \otimes m_t$.
	Then there exists a constant $C < \infty$ depending on $|\cI|$ such that:
	\begin{enumerate}[(i)]
		\item $	\displaystyle		
		\| 	\iota_\eps T_\tau \bfm -  
		\iota_\eps \bfm 
		\|_{\KR(\overline{\cI_\tau} \times \Td)}
		\leq C \tau
		\sup_{t \in \cI}
		\big\| 	m_t
		\big\|_{\ell^1(\cX_\eps)}
		$ for any $\tau < |\cI|/2$.
		%%%%%%%%%%%%%%%%%%%%%%%%%%
		\item
		$\displaystyle
		\|
		\iota_\eps S_\lambda \bfm -  
		\iota_\eps \bfm 
		\|_{\KR(\overline{\cI} \times \Td)}
		\leq C \sqrt{\lambda} 	\sup_{t \in \cI}
		\big\| 
		m_t
		\big\|_{\ell^1(\cX_\eps)}$
		for any $\lambda > 0$.
		%%%%%%%%%%%%%%%%%%%%%%%%%%
		\item 
		$\displaystyle
		\|
		\iota_\eps R_\delta \bfm - 
		\iota_\eps \bfm 
		\|_{\KR(\overline{\cI} \times \Td)}
		\leq C \delta
		\Big( 
		m^\circ(\XQ)
		+
		\sup_{t \in \cI}
		\big\| 
		m_t
		\big\|_{\ell^1(\cX_\eps)}
		\Big)
		$
		for any $\delta \in (0,1)$.
	\end{enumerate}
\end{lemma}

\begin{proof}
	$(i)$: \ 
%	First we observe that 
%	$\iota_\eps T_\tau \bfm(\cI_\tau \times \T^d) 
%	=
%	\iota_\eps \bfm(\cI \times \T^d)   
%	$, since $(m_t)_{t \in [0,T]}$ is of constant total mass.
	For any $\bfmu \in \cM(\cI \times \T^d)$ 
	and any Lipschitz function 
	$\phi : \overline{\cI_\tau} \times \T^d \to \R$ 
	(and, in fact, for any temporally Lipschitz function)
	we have
	\begin{align*}
		& \bigg| \int_{\overline{\cI_\tau} \times \T^d}
		\phi(t,x)
		\dd \bfmu(t,x)
		- 
		\int_{\overline{\cI_\tau} \times \T^d}
		\phi(t,x)
		\dd (T_\tau \bfmu)(t,x)
		\bigg|
		\\ & = 
		\bigg|\int_{\overline{\cI_\tau} \times \T^d} 
		\fint_{t-\tau}^{t+\tau}
		\phi(s,x) - \phi(t,x)
		\dd s	\dd \bfmu(t,x)
		\bigg|
		\leq \tau [\phi]_{\Lip} \bfmu\big(\cI \times \T^d\big).
	\end{align*}
	Since $\iota_\eps \bfm \big(\cI \times \T^d\big) 
	\leq 
	|\cI| \sup_{t \in \cI}
	\big\| 	m_t
	\big\|_{\ell^1(\cX_\eps)}$
	we obtain the result.
	
	\smallskip $(ii)$: \ 
	In view of mass-preservation, we have
	\begin{align*}
		\|
		\iota_\eps S_\lambda \bfm -  
		\iota_\eps \bfm 
		\|_{\KR(\overline{\cI} \times \Td)}
		\leq &
		\int_{\cI}
		\big\| 	
		\iota_\eps S_\lambda m_t -  
		\iota_\eps m_t
		\big\|_{\KR(\Td)}  
		\dd t \\
		 \leq & \
\sup_{t\in \cI} m_t(\cX_{\eps}) \int_{\cI}
\big\| 	
\iota_\eps H_\lambda -  
\iota_\eps H_0
\big\|_{\KR(\Td)}  
\dd t \\
\leq  &C \sqrt{\lambda} \sup_{t\in \cI} m_t(\cX_\eps).	\end{align*}
	Here in the last inequality we used scaling law of the heat kernel.
	
	\smallskip $(iii)$: \ 
	Let us write $\bfm_\eps^\circ := 
	\dd t \otimes m_\eps^\circ$ 
	for brevity.
	By linearity, we have
	\begin{align*}
		\|
		\iota_\eps (R_\delta \bfm - \bfm)
		\|_{\KR(\overline{\cI} \times \Td)}
		&= \delta
		\|		 
		\iota_\eps (\bfm_\eps^\circ - \bfm)
		\|_{\KR(\overline{\cI} \times \Td)} \\
		& \leq \delta (1 + |\cI|) 
		\Big( 
		\bfm_\eps^\circ\big(\cI \times \T_\eps^d\big)	
		+
		\bfm\big(\cI \times \T_\eps^d\big)
		\Big)	
		\\ & \leq \delta |\cI| (1 + |\cI|) 
		\Big( 
		m^\circ(\XQ)
		+	
		\sup_{t \in \cI} m_t(\cX_\eps)
		\Big).	
	\end{align*}
\end{proof}

\begin{proof}[Proof of Proposition \ref{prop:regularisation_discrete}]
	We define 
	\begin{align*}
		\tilde \bfm
		:= \Big(
		R_\delta 
		\circ S_\lambda 
		\circ T_\tau 
		\Big) 
		\bfm
		\tand		
		\tilde \bfJ
		:= \Big(
		R_\delta 
		\circ S_\lambda 
		\circ T_\tau 
		\Big) 
		\bfJ.
	\end{align*}
	We will show that the desired inequalities hold if $\delta, \lambda, \tau > 0$ are chosen to be sufficiently small, depending on the desired accuracy $\eta > 0$. Set $\cI_\tau:=(\tau,T-\tau)$.
	
	\smallskip
	$(i)$: \ 
	We use the shorthand notation $\KR_\tau:= \KR(\overline{\cI}_\tau\times \Td)$. Using Lemma \ref{lem:W1-bounds} we obtain
	\begin{equation}
		\begin{aligned}
			\label{eq:W-bound}	
			\|
			\iota_\eps \bfm - 
			\iota_\eps \tilde \bfm
			\|_{\KR_\tau}
			& \leq 
			\| 	   \iota_\eps \bfm
			-  \iota_\eps T_\tau \bfm 
			\|_{\KR_\tau}
			+
			\|	\iota_\eps T_\tau \bfm -  
			\iota_\eps (S_\lambda T_\tau) \bfm 
			\|_{\KR_\tau}
			\\& \qquad +
			\| 	\iota_\eps (S_\lambda T_\tau) \bfm - 
			\iota_\eps (R_\delta S_\lambda T_\tau) \bfm 
			\|_{\KR_\tau} 
			\\& \lesssim
			M( \tau + \sqrt{\lambda} + \delta )
			+ m^\circ(\XQ) \delta.
		\end{aligned}
	\end{equation}
	Furthermore, 
	using Lemma \ref{lem:reg-energy}, Lemma \ref{lem:reg-space}(i), and Lemma \ref{lem:reg-time}(i)  we obtain the action bound
	\begin{equation}
		\begin{aligned}
			\label{eq:energy}	
			\cA_\eps^{\cI_\tau}(\tilde \bfm, \tilde \bfJ)
			& = 
			\cE_\eps
			\Big(
			(
			R_\delta \circ S_\lambda \circ T_\tau 
			) \bfm, 
			(
			R_\delta \circ S_\lambda \circ T_\tau 
			) \bfJ
			\Big)
			\\& \leq 
			(1 - \delta)
			\cA_\eps
			\Big(
			(
			S_\lambda \circ T_\tau 
			) \bfm, 
			(
			S_\lambda \circ T_\tau 
			) \bfJ
			\Big)
			+ \delta T 
			\cF_\eps(m_\eps^\circ, J_\eps^\circ)
			\\& \leq 
			(1 - \delta)
			\cA_\eps
			(
			\bfm, 
			\bfJ
			)
			+ \delta T 
			\cF(m^\circ, J^\circ)
			+
			C \tau (M+1).
		\end{aligned}
	\end{equation}
	The desired inequalities \eqref{eq:approx-bounds} follow by choosing $\delta$, $\lambda$, and $\tau$ sufficiently small.

	\smallskip
	$(ii)$: \ 
	We will show that all the estimates hold with constants depending on $\eta$ through the parameters $\delta$, $\lambda$, and $\tau$.
	
	\smallskip
	\emph{Boundedness}: \
	We apply Lemma \ref{lem:reg-energy}, Lemma \ref{lem:reg-space}(ii), and Lemma \ref{lem:reg-time}(ii)\&(iii) and obtain the uniform bounds on the mass
	\begin{equation}
		\begin{aligned}
			\label{eq:boundedness-m}
			\sup_{t \in \cI_\tau}	
			\| \tilde m_t \|_{\ell^\infty(\cX_\eps)}
			& \leq 
			\eps^d
			\bigg(
			(1-\delta) C_\lambda
			\sup_{t \in [0,T]}
			\| m_t \|_{\ell^1(\cX_\eps)}
			+
			\delta	\| m^\circ \|_{\ell^\infty(\cX_\eps)}
			\bigg),
			\\
			& \leq 
			\eps^d
			\bigg(
			C_\lambda
			M
			+
			\delta	\| m^\circ \|_{\ell^\infty(\XQ)}
			\bigg)
		\end{aligned}
	\end{equation}
	as well as the uniform bounds on the momentum
	\begin{equation}
		\begin{aligned}
			\label{eq:boundedness-J}
			\sup_{t \in \cI_\tau}	
			\| \tilde J_t \|_{\ell^\infty(\cX_\eps)}
			& \leq 
			\eps^{d-1}
			\bigg(
			\frac{1 - \delta}{\tau}C_\lambda
			\sup_{t \in [0,T]}
			\int_\cI \eps \| J_t \|_{\ell^1(\cX_\eps)} \dd t
			+
			\delta	\| J^\circ \|_{\ell^\infty(\cX_\eps)}
			\bigg),
			\\& \lesssim 
			\eps^{d-1}
			\bigg(
			\frac{C_\lambda}{\tau}
			\Big(
			T (1 +  M) 
			+ E
			\Big)
			+
			\delta	\| J^\circ \|_{\ell^\infty(\EQ)}
			\bigg).
		\end{aligned}
	\end{equation}

	\smallskip
	\noindent
	\emph{Time-regularity}: \
	From Lemma \ref{lem:reg-time}(iv), together with Lemma \ref{lem:reg-energy} and Lemma \ref{lem:reg-space}(ii), we obtain the uniform bound on the time derivative 
	\begin{equation}
		\begin{aligned}
			\label{eq:time-reg}
			\sup_{t \in \cI_\tau}	
			\| \partial_t \tilde m_t \|_{\ell^\infty(\cX_\eps)}
			& \leq 
			\eps^d
			\bigg(
			2
			\frac{1-\delta}{\tau} C_\lambda
			\sup_{t \in [0,T]}
			\| m_t \|_{\ell^1(\cX_\eps)}
			+
			\delta	\| m^\circ \|_{\ell^\infty(\cX_\eps)}
			\bigg),
			\\ & \leq 
			\eps^d
			\bigg(
			2
			\frac{C_\lambda}{\tau} 
			M
			+
			\delta	\| m^\circ \|_{\ell^\infty(\XQ)}
			\bigg).	
		\end{aligned}
	\end{equation}
	
	\smallskip
	\noindent	
	\emph{Space-regularity}: 
	For $z,z' \in \Z_\eps^d$ and $v \in \V$, Lemma \ref{lem:reg-space}(iv) and Lemma \ref{lem:reg-time}(ii) yield
	\begin{align*}
		| 
		\tilde m_t(z,v) - 
		\tilde m_t(z',v) 
		|
		& \leq (1-\delta) 
		\big| 
		\big(S_\lambda \circ T_\tau\big)m_t(z,v)
		- \big(S_\lambda \circ T_\tau\big)m_t(z',v)
		\big|
		\\& \leq C_\lambda
		\eps^{d-1} |z-z'|
		\big\| 
		T_\tau m_t
		\big\|_{\ell^1(\cX_\eps)}
		\\& \leq C_\lambda
		\eps^{d+1} |z-z'|
		\sup_{t \in [0,T]}
		\big\| 
		m_t
		\big\|_{\ell^1(\cX_\eps)},
	\end{align*}
	which shows that
	\begin{equation}
		\begin{aligned}
			\label{eq:space-reg-m}	
			\sup_{t\in \cI_\tau}
			\|
			\sigma_\eps^z \tilde m_t
			- \tilde m_t
			\|_{\ell^\infty(\cX_\eps)}
			& \leq
			C_\lambda
			\eps^{d+1} |z|
			\sup_{t \in [0,T]}
			\big\| 
			m_t
			\big\|_{\ell^1(\cX_\eps)}			
			% \\ &
			\leq C_\lambda
			\eps^{d+1}
			|z|
			M.	
		\end{aligned}
	\end{equation}
	Similarly, using the growth condition \eqref{eq: growth} we deduce 
	\begin{equation}
		\begin{aligned}
			\label{eq:space-reg-J}
			\sup_{t\in \cI_\tau}
			\|
			\sigma_\eps^z \tilde J_t
			- \tilde J_t
			\|_{\ell^\infty(\cE_\eps)}
			& \leq
			\frac{C_\lambda}{\tau}
			\eps^{d+1} |z|
			\int_\cI
			\big\| 
			J_s
			\big\|_{\ell^1(\cE_\eps)}			
			\dd s
			\\ &\leq
			\frac{C_\lambda}{\tau}
			\eps^{d} 	
			|z|
			\Big(
			T (1 +  M) 
			+ E
			\Big).	
		\end{aligned}
	\end{equation}
	
	\smallskip
	\noindent
	\emph{Domain-regularity}: \
	For all $t \in \cI_\tau$, reasoning as in \eqref{eq:boundedness-m} and \eqref{eq:boundedness-J}, we observe that
	\begin{align*}
		\eps^{-d} 	\|	(S_\lambda T_\tau m)_t  
		\|_{\ell^\infty(\cX_\eps)}
		\leq C_\lambda \| (T_\tau m)_t  
		\|_{\ell^1(\cX_\eps)} 
		\leq C_\lambda \sup_{t \in [0,T]}
		\| m_t  
		\|_{\ell^1(\cX_\eps)} 
		\leq C_\lambda M,\\
		%%%%%%%%%%%%%%%%%%%%%%%%%
		\eps^{-d} 	\|	(S_\lambda T_\tau J)_t  
		\|_{\ell^\infty(\cE_\eps)}
		\leq C_\lambda \| (T_\tau m)_t  
		\|_{\ell^1(\cE_\eps)} 
		\leq \frac{C_\lambda}{\tau} \int_\cI
		\| J_t \|_{\ell^1(\cE_\eps)} 
		\dd t
		\leq  \frac{C_\lambda}{\tau\eps} \Big(
		T (1 +  M) 
		+ E
		\Big).
	\end{align*}
	We infer that  
	\begin{align*}
		\bigg\|
		\frac{\tau_\eps^z (S_\lambda T_\tau m)_t  }
		{\eps^d}
		\bigg\|_{\ell^\infty(\cX)} 
		\leq C_\lambda M
		\quad	\text{and} \quad
		\bigg\|
		\frac{\tau_\eps^z (S_\lambda T_\tau J)_t  }
		{\eps^{d-1}}
		\bigg\|_{\ell^\infty(\cE)}
		\leq \frac{C_\lambda}{\tau} \Big(
		T (1 +  M) 
		+ E
		\Big)
	\end{align*}
	Since
	\begin{align*}
		\bigg(\frac{\tau_\eps^z \tilde m_t}
		{\eps^d}
		,
		\frac{\tau_\eps^z \tilde J_t}
		{\eps^{d-1}}
		\bigg)
		= (1-\delta)
		\bigg(
		\frac{\tau_\eps^z (S_\lambda T_\tau m)_t  }
		{\eps^d},
		\frac{\tau_\eps^z (S_\lambda T_\tau J)_t  }
		{\eps^{d-1}}
		\bigg)
		+ 
		\delta 
		(m^\circ, J^\circ),
	\end{align*}
	the claim follows by an application of Lemma \ref{lemma:prop_domf} to the product of balls in  ${\ell^\infty(\cX)}$ and ${\ell^\infty(\cE)}$, taking into account that $F$ is defined on a finite-dimensional subspace by the locality assumption.
\end{proof}

\section{Proof of the upper bound}	\label{sec: upper}
In this section we present the proof of the $\Gamma$-limsup inequality in Theorem \ref{thm:main}. The first step is to introduce the notion of \textit{optimal microstructures}.

\subsection{The optimal discrete microstructures}

Let $\cI$ be an open interval in $\R$.
We will make use of the following canonical discretisation of measures and vector fields on the cartesian grid $\Z_\eps^d$.

\begin{definition}[$\Z_\eps^d$-discretisation of measures]
	\label{def:discretisation_cube}	
	Let 
		$\mu \in \cM_+(\T^d)$ 
	and 
		$\nu \in \cM^d(\T^d)$ 
	have continuous densities $\rho$ and $j$, respectively, with respect to the Lebesgue measure.
	Their $\Z_\eps^d$-discretisations 
		$\etP_\eps \mu: \Z_\eps^d \to \R_+$ 
	and $\etP_\eps \nu: \Z_\eps^d \to \R^d$ 
	are defined by
	\begin{align*}
			\etP_\eps \mu(z):=\mu(Q_\eps^z), 
		\quad 
			\etP_\eps \nu(z):=\left(\int_{\partial Q_\eps^z\cap \partial Q_\eps^{z+ e_i}}j\cdot e_i\dd\cH^{d-1}\right)_{i=1}^d.
	\end{align*}
\end{definition}

An important feature of this discretisation is the preservation of the continuity equation, in the following sense.

\begin{definition}[Continuity equation on $\Z_\eps^d$]
		\label{def:cont_eq_Zepsd}
	Fix $\cI \subset \R$ an open interval. We say that $\bfr: \cI \times \Z_\eps^d \to \R_+$ and $\bfu : \cI \times \Z_\eps^d \to \R^d$ satisfy the continuity equation on $\Z_\eps^d$, and write 
	$(\bfr, \bfu) \in \CE_{\eps,d}^\cI$, 
if $\bfr$ is continuous, $\bfu$ is Borel measurable, and the following discrete continuity equation is satisfied in the sense of distributions:
	\begin{align}	\label{eq:cont_eq_Zepsd}
		\partial_t r_t(z) 
		+ \sum_{i=1}^d 
		\big( u_t(z) - u_t(z - e_i) \big)
		\cdot e_i 
		= 0, \quad \text{for } z \in \Z^d_\eps.
	\end{align}
\end{definition}

\begin{lemma}[Discrete continuity equation on $\Z_\eps^d$]
		\label{lemma:disc_cont_Zepsd}
	Let $(\bfmu,\bfnu) \in \bCE^\cI$ 
	have continuous densities with respect to the space-time Lebesgue measure on $\cI \times \Td$. 
	Then $(\etP_\eps \bfmu, \etP_\eps \bfnu) \in \CE_{\eps,d}^\cI$. 
\end{lemma}

\begin{proof}
This follows readily from the Gau{\ss} divergence theorem.	
\end{proof}

The key idea of the proof of the upper bound in Theorem \ref{thm:main} is to start from a (smooth) solution to the continuous equation $\bCE^\cI$, and to consider the optimal discrete microstructure of the mass and the flux in each cube $Q_\eps^z$. 
The global candidate is then obtained by gluing together the optimal microstructures \emph{cube by cube}.

We start defining the \textit{gluing operator}. Recall the operator $T_\eps^0$ defined in \eqref{eq:def-T}.

\begin{definition}[Gluing operator]	\label{def:gluing}
	Fix $\eps > 0$.
	For each $z \in \Z_\eps^d$, let
	\begin{align*}
		m^z \in \R_+^{\cX}
			\tand 
		J^z \in \R_a^{\cE}
	\end{align*}
	be $\Z^d$-periodic.
	The \emph{gluings} of 
		$m = (m^z)_{z \in \Z_\eps^d}$
	and 
		$J = (J^z)_{z \in \Z_\eps^d}$
	are the functions 
	$\cG_\eps m \in \Meps$ and $\cG_\eps J \in \Mdeps$ 
	defined by
	\begin{equation}
		\begin{aligned}
			\label{eq:Gluing_mass}
			\cG_\eps m \big( T_\eps^0(x) \big) 
					&:= m^{ x_\sz }(x) 
			&& \text{for } x \in \cX, \\ 
				\cG_\eps J
				\big( T_\eps^0(x), T_\eps^0(y) \big) 
					&:= 
					\frac12 
				\Big( 
					J^{x_\sz}(x,y)  
					+ 
					J^{y_\sz}(x,y) 
				\Big)
				&& \text{for } 
				(x,y) \in \cE.			
		\end{aligned}
	\end{equation}
\end{definition}

\begin{remark}[Well-posedness]
	Note that $\cG_\eps m$ and $\cG_\eps J$ are well-defined thanks to the $\Z_\eps^d$-periodicity of the functions $m^z$ and $J^z$.
\end{remark}

\begin{remark}(Mass preservation and KR-bounds)	
	\label{rem:mass_preservation_tP*}
	The gluing operation is locally mass-preserving in the following sense.
	Let $\mu \in \cM_+(\Td)$ and consider a family of measures 
		$m = (m^z)_{z \in \Z_\eps^d} 
		\subseteq \R_+^{\cX}$ 
	satisfying 
		$m^z \in \Rep\big( \etP_\eps\mu(z) \big)$ 
	for some $z \in \Z_\eps^d$. 
	Then:
	\begin{align*}
		\cG_\eps m 
			\Big( \cX_\eps \cap \{ x_\sz = z\} \Big) 
		= \mu(Q_\eps^z)
	\end{align*}
	for every 
	$\eps > 0$.
	Consequently,
	\begin{align}
		\label{eq:rem_KR_iotaP*}
		\| 
			\iota_\eps \cG_\eps \bfm
			- \bfmu  
		\|_{{\KR}(\overline{\cI} \times \T^d)} 
		\leq 
		\bfmu\big(\overline \cI \times \T^d\big) \sqrt d \eps
	\end{align}
	for all weakly continuous curves
		$\bfmu 
			= (\mu_t)_{t \in \overline\cI} 
			\subseteq 
			\cM_+(\T^d)$ 
	and all 
		$\bfm = 
			(m_t^z)_{t \in \overline\cI, 
					 z \in \Z_\eps^d}$ 
	such that $m_t^z \in \Rep\big(\tP_\eps \mu_t(z)\big)$ 
	for all $t \in \overline\cI$ and 
	$z \in \Z_\eps^d$.
\end{remark}

\subsubsection{Energy estimates for Lipschitz microstructures}
The next lemma shows that the energy of glued measures can be controlled under suitable regularity assumptions.

\begin{lemma}[Energy estimates under regularity]	\label{lemma:energy_est_glued_measures_regular}
	Fix $\eps > 0$. 
	For each $z \in \Z_\eps^d$, 
	let
		$m^z \in \R_+^{\cX}$ and 
		$J^z \in \R_a^{\cE}$ 
	be $\Z^d$-periodic functions satisfying:
	\begin{enumerate}[(i)]
		\item
		(Lipschitz dependence): For all 
			$z, \tilde z \in \Z_\eps^d$ 
		\begin{align*}
			\big\| m^z - m^{\tilde z} 
			\big\|_{\ell^\infty(\cX)}
			+ \eps
			\big\|  J^z - J^{\tilde z}
			\big\|_{\ell^\infty(\cE)}
			 \leq L |z - \tilde z| 
			 \eps^{d+1}.
		\end{align*}
		\item (Domain regularity): There exists a compact and convex set $K\Subset \Dom(F)^\circ$ such that, for all $z \in \Z_\eps^d$,
		\begin{align}	
					\bigg( 
					\frac{m^z}{\eps^d}, 
					\frac{J^z}{\eps^{d-1}} 
					\bigg)
			\in K.
		\end{align} 
	\end{enumerate}
 Then there exists $\eps_0 > 0$ depending only on $K$, $F$ such that for $\eps \leq \eps_0$
\begin{align}	\label{eq:lemma_energy_regularity}
	\cF_\eps 
		\big(
			\cG_\eps m, \cG_\eps J	
		\big) \leq 
		\sum_{z \in \Z_\eps^d}  
			\eps^d F 
				\bigg(  
					\frac{m^z}{\eps^d}, 
					\frac{J^z}{\eps^{d-1}}
				\bigg) 
			+ c \eps,
\end{align}
where $c < \infty$ depends only on $L$, 
the (finite) Lipschitz constant $\Lip(F;K)$, 
and the locality radius $R_1$.
\end{lemma}

\begin{proof}
	Fix $\bar z \in \Z_\eps^d$.  
		As $m$ is $\Z^d$-periodic, $(i)$ yields for $x = (z,v) \in \cX_{R_1}$,
	\begin{align}		\label{eq:proof8.5_i}	
		| 
		\tau_\eps^{\bar z} \cG_\eps m(x) 
		- m^{\bar z}(x)   
		|  
		& =
		| 
		m^{\bar z + z}(x) 
		- m^{\bar z}(x)   
		| 
		\leq L R_1 \eps^{d + 1}, 
	\end{align}
	Similarly, using the $\Z^d$-periodicity of $J$,
	$(i)$ yields 
	for $(x, y) \in \cE$ with 
	$x = (z,v) \in \cX_{R_1}$ and 
	$y = (\tilde z, \tilde v) \in \cX_{R_1}$,
	\begin{align}	\label{eq:proof8.5_ii}
		| 
		\tau_\eps^{\bar z} \cG_\eps J(x, y) 
		- J^{\bar z}(x, y)   
		|  
		& =
		\Big| 
		\Big( \tfrac12 J^{\bar z + z}
		+ \tfrac12 J^{\bar z + \tilde z}
		- J^{\bar z}
		\Big)(x,y)
		\Big|
		\leq L R_1 \eps^d.
	\end{align}

	Hence the domain regularity assumption $(ii)$
	imply a domain regularity property for the glued measures, namely
	 \begin{align*}
	 	\bigg( 
			 \frac{\tau_\eps^{\bar z} \cG_\eps m}{\eps^d}, 
			 \frac{\tau_\eps^{\bar z} \cG_\eps J}{\eps^{d-1}} 
		\bigg) 
		\in \tilde K 
	 \end{align*}
	 for all ${\bar z} \in \Z_\eps^d$ and $\eps \leq  \eps_0:=\frac12\text{dist}(K, \partial \Dom(F))  \in (0,+\infty)$, where $\tilde K \Subset \Dom(F)^\circ$ is a slightly bigger compact set than $K$.
	 
	 Consequently, 
	 we can use the Lipschitzianity of $F$ on the compact set $\tilde K$ and its locality to estimate the energy as
	 \begin{align*}
		&\left| 
	 		F \bigg( 
	 			\frac{\tau_\eps^{\bar z} 
				 \cG_\eps m}
				 {\eps^d}, 
				 \frac{\tau_\eps^{\bar z} 
				 \cG_\eps J
				 }{\eps^{d-1}}
	 		\bigg)
	 			-
	 		F \bigg(  
	 			\frac{M^{\bar z}}{\eps^d}, 
			 	\frac{J^{\bar z}}{\eps^{d-1}}
	 		\bigg)
	 	 \right| 
		\\ & 
	 	\leq
			\Lip(F;\tilde K) 
			\bigg(
	 	 	\frac{
				  \| 
			  	\tau_\eps^{\bar z} \cG_\eps m 
				  - m^{\bar z} 
				  \|_{\ell^\infty(\cX_{R_1})}}
				{\eps^d}
			+
			\frac{
				\| 
			  	\tau_\eps^{\bar z} \cG_\eps J 
			  	- J^{\bar z} 
				  \|_{\ell^\infty(\cE_{R_1})}}
				{\eps^{d-1}}
			\bigg),
	 \end{align*}
	 where $\cX_R := \{x \in \cX \ : \ |x|_{\ell_\infty^d} \leq R\}$
	 and 
	 $\cE_R := \{(x,y) \in \cE \ : \ |x|_{\ell_\infty^d} , |y|_{\ell_\infty^d} \leq R\}$.

Combining the estimate above with \eqref{eq:proof8.5_i} and \eqref{eq:proof8.5_ii}, we conclude that
\begin{align*}
	&\left| 
	 		F \bigg( 
	 			\frac{\tau_\eps^{\bar z} 
				 \cG_\eps m}
				 {\eps^d}, 
				 \frac{\tau_\eps^{\bar z} 
				 \cG_\eps J
				 }{\eps^{d-1}}
	 		\bigg)
	 			-
	 		F \bigg(  
	 			\frac{M^{\bar z}}{\eps^d}, 
			 	\frac{J^{\bar z}}{\eps^{d-1}}
	 		\bigg)
	\right| 
		\leq
	2 L R_1 \Lip(F;\tilde K) \eps.
\end{align*}
for $\eps \leq \eps_0$. Summation over $\bar z \in \Z_\eps^d$ yields the desired estimate \eqref{eq:lemma_energy_regularity}.
\end{proof}

We now introduce the notion of \textit{optimal microstructure} associated with a pair of measures $(\mu,\nu) \in \cM_+(\T^d) \times \cM^d(\T^d)$. 
First, let us define regular measures.

\begin{definition}[Regular measures]	\label{def:regular}
 	We say that $(\mu,\nu) \in \cM_+(\T^d) \times \cM^d(\T^d)$ is a \emph{regular pair of measures} if the following properties hold:
 	\begin{enumerate}[(i)]
 		\item (Lipschitz regularity): 
		 With respect to the Lebesgue measure on $\T^d$, the measures $\mu$ and $\nu$ have Lipschitz continuous densities $\rho$ and $j$ respectively.
 		\item (Compact inclusion): There exists a compact set 		$\tilde K \Subset \Dom (f_\hom)^\circ$ 
		such that
 		\begin{align*}
 			\big(\rho(x),j(x)\big) 
			 	\in \tilde K \quad \text{ for all } x \in \Td.
 		\end{align*}
 	\end{enumerate}
 	We say that 
	 	$(\mu_t, \nu_t)_{t \in \cI} 
		 	\subseteq 
		\cM_+(\T^d) \times \cM^d(\T^d)$ 
	is a \emph{regular curve of measures} 
	if $(\mu_t,\nu_t)$ are regular measures 
		for every $t \in \cI$ and 
	$t \mapsto (\rho_t(x), j_t(x))$ is measurable 
		for every $x \in \T^d$.
\end{definition}

\begin{definition}[Optimal microstructure]	\label{def:optimal_micro}
	Let $(\mu, \nu) \in \cM_+(\T^d) \times \cM^d(\T^d)$ be a regular pair of measures. 
	\begin{enumerate}[(i)]
	\item 
	 	We say that 
			$(m^z, J^z)_{z \in \Z_\eps^d} 
			\subseteq \R_+^\cX \times \R_a^\cE$ 
		is an \emph{admissible microstructure} for $(\mu, \nu)$
		if
		\begin{align*}
			(m^z, J^z) 
			\in \Rep
				\bigg(
				\frac{\etP_\eps \mu(z)}{\eps^d},
				\frac{\etP_\eps \nu(z)}{\eps^{d-1}}
				\bigg)
		\end{align*}
		for every $z \in \Z_\eps^d$.
	\item 	
		If, additionally, $(m^z, J^z) 
		\in \Rep_o
			\Big(
			\frac{\etP_\eps \mu(z)}{\eps^d},
			\frac{\etP_\eps \nu(z)}{\eps^{d-1}}
			\Big)$
		for every $z \in \Z_\eps^d$, we say that
		$(m^z, J^z)_{z \in \Z_\eps^d}$ 
		is an \emph{optimal microstructure} for $(\mu, \nu)$.
	\end{enumerate}
\end{definition}

\begin{remark}[Measurable dependence]	
	\label{rem:well-posed}
	If $t \mapsto (\mu_t,\nu_t)$ is a measurable curve in $\cM_+(\T^d) \times \cM^d(\T^d)$, it is possible to select a  
	collection of admissible (resp. optimal) microstructures that depend measurably on $t$. 
	This follows from 
	Lemma \ref{lem:cell-formula}; see e.g. \cite[Theorem 14.37]{rockafellar1998}. 
	In the sequel, we will always work with measurable selections.
\end{remark}

The next proposition shows that each optimal microstructures associated with a regular pair of measures $(\mu, \nu)$ has discrete energy which can be controlled by the homogenised continuous energy $\bF_\hom(\mu,\nu)$.

\begin{proposition}[Energy bound for optimal microstructures]	\label{prop:energy_optmicro}
	Let 
		$(m^z, J^z)_{z \in \Z_\eps^d} 
		\subseteq \R_+^\cX \times \R_a^\cE$
	be an optimal microstructure
	for a regular pair of measures
		$(\mu, \nu)
		\in \cM_+(\T^d) \times \cM^d(\T^d)$.
	Then:
	\begin{align*}
	\sum_{z \in \Z_\eps^d} 
		\eps^d 
		F 
		\left(  
			\frac{ m^z}{\eps^d}, 
			\frac{ J^z}{\eps^{d-1}} 
		\right) 
		\leq \bF_{\hom}(\mu, \nu) 
			+ C \eps,
	\end{align*}
where $C < \infty$ depends only on $\Lip(f_\hom; \tilde K)$ and the modulus of continuity of the densities $\rho$ and $j$ of $\mu$ and $\nu$.
\end{proposition}

\begin{proof}
	Let us denote the densities of 
		$\mu$ and $\nu$ by 
		$\rho$ and $j$ respectively.
	Using the regularity of $\bfmu$ and $\bfnu$, 
	and the fact that $f_\hom$ is Lipschitz on $\tilde K$, we obtain
	\begin{align*}
		\sum_{z \in \Z_\eps^d} 
			\eps^d F \left(  \frac{ m^z}{\eps^d}, 
			\frac{ J^z}{\eps^{d-1}}  \right) 
			& =  
				\sum_{z \in \Z_\eps^d} 
					\eps^d 
					f_\hom
					\bigg(
					\frac{\etP_\eps \mu(z)}{\eps^d},
					\frac{\etP_\eps \nu(z)}{\eps^{d-1}}
					\bigg)
			\leq 
				\int_{\T^d} 	
					f_\hom(\rho_t(a), j_t(a)) 
				\dd a 
				+ C \eps,
	\end{align*}
which is the desired estimate.
\end{proof}

\begin{remark}[Lack of regularity]	\label{rem:lack_regularity}
	Suppose that $\hat m := \cG_\eps m$ and $\hat J := \cG_\eps J$ are constructed by gluing the optimal microstructure $(m,J) = (m^z, J^z)_{z\in \Z_\eps^d}$ from the previous lemma. 
	It is then tempting to seek for an estimate of the form
	\begin{align*}
		\cF_\eps(\hat m , \hat J) 
		\leq 
		\sum_{z \in \Z_\eps^d} 
		\eps^d 
		F 
		\left(  
			\frac{ m^z}{\eps^d}, 
			\frac{ J^z}{\eps^{d-1}} 
		\right) 
		+ \ \text{\{small error\}}.
	\end{align*}
	However, $(m,J)$ does not have the required \emph{a priori} regularity estimates to obtain such a bound. 
	Moreover, the gluing procedure does not necessarily produce solutions to the discrete continuity equation if we start with solutions to the continuous continuity equation.
\end{remark}

We conclude the subsection with the following $L^1$ and $L^\infty$ estimates.
\begin{lemma}[$L^1$ and $L^\infty$ estimates] 	\label{lemma:L1-infty_estimates_Peps}
	Let $(\mu_t,\nu_t)_{t\in\cI} \subset \cM_+(\T^d) \times \cM^d(\T^d)$ be a regular curve of measures  satisfying
	\begin{align}
		M := \sup_{t \in \cI} \mu_t(\T^d) < \infty 
		\tand 
		A := \bA_\hom^\cI(\bfmu,\bfnu) < \infty.
	\end{align} 
	Let 
	$(m_t^z, J_t^z)_{z \in \Z_\eps^d} 
		\subseteq \cM_+(\T^d) \times \cM^d(\T^d)$
	be corresponding optimal microstructures. Then: 
\begin{enumerate}[(i)]
	\item 
	$(\etP_\eps \mu_, \etP_\eps \nu)$ satisfies the uniform estimate
	\begin{align}	\label{eq:l1_bounds_r_u}
		\sup_{\eps>0} 
		\sup_{t \in \cI} 
				\| \etP_\eps \mu_t \|_{\ell^1(\Z_\eps^d)} 		
		= M.
	\end{align} 
	\item 
	$(m_t, J_t)_{t \in \cI}$ satisfies the uniform estimate
	\begin{align}
	\label{eq:upperbound_Mteps}
		& \sup_{\eps>0} 
		\sup_{(t,x) \in \cI \times \cX} 
			\sum_{z \in \Z_\eps^d} 
			m_t^z(x)
			\leq M
	\\
	\label{eq:upperbound_Uteps}
		& \sup_{\eps>0} 
		\sup_{(x,y) \in \cE} 
			\eps 
			\int_\cI 
			\sum_{z \in \Z_\eps^d} 
				\big| J_t^z(x,y) \big| 
			\dd t \lesssim A + M.
	\end{align}
\end{enumerate}
\end{lemma}

\begin{proof}
	The first claim follows since
	$\| \etP_\eps \mu_t \|_{\ell^1(\Z_\eps^d)}
	= \mu_t(\T^d)$ by construction.
	
To prove $(ii)$, note that
\begin{align*}
		\sum_{z \in \Z_\eps^d}
		\sum_{x \in \XQ}
			m_t^z(x)
		= \sum_{z \in \Z_\eps^d}
			\etP_\eps \mu(z)
		= \mu_t(\T^d),
\end{align*}
which yields \eqref{eq:upperbound_Mteps}.

To prove \eqref{eq:upperbound_Uteps}, we use the growth condition on $F$, the periodicity of $J_t^z$, and $(i)$ to obtain for $(x,y)\in\cE$ and $t \in \cI$: 
\begin{align*}
	\eps \sum_{z \in \Z_\eps^d} \big|J_t^z(x,y) \big| 
		\leq
			\sum_{z \in \Z_\eps^d} \eps^d \sum_{(\tilde x , \tilde y) \in \cE^Q}  \bigg| \frac{J_t^z(\tilde x , \tilde y)}{\eps^{d-1}} \bigg| 
		&\lesssim 
			\sum_{z \in \Z_\eps^d} \eps^d F \Big( \frac{m_t^z}{\eps^d},\frac{J_t^z}{\eps^{d-1}} \Big) + M \\
		&\lesssim 
			\int_{\Td} f_\hom \Big( \frac{\dd\mu_t}{\dd x},\frac{\dd j_t}{\dd x}\Big) \dd x + M,
\end{align*}
where in the last inequality we applied Proposition \ref{prop:energy_optmicro}. Integrating in time and taking the supremum in space and $\eps>0$, we obtain \eqref{eq:upperbound_Uteps}.
\end{proof}

\subsection{Approximation result}
The goal of this subsection is to show that 
despite the issues of Remark \ref{rem:lack_regularity}, 
we can find a solution to $\cCE_\eps^\cI$ with almost optimal energy that is $\| \cdot \|_{\KR}$-close to a glued optimal microstructure.

In the following result, 
	$\cI_\eta = (a - \eta, b + \eta)$ denotes the $\eta$-extension of the open interval 
	$\cI = (a,b)$ for $\eta > 0$.

\begin{proposition}[Approximation of optimal microstructures]	\label{prop:approx_optmicro}
	Let $(\bfmu, \bfnu) \in \bCE^{\cI_\eta}$ 
	be a regular curve of measures 
	sastisfying
	\begin{align*}
		M := \mu_0(\T^d) 
				< \infty 
		\tand 
		A := \bA_{\hom}^{\cI_\eta}(\bfmu, \bfnu) 
				< \infty.
	\end{align*}
	Let $(m_t^z, J_t^z)_{t \in \cI, z \in \Z_\eps^d}
	\subseteq \R_+^{\cX} \times \R_a^{\cE}$
	be a measurable family of optimal microstructures associated to $(\mu_t, \nu_t)_{t \in \cI}$
	and consider their gluing
	$(\hat m_t, \hat J_t)_{t \in \cI} 
		\subseteq \R_+^{\cX_\eps} \times \R_a^{\cE_\eps}$. 
	Then, for every $\eta' > 0$, 
	there exists $\eps_0 > 0$ such that the following holds
	for all $0 < \eps \leq \eps_0$:
	there exists a solution 
		$(\bfm^*, \bfJ^*) \in \cCE_\eps^\cI$ 
	satisfying the bounds
		\begin{subequations}
			\begin{align}
					\label{eq:proposition_approx_measure}
				& \textrm{(measure approximation)} & &
				\| 
					\iota_\eps ( \hat \bfm - \bfm^* ) 
				\|_{\KR(
						\overline{\cI} \times \Td
						)} 
			 	\leq \eta', 
				\\
				%%%%%%%%%%%%%%%%%%%%%%%%%%%%%%%%%%%%%	
					\label{eq:proposition_approx_energy}  
				& \textrm{(action approximation)} & &
				\cA_\eps^\cI(\bfm^*, \bfJ^*)
				\leq 
					\bA_\hom^\cI(\bfmu, \bfnu)
				 +\eta' + C \eps,
			\end{align}
			 where $C < \infty$ depends on $M$, $A$, $|\cI|$, and $\eta'$, but not on $\eps$.
		\end{subequations}
\end{proposition}

\begin{remark}
It is also true that
\begin{align*}
	\cA_\eps^\cI(\bfm^*, \bfJ^*)
	\leq 
	\cA_\eps^\cI(\hat\bfm, \hat\bfJ)
	 + \eta' + C \eps,
\end{align*}
but this information is not ``useful'', 
as we do not expect to be able to control 
$\cA_\eps^\cI(\hat\bfm, \hat\bfJ)$ in terms of
$\bA_\hom^\cI(\bfmu, \bfnu)$;
see also Remark \ref{rem:lack_regularity}.
\end{remark}

The proof consists of four steps: the first one is to consider optimal microstructures associated with $(\bfmu,\bfnu)$ on every scale $\eps>0$ and glue them together to obtain a discrete curves $(\bfm^*, \bfJ^*)$ (we omit the $\eps$-dependence for simplicity). The second step is the space-time regularisation of such measures in the same spirit as done in the proof of Proposition \ref{prop:regularisation_discrete}. 
Subsequently, we aim at finding suitable correctors in order to obtain a solution to the continuity equation and thus a discrete competitor (in the definition of $\cA_\eps$). Finally, the energy estimates conclude the proof of Proposition \ref{prop:approx_optmicro}.

Let us first discuss the third step, i.e. how to find small correctors for $(\bfm^*, \bfJ^*)$ in order to obtain discrete solutions to $\cCE_\eps^\cI$ which are close to the first ones. Suppose for a moment that $(\bfm^*, \bfJ^*)$ are "regular", as in the outcome of Proposition \ref{prop:regularisation_discrete}. Then the idea is to consider how far they are from solving the continuity equation, i.e. to study the error in the continuity equation
\begin{align*}
	g_t(x) := \partial_t m_t^*(x) + \dive J_t^*(x), \quad x \in \cX_\eps,
\end{align*}
and find suitable (small) correctors $\bm {\tilde J}$ to $\bfJ^*$ in such a way that  $(\bfm^* , \bfJ^* + \tilde{\bfJ}) \in \cCE_\eps^\cI$.

This is based on the next result, which is obtained on the same spirit of Lemma \ref{lemma:bounds_divergence_eq 2} in a non-periodic setting. In this case, we are able to ensure good $\ell^\infty$-bounds and support properties.

\begin{lemma}[Bounds for the divergence equation]
	\label{lemma:bounds_divergence_eq}
	Let $g : \cX_\eps \to \R$ with 
	$\sum_{x \in \cX_\eps} g(x) = 0$. 
	There exists a vector field 
	$J : \cE_\eps \to \R$ 
	such that 
	\begin{gather}
		\dive J = g 
		\tand
		\label{eq:linf_l1_lemmaconteq_eps}
		\| J \|_{\ell^\infty(\cE_\eps)} 	
		\leq 
		\tfrac12 \| g \|_{\ell^1(\cX_\eps)}.
	\end{gather}
	Moreover, $ \supp V \subseteq \conv \supp g+ B_{C\eps}$ with $C$ depending only on $\cX$.
\end{lemma}

\begin{proof}
	Let $g_+$ be the positive part of $g$, and let $g_-$ be the negative part. 
	By assumption, these functions have the same $\ell^1$-norm 		$N := \|g_-\|_{\ell^1(\cX_\eps)}
	= \|g_+\|_{\ell^1(\cX_\eps)}
	$.
	Let $\Gamma$ be an arbitrary coupling between the discrete probability measures $g_-/N$ and $g_+/N$.
	
	For any $x, y \in \supp g$: take an arbitrary path 
	$P_{xy}$
	connecting these two points. 
	Let $J_xy$ be the unit flux field constructed in Definition \ref{def:vectoralongpath}.
	Then the vector field
	$J := \sum_{x, y} \Gamma(x,y) J_{xy}$
	has the desired properties.
\end{proof}

\begin{remark}[Measurability]
	\label{rem:measurability}
	It is clear from the previous proof that one can choose the vector field $J:\cE_\eps \to \R$ in such a way that the function $g \mapsto J$ is a measurable map.
\end{remark}

The plan is to apply Lemma \ref{lemma:bounds_divergence_eq} to a suitable localisation of $g_t$, in each cube $Q_\eps^z$, for every $z \in \Z_\eps^d$. Precisely, the goal is to find $g_t(z; \cdot)$ for every $z\in \Z_\eps^d$ such that
\begin{gather}	\label{eq:localisation}
	\sum_{z\in\Z^d_\eps}g_t(z;x)=g_t(x), \quad 	\sum_{x\in\cX_\eps} g_t(z; x) =0, 
\end{gather}
which is small on the right scale, meaning
\begin{align}	\label{eq:support_g}
	\supp g_t(z;\cdot) \subset B_\infty(z,R \eps), \quad  \|g_t(z;\cdot)\|_{\infty}\leq C\eps^d \,  .
\end{align}

\begin{remark}
	Note that 
		$\sum_{x \in \cX_\eps} g_t(x) = 0$ 
	for all $t\in\cI$, 
	since $\bfm^*$ has constant mass in time and $\bfJ^*$ is skew-symmetric. 
	However, an application of Lemma \ref{lemma:bounds_divergence_eq} without localisation would not ensure a uniform bound on the corrector field, as we are not able to control the $\ell^1$-norm of $g_t$ \emph{a priori}.
\end{remark}

\begin{remark}
	A seemingly natural attempt would be to define 	
		$g_t(z;x):= g_t(x) \mathds{1}_{\{ z \}}(x_\sz)$. 
	However, this choice is not of zero-mass, due to the flow of mass across the boundary of the cubes.
\end{remark}

Recall that we use the notation
	$(\bfr, \bfu) \in \CE_{\eps,d}^\cI$
to denote solutions to the continuity equation on $\Z_\eps^d$ in the sense of Definition \ref{def:cont_eq_Zepsd}.
We shall later apply Lemma \ref{lemma:localisation} to the pair $(\bfr, \bfu) = (\tP_\eps \bfmu, \tP_\eps \bfnu) \in \CE_{\eps,d}^\cI$, thanks to Lemma \ref{lemma:disc_cont_Zepsd}.

The notion of \textit{shortest path} in the next definition refers to the $\ell_1$-distance on $\Z_\eps^d$.

\begin{definition}\label{defi: signs}
For all $z',z''\in \Z_\eps^d$, we choose simultaneously a shortest path $p(z',z'') := (z_0,\ldots,z_N)$ of nearest neighbors in $\Z_\eps^d$ connecting $z_0 = z'$ to $z_N = z''$ such that $p(z' + \tilde z, z'' + \tilde z) =  p(z',z'') + \tilde z$ for all $\tilde z\in \Z_\eps^d$. Then define for $z,z',z''\in \Z_\eps^d$ and $i=1,\ldots,d$ the signs $\sigma_i^{z;z',z''}\in \{-1,0,1\}$ as
\begin{align*}
\sigma_i^{z;z',z''} :=
\begin{cases}
-1 & \text{if } (z_{k-1},z_k) = (z,z-e_i)\text{ for some $k$ within }p(z',z''),\\
1  & \text{if }  (z_{k-1},z_k) = (z-e_i,z)\text{ for some $k$ within }p(z',z''),\\
0  &  \text{otherwise.}
\end{cases}
\end{align*}
\end{definition}

Note that since the paths $p(z',z'')$ are simple, each pair of nearest neighbours appears at most once in any order, so that $\sigma_i^{z; z', z''}$ is well-defined.

It follows readily from Definition \ref{defi: signs} that
\begin{align}
	\label{eq:path-consequence}
	\sum_{z\in \Z_\eps^d} \sigma_i^{z; z', z''} 
	= (z'' - z') \cdot e_i
\end{align}
for all $z', z''\in \Z_\eps^d$ and $i = 1,\ldots,d$.

\begin{remark}
A canonical choice of the paths $p(z',z'')$ is to interpolate first between $z'_1\in  \Z_\eps^1$ and $z''_1\in \Z_\eps^1$ one step at a time, then between $z'_2$ and $z''_2$, and so on. The precise choice of path is irrelevant to our analysis as long as paths are short and satisfy $p(z' + \tilde z, z'' + \tilde z) =  p(z',z'') + \tilde z$. Since the paths are invariant under translations, so are the signs, i.e. 
\begin{equation}\label{eq: sigma shift}
\sigma_i^{z;z' + \tilde z, z'' + \tilde z} = \sigma_i^{z-\tilde z; z', z''}
\end{equation}
for all $z,\tilde z, z',z''\in \Z_\eps^d$, which is used in the prof of Lemma \ref{lemma:localisation} below.
\end{remark}

Lemma \ref{lemma:localisation} shows that if we start from a solution to the continuity equation $(\bfmu,\bfnu) \in \bCE^\cI$ and consider an admissible microstructure $(\bfm, \bfJ)=(m_t^z, J_t^z)_{t \in \cI, z \in \Z_\eps^d}$ associated to $(\tP_\eps \bfmu, \tP_\eps \bfnu)$, then it is possible to localise the error in the continuity equation arising from the gluing $(\cG_\eps \bfM, \cG_\eps \bfU)$ as in \eqref{eq:localisation}. 

\begin{lemma}[Localisation of the error to $\cCE_\eps^\cI$]
		\label{lemma:localisation}
	Let $(\bfr, \bfu) \in \CE_{\eps,d}^\cI$ 
	and suppose that
		$m_t := (m_t^z)_{z \in \Z_\eps^d} \subseteq \R_+^\cX$
	and 
		$J_t := (J_t^z)_{z \in \Z_\eps^d} \subseteq \R_a^\cE$
	satisfy
	\begin{align*}
		( m_t^z , J_t^z )
		\in \Rep 
		\big( r_t(z), u_t(z) \big)
	\end{align*}
	for every $t \in \cI$ and $z \in \Z_\eps^d$. 
	Consider their gluings 
		$\hat m_t := \cG_\eps m_t$ 
	and
		$\hat J_t := \cG_\eps J_t$ 
	and define, 
	for $z \in \Z_\eps^d$ and $x \in \cX_\eps$,
	\begin{align}	
	\label{eq:error_CE}
		g_t(x) & := \partial_t \hat m_t(x) 
					+ \dive \hat J_t(x), 
	\\
	\label{eq:def_g_error}
		g_t(z; x) & := \partial_t \hat m_t(x) 
			\one_{\{ z \}}(x_\sz) 
			+ \frac12 \sum_{y\sim x}
				\sum_{i=1}^d \sigma_i^{z; x_\sz, y_\sz}
				\Big(
					\tilde J_t(z; x, y) 
						-
					\tilde J_t(z - e_i; x, y)
				\Big),
	\end{align}
where 
	$\tilde J_t(z; \cdot) : \cE_\eps \to \R$ 
is the $\T_\eps^d$-periodic map satisfying 
	$\tilde J_t\big(z; T_\eps^0(x'), T_\eps^0(y')\big) 
		= J_t^z(x', y')$ 
for all $(x', y') \in \cE$.
Then the following statements hold for every $t \in \cI$:
\begin{enumerate}[(i)]
	\item $g_t(z;x)$ is a localisation 
	of the error $g_t(x)$ of $(\hat m, \hat J)$ from solving $\cCE_\eps^\cI$, i.e., 
	\begin{align*}
		\sum_{z\in\Z^d_\eps} g_t(z;x) = g_t(x) 
		\quad \text{for all } x \in \cX_\eps.
	\end{align*}
	\item Each localised error $g_t(z;\cdot)$ has zero mass, i.e.,
	\begin{align*}
		\sum_{x \in \cX_\eps} g_t(z; x) = 0
		\quad \text{for all } z \in \Z_\eps^d.
	\end{align*}
\end{enumerate}

\end{lemma}
\begin{proof}
$(i)$: \
 For $(x,y) \in \cE_\eps$, 
 consider the path $p(x_\sz, y_\sz) = (z_0, \ldots, z_N)$ constructed in Definition \ref{defi: signs}. 
 For all $t \in \cI$ we have
\begin{align*}
	& 
	\sum_{z \in \Z^d_\eps} 
	\sum_{i=1}^d  
		\sigma_i^{z; x_\sz, y_\sz}
		\Big(
			\tilde J_t(z;x,y) 
			- 
			\tilde J_t(z- e_i;x,y)
		\Big) \\
	& =  \sum_{k=1}^{N} 	
		\Big( 
			\tilde J_t(z_k; x,y) - \tilde J_t(z_{k-1}; x,y) 
		\Big)
	= \tilde J_t(y_\sz; x,y) - \tilde J_t(x_\sz; x,y) .
\end{align*}
Summation over all neighbours of $x\in \cX_\eps$ yields, 
for all $t\in \cI$,
\begin{align*}
	 \sum_{z\in\Z^d_\eps}g_t(z;x)
	 & =	\partial_t m_t(x) + \frac12 \sum_{y\sim x} 
			\sum_{z \in \Z^d_\eps} 
			\sum_{i = 1}^d 
				\sigma_i^{z; x_\sz, y_\sz}	
					\Big(
						\tilde J_t(z; x, y) - 
						\tilde J_t(z-e_i; x, y)
					\Big)
		\\
		& = \partial_t m_t(x) 
			+ \frac12 \sum_{y\sim x}
				\Big(
					\tilde J_t(y_\sz;x, y) 
				- 	\tilde J_t(x_\sz;x, y)
				\Big)
		\\
		& = \partial_t m_t(x) 
			+ \frac12 \sum_{y\sim x}
			 \Big( 
				 \tilde J_t(y_\sz; x, y) 
			+ 	 \tilde J_t(x_\sz; x, y)
			 \Big) 
		= g_t(x),
\end{align*}
where we used the $\Z^d$-periodicity of $(\cX, \cE)$ and the vanishing divergence of $J_t^{x_\sz}$.

\medskip
\noindent
$(ii)$: \
Fix $z \in \Z_\eps^d$ and $t \in \cI$. 
Using the periodicity of $\tilde J_t(z; \cdot)$, 
the identity \eqref{eq: sigma shift}, 
the group structure of $\Z_\eps^d$, 
the relation between $\tilde J$ and $J$, 
the fact that $J_t^z \in \Rep\big( u_t(z) \big)$, 
and the identity \eqref{eq:path-consequence}, 
we obtain
\begin{align*}
	& \sum_{(x,y) \in \cE_\eps} 
		\sum_{i=1}^d \sigma_i^{z;x_\sz,y_\sz} \left(\tilde J_t(z; x,y)  - \tilde J_t(z-e_i; x,y) \right)\\
			& = \sum_{\substack{ (x,y) \in \cE_\eps 
						\\ x_\sz = z}} 
					\sum_{\tilde z\in \Z_\eps^d} 
						\sum_{i=1}^d \sigma_i^{z;x_\sz + \tilde z,y_\sz + \tilde z} 
							\left(
								\tilde J_t(z; x,y)  - \tilde J_t(z-e_i; x,y) 
							\right) \\
			& = \sum_{\substack{(x,y)\in \cE_\eps\\ x_\sz = z}} 
					\sum_{\tilde z\in \Z_\eps^d} 
						\sum_{i=1}^d \sigma_i^{z - \tilde z;x_\sz,y_\sz} 		
							\left(
								\tilde J_t(z; x,y)  - \tilde J_t(z-e_i; x,y) 
							\right) \\
			& = \sum_{\substack{(x,y) \in \cE_\eps \\ x_\sz = z}}
					\sum_{i=1}^d \left(\tilde J_t(z; x,y)  - \tilde J_t(z-e_i; x,y) \right) 
						\left(
							\sum_{\tilde z\in \Z_\eps^d} \sigma_i^{\tilde z;x_\sz,y_\sz}
						\right) \\
			& =  \sum_{(x',y') \in \cE^Q} 
					\sum_{i=1}^d
						\Big(
							J_t^z(x',y') - 
							J_t^{z-e_i}(x',y')
						\Big) (y'_\sz-x'_\sz) \cdot e_i \\
			& = 2 \sum_{i=1}^d \big(u_t(z) - u_t(z-e_i)\big) \cdot e_i.
\end{align*}
By definition of $g_t(z;\cdot)$ we obtain
\begin{align*}
	\sum_{x\in\cX_\eps} g_t(z;x) 
		& =  
			\sum_{\substack{x\in \cX_\eps \\ x _\sz = z}} 
				\partial_t m_t(x) + 
				\frac12 \sum_{i=1}^d
					\sum_{(x,y)\in \cE_\eps} 
			\sigma_i^{z;x_\sz,y_\sz} 
				\Big(\tilde J_t(z; x,y) 
				- \tilde J_t(z-e_i;x,y)
			\Big)
		\\
		& = \partial_t r_t(z) 
			+ \sum_{i=1}^d \big(u_t(z)-u_t(z- e_i)\big)\cdot e_i 
		=  0, 
\end{align*}
where we used that $m_t^z \in \Rep\big(r_t(z)\big)$ and eventually that $(\bfr, \bfu) \in \CE_{\eps,d}^\cI$.
\end{proof}

Now we are ready to prove Proposition \ref{prop:approx_optmicro}.
\begin{proof}[Proof of Proposition \ref{prop:approx_optmicro}]
	The proof consists of four steps. For simplicity: $\cI:=\cI_\eta$.
	
\medskip
\noindent
\textit{Step 1: Regularisation}. \
Recall the operators $R_\delta$, $S_\lambda$, and $T_\tau$ as defined in Section \ref{sec:disc-reg}. We define
\begin{align*}
	\bfm^*
		:= \Big(
				R_\delta 
				\circ S_\lambda 
				\circ T_\tau 
			\Big) 
		\hat \bfm 
	\tand		
		\bar \bfJ^*
		:= \Big(
				R_\delta 
				\circ S_\lambda 
				\circ T_\tau 
			\Big) 
	\hat{\bfJ},
\end{align*}
where $\delta, \lambda>0$, $0<\tau <\eta$ will be chosen sufficiently small, depending on the desired accuracy $\eta' >0$. Due to special linear structure of the gluing operator $\cG_\eps$, it is clear that
\begin{align*}
	\bfm^* = \cG_\eps \overline \bfm \tand \quad \bar  \bfJ^* = \cG_\eps \bar \bfJ, 
\end{align*}
for some $\big( \overline \bfm, \bar \bfJ \big) = (\overline m_t^z, \bar J_t^z)_{t \in \cI, z \in \Z_\eps^d}$. More precisely, they correspond to the regularised version of the measures $(m_t^z, J_t^z)_{t \in \cI, z \in \Z_\eps^d}$ with respect to the graph structure of $\Z_\eps^d$. In particular, an application\footnote{To be precise, this is an application of these lemmas to the case of $V:= \{v\}$, thus $\cX_\eps \simeq \Z_\eps^d$.} of Lemma \ref{lemma:L1-infty_estimates_Peps}, Lemma \ref{lem:reg-space}, and Lemma \ref{lem:reg-time} yields
\begin{align}	\label{eq:proof_approximation_regularity}
\begin{gathered}
	\sup_{t \in \cI}
		\big\|  \overline m_t^{\cdot + z}
	- \overline m_t
		\big\|_{\ell^\infty(\Z_\eps^d \times \cX)}
	+ \eps
		\big\|  \bar J_t^{\cdot + z}		
	- \bar J_t
		\big\|_{\ell^\infty(\Z_\eps^d \times \cE)}
	\leq C |z| \eps^{d+1}, \\
	\sup_{t \in \cI}
		\big\|  
			\partial_t \overline m_t
		\big\|_{\ell^\infty(\Z_\eps^d \times \cX)}
	\leq  C \eps^d,
\end{gathered}
\end{align}
for any $z \in \Z_\eps^d$, as well as the domain regularity
\begin{align}	\label{eq:proof_approx_compact_microstr}
	\bigg\{ \bigg( \frac{\overline m_t^z}{\eps^d}, \frac{\bar J_t^z}{\eps^{d-1}} \bigg) \suchthat z \in \Z_\eps^d,  \, t \in \cI  \bigg\} \subset K \Subset    (\Dom F)^\circ,
\end{align}
for a constant $C$ and a compact set $K$ depending only on $M$, $A$, $\delta$, $\lambda$, and $\tau$. We can then apply Lemma \ref{lemma:energy_est_glued_measures_regular} and deduce that for every $t \in \cI$, $\eps \leq \eps_0$ (depending on $K$ and $F$),
\begin{align}	\label{eq:proof_approx_energy_regularisation}
	\cF_\eps 
		\big(
			m_t^*, \bar J_t^*	
		\big) \leq 
			\sum_{z \in \Z_\eps^d}  
		\eps^d F \bigg(  
			\frac{\overline m_t^z}{\eps^d} , \frac{\bar J_t^z}{\eps^{d-1}}
		\bigg) + c \eps,
\end{align}
for a $c \in \R^+$ depending on the same set of parameters (via $C$ and $\Lip(F;K)$) and $R_1$.

\medskip
\noindent
\textit{Step 2: Construction of a solution to $\cCE_\eps^\cI$}. \
From now on, the constants $C$ appearing in the estimates might change line by line, but it always depends on the same set of parameters as the constant $C$ in Step 1, and possibly on the size of the time interval $|\cI|$.

The next step is to find a quantitative small corrector $\bfV$ in such a way that $(\bfm^*, \bar \bfJ^* + \bfV) \in \cCE_\eps^\cI$. To do so, we observe that by construction we have for every $t \in \cI$
\begin{align*}
	\Big(
		\overline m_t^z , \bar J_t^z
	\Big)
		\in  \Rep 
	\Big(
		r_t^*(z) , u^*_t(z)
	\Big),
\end{align*}
where $(\bfr^*, \bfu^*) \in \CE_{\eps,d}^\cI$ (by the linearity of equation \eqref{eq:cont_eq_Zepsd}). Consider the corresponding error functions, for $(x,y) \in \cE_\eps$, $t \in \cI$, $z \in \Z_\eps^d$ given by \eqref{eq:error_CE} and \eqref{eq:def_g_error},
\begin{align*}	
	g_t(x) &:= \partial_t m_t^*(x) + \dive \bar J_t^*(x), 		\\
	g_t(z;x)&:=\partial_t m_t^*(x) \mathds{1}_{\{ x_\sz = z \}}(x)  + \frac12 \sum_{y\sim x}\sum_{i=1}^d \sigma_i^{z;x_\sz,y_\sz}(\tilde J(z;x,y)-\tilde J(z-e_i;x,y)),
\end{align*}
where $\tilde J(z;\cdot):\cE_\eps \to \R$ is the $\T_\eps^d$-periodic map satisfying $\tilde J(z;T_\eps^0(x'),T_\eps^0(y')) = \bar J_t^z(x',y')$, for any $(x',y')\in \cE$. Thanks to Lemma \ref{lemma:localisation}, we know that
\begin{align*}
	\sum_{x\in\cX_\eps} g_t(z; x) = 0, 
		\quad
	\sum_{z' \in \Z_\eps^d} g_t(z';x) = g_t(x), \quad \forall x \in \cX_\eps, \, z \in \Z_\eps^d.
\end{align*}
Moreover, from the regularity estimates  \eqref{eq:proof_approximation_regularity} and the local finiteness of the graph $(\cX,\cE)$, we infer for every $z \in \Z_\eps^d$
\begin{align}	
	\left\|   g_t(z;\cdot)   \right\|_{\ell^\infty(\cX_\eps)} \leq C \eps^d, \quad \supp g_t(z; \cdot) \subset \{ x \in \cX_\eps \suchthat \| x_\sz - z \|_{\ell^\infty(\Z_\eps^d)}  \leq C' \},
\end{align}
where $C'$ only depends on $(\cX,\cE)$. Hence, as an application of Lemma \ref{lemma:bounds_divergence_eq}, we deduce the existence of corrector vector fields $V_t \in \R_a^{\Z_\eps^d \times \cE_\eps}$ such that 
\begin{align}	\label{eq:proof_apporx_corrector}
\begin{gathered}
	\dive V_t(z;\cdot) = g_t(z;\cdot)
		\ , 
			\quad  
		\supp V_t(z;\cdot)  \subset \{ (x,y) \in \cE_\eps \suchthat \| x_\sz - z \|_{\ell^\infty(\Z_\eps^d)}  \leq \tilde C' \}, \\
	\| V_t(z;\cdot) \|_{\ell^\infty(\cE_\eps)} 	
		\leq 
	\tfrac12 \| g_t(z;\cdot) \|_{\ell^1(\cX_\eps)} \leq C \eps^d,
\end{gathered}
\end{align}
for every $t \in \cI$, $z \in \Z_\eps^d$. The existence of a measurable (in $t \in \cI$ and $z \in \Z_\eps^d$) map $V_t(z;\cdot)$ follows from the measurability of $g_t(z;\cdot)$ and Remark \ref{rem:measurability}.

We then define $\bfV: \cI \to \R_a^{\cE_\eps}$ and $\bfJ^*: \cI \to \R_a^{\cE_\eps}$ as
\begin{align*}
	\bfV := \sum_{z \in \Z_\eps^d}\bfV(z;\cdot), \quad \bfJ^*:= \bar \bfJ^* + \bfV, 
\end{align*}
and obtain a solution to the discrete continuity equation $(\bfm^*, \bfJ^*) \in \cCE_\eps^\cI$.

\medskip
\noindent
\textit{Step 3: Energy estimates}. \
The locality property $\eqref{eq:proof_apporx_corrector}$ of $V_t(z;\cdot)$ and local finiteness of the graph $(\cX,\cE)$ allow us to deduce the same uniform estimates on the global corrector as well. Indeed for every $t\in \cI$, $x \in \cX_\eps$ we have
\begin{align*}
	V_t(x,y) := \sum_{z \in B_\infty(x_\sz ; \tilde C') } V(z;x,y), 
		\quad 
	B_\infty(x_\sz ; \tilde C') := 
	\left\{
		z \in \Z_\eps^d \suchthat \| z - x_\sz \|_{\ell^\infty(\Z_\eps^d)}  \leq \tilde C'
	\right\},
\end{align*}
and hence from the estimate \eqref{eq:proof_apporx_corrector} we also deduce $\| \bfV \|_{\ell^\infty(\cI \times \cE_\eps)}  \leq C \eps^d$.

Since \eqref{eq:proof_approx_compact_microstr} implies that 
$\displaystyle 
\bigg(\frac{\tau_\eps^z m_t^*}
{\eps^d}
,
\frac{\tau_\eps^z \bar J_t^*}
{\eps^{d-1}}
\bigg)
\in K$
, it then follows that 
$\displaystyle 
\bigg(\frac{\tau_\eps^z m_t^*}
{\eps^d}
,
\frac{\tau_\eps^z J_t^*}
{\eps^{d-1}}
\bigg)
\in K'$
for $0<\eps \leq \eps_0$ sufficiently small, where $\eps_0$ depends on $K$ and $C$.
Here $K'$ is a compact set, possibly slightly larger than $K$, contained in $\Dom(F)^\circ$. 

Therefore, we can estimate the energy 
\begin{align*}
	\sup_{t \in \cI} \sup_{z \in \Z_\eps^d}
	\left| 
		F 
		\bigg(\frac{\tau_\eps^z m_t^*}
			{\eps^d}
				,
			\frac{\tau_\eps^z \bar J_t^*}
			{\eps^{d-1}}
		\bigg)
	-
		F 
			\bigg(\frac{\tau_\eps^z m_t^*}
				{\eps^d}
				,		
				\frac{\tau_\eps^z J_t^*}
				{\eps^{d-1}}
			\bigg)
	\right| 
		\leq
	\Lip(F;K') \frac{1}{\eps^{d-1}} \| \bfV \|_{\ell^\infty(\cI \times \cE_\eps)} 
		\leq
	C \eps,
\end{align*}
and hence
$
	\cA_\eps^\cI 
		\big(
			\bfm^*, \bfJ^*	
		\big) 
	\leq 
	\cA_\eps ^\cI
		\big(
			\bfm^*, \bar \bfJ^*
		\big) 
	+
		C  \eps
$. Together with \eqref{eq:proof_approx_energy_regularisation}, this yields
\begin{align*}	
%\label{eq:proof_approx_energy_regularisation_2}
	\cA_\eps^\cI
		\big(
			\bfm^*, \bfJ^*	
		\big) 
	\leq 
		\int_\cI
		\sum_{z \in \Z_\eps^d}  
			\eps^d F \bigg(  
					\frac{\overline m_t^z}{\eps^d} , \frac{\bar J_t^z}{\eps^{d-1}}
				\bigg) \dd t + C  \eps.
\end{align*}
Finally, to control the action of the regularised microstructures $(\bar \bfm, \bar \bfJ)$, we take advantage (as in \eqref{eq:energy}) of Lemma \ref{lem:reg-energy}, Lemma \ref{lem:reg-space} $(i)$, and Lemma \ref{lem:reg-time} $(i)$ to obtain\footnote{As before, it's an application of these lemmas on $\Z_\eps^d$ (corresponding to $\V = \{v\}$).} 
\begin{align*}
	\int_\cI
		\sum_{z \in \Z_\eps^d}  
			\eps^d F 
		\bigg(  
			\frac{\overline m_t^z}{\eps^d} , \frac{\bar J_t^z}{\eps^{d-1}}
		\bigg) \dd t 
	&\leq
	\int_\cI 
	\fint_{t- \tau}^{t+ \tau} 
		\sum_{z \in \Z_\eps^d} 
			\eps^d F 
		\left(  
			\frac{m_s^z}{\eps^d} , \frac{ J_s^z}{\eps^{d-1}}
		\right) \dd s \dd t 
	+ \delta |\cI| F(m^\circ, J^\circ) \\
	&\leq 
	\int_\cI 
	\fint_{t- \tau}^{t+ \tau} 
		\bF_\hom(\mu_s,\nu_s) \dd s \dd t 
	+ \delta |\cI| F(m^\circ, J^\circ) + c' \eps \\
	&\leq 
	\int_\cI 
		\bF_\hom(\mu_t,\nu_t) \dd t 
	+ \delta |\cI| F(m^\circ, J^\circ) + c' (\eps + \tau),
\end{align*}
for a $c' < \infty$, where at last we used Proposition \ref{prop:energy_optmicro} and that $f_\hom$ is Lipschitz on $\tilde K$. 

For every given $\eta'>0$, the action bound \eqref{eq:proposition_approx_energy} then follows choosing $\tau,\delta>0$ small enough.

\medskip
\noindent
\textit{Step 4: Measures comparison.} \
We have seen in \eqref{eq:W-bound} that Lemma \ref{lem:W1-bounds} implies
\begin{align*}
	\|
		\iota_\eps \bfm^* - 
		\iota_\eps \hat \bfm
	\|_{\KR([0,T] \times \Td)}
	 \lesssim
	M( \tau + \sqrt{\lambda} + \delta )
	+ m^\circ(\XQ) \delta,
\end{align*}
where we also used that mass preservation of the gluing operator, see Remark \ref{rem:mass_preservation_tP*}. For every $\eta'>0$, the distance bound \eqref{eq:proposition_approx_measure} can be then obtained choosing $\tau$, $\lambda$, $\delta$ sufficiently small.
\end{proof}

\subsection{Proof of the upper bound}	\label{sec:upperbound}

This subsection is devoted to the proof of the limsup inequality in Theorem \ref{thm:main}. 
First we formulate the existence of a recovery sequence in the smooth case.
\begin{proposition}[Existence of a recovery sequence, smooth case]	\label{prop:limsup_smooth}
	 Fix $\cI = (a,b)$, $a<b$, $\eta >0$, and set $\cI_\eta:= (a-\eta, b+\eta)$. Let $(\bfmu, \bfnu) \in \bCE^{\cI_\eta}$ be a solution to the continuity equation with smooth densities $(\rho_t, j_t)_{t\in \cI_\eta}$ and such that 
	\begin{align}	\label{eq:assumptions_smooth_upperb}
		\bA_\hom^{\cI_\eta}(\bfmu, \bfnu) < \infty
		\tand
		 \Big\{
			\big(\rho_t(x),j_t(x)\big) 
				\suchthat (t,x) \in \cI_\eta \times \Td 
		\Big\} 
			\Subset  \Dom(f_{\hom})^\circ.
	\end{align}
	Then there exists a sequence of curves $(m_t^\eps)_{t \in \overline{\cI}} \subseteq \Meps$ such that $\iota_\eps \bfm^\eps \to \bfmu|_{\overline{\cI}\times \Td}$ weakly in $\cM_+(\overline{\cI}\times \Td)$ as $\eps \to 0$ and 
	\begin{align}	\label{eq:upper_bound_smooth}
		\limsup_{\eps \to 0} 
			\cA_\eps^\cI (\bfm^\eps) 
		\leq 	
			\bA_\hom^{\cI_\eta}(\bfmu, \bfnu)
				 + C \eta |\cI| \big(\mu_0(\Td) + 1\big), 
	\end{align} 
	for some $C < \infty$. 
\end{proposition}
\begin{proof}
	We write $\KR_\cI:= \KR(\overline{\cI} \times \Td)$. Let $(\bfmu,\bfnu)\in \bCE^{\cI_\eta}$ be smooth curves of measures satisfying the assumptions  \eqref{eq:assumptions_smooth_upperb}. Let  $(\hat \bfm , \hat \bfJ)$ be the gluing of a measurable family of optimal microstructure associated with $(\bfmu, \bfnu)$, for every $\eps>0$. For every $\eta'>0$, Proposition \ref{prop:approx_optmicro} yields the existence of $(\bfm^{\eta'}, \bfJ^{\eta'}) \in \cCE_\eps^\cI$,  a constant $C_{\eta'}$, and $\eps_0=\eps_0(\eta')$ depending on $\eta'$ such that 
	\begin{align*}
		\| \iota_\eps (\bfm^{\eta'} - \hat \bfm ) 
		\|_{\KR_\cI} 
			\leq \eta',
			\quad
		\cA_\eps(\bfm^{\eta'}, \bfJ^{\eta'})
			\leq 
		\bA_{\hom}(\bfmu, \bfnu)  + \eta' + \eps C_{\eta'}, 
	\end{align*}
	for every $\eps \leq \eps_0$.
	
	Using Remark \eqref{rem:mass_preservation_tP*}, in particular \eqref{eq:rem_KR_iotaP*}, and that $(\bfm^{\eta'}, \bfJ^{\eta'}) \in \cCE_\eps^\cI$, we infer
	\begin{align*}
		\| \iota_\eps (\bfm^{\eta'}) - \bfmu  
		\|_{\KR_\cI} 
			\leq \eta' + \bfmu(\overline{\cI} \times \T^d) \eps^d,
			\quad
		\cA_\eps(\bfm^{\eta'})
			\leq 
		\bA_{\hom}(\bfmu, \bfnu)  + \eta' + \eps C_{\eta'}.
	\end{align*}
	for every $\eps \leq \eps_0$. Therefore, we can find a diagunal sequence $\eta'=\eta'(\eps) \to 0$ as $\eps \to 0$ such that, if we set $\bfm^\eps:= \bfm^{\eta'(\eps)}$, we obtain
	\begin{gather*}
		\lim_{\eps \to 0} 
			\| \iota_\eps (\bfm^\eps) - \bfmu  
			\|_{\KR_\cI}  = 0, 
		\\
		\limsup_{\eps \to 0} \cA_\eps^\cI(\bfm^\eps) 
			\leq 
		\bA_{\hom}^\cI(\bfmu, \bfnu) 
			\leq 
		\bA_{\hom}^{\cI_\eta}(\bfmu, \bfnu) + C \eta |\cI|  (\mu_0(\T^d) +1)
			,
	\end{gather*}
	where at last we used the growth condition \eqref{eq:growth_f}.		
\end{proof}

In order to apply Proposition \ref{prop:limsup_smooth} for the existence of the recovery sequence in Theorem \ref{thm:main} we prove that the set of solutions to the continuity equation \eqref{eq:def_CE} with smooth densities are dense-in-energy for $\bA_\hom^\cI$.

\begin{definition}[Affine change of variable in time]
	Fix $\cI = (a,b)$. For any $\eta >0$, we consider the unique bijective increasing affine map $S^\eta:\cI \to (a-2\eta, b + 2\eta)$. For every interval $\tilde \cI \subseteq \cI$ and every vector-valued measure $\bfxi \in \cM^n(\tilde \cI \times \T^d)$, $n \in \N$,  we define the changed-variable measure
	\begin{gather}	\label{eq:def_mueta}
		\begin{gathered}
			\emph S^\eta[\bfxi] \in \cM^n(S^\eta(\tilde \cI) \times \Td), 
			\quad 
			\emph S^\eta[\bfxi]:= \frac{|\cI|+4\eta}{|\cI|}\big( S^\eta, \id \big)_{\#} \bfxi. 
		\end{gathered}
	\end{gather} 
\end{definition}

\begin{remark}[Properties of $\text S^\eta$]
	The scaling factor of $S^\eta[\bfxi]$ is chosen so that if $\bfxi \ll \Leb^{d+1}$, then $\text S^\eta[\bfxi] \ll \Leb^{d+1}$ and we have for $(t,x) \in S^\eta(\tilde \cI) \times \Td$ the equality of densities
	\begin{align}	\label{eq:densities_mu-j_eta}
		\frac{ \dd \text S^\eta[\bfxi] }{ \dd \Leb^{d+1} }(t,x) 
		= \frac{ \dd \bfxi }{ \dd \Leb^{d+1} }((S^\eta)^{-1}(t),x).
	\end{align} 
	Moreover, if $(\bfmu,\bfnu) \in \bCE^\cI$ then $ \big( \frac{|\cI|+4\eta}{|\cI|} \text S^\eta[\bfmu], \text S^\eta[\bfnu] \big) \in \bCE^{S^\eta(\cI)}$. 
\end{remark}

We are ready to state and prove the last result of this section. 

\begin{proposition}[Smooth approximation of finite action solutions to $\bCE^\cI$]
	\label{prop:density}
	Fix $\cI:= (a,b)$ and fix $(\bfmu,\bfnu)$ $\in \bCE^\cI$ with $\bA_\hom(\bfmu,\bfnu)<\infty$. Then there exists a sequence $\{ \eta_k\}_k \subset \R^+$ such that $\eta_k \to 0$ as $k \to \infty$ and measures $(\bfmu^k,\bfnu^k) \in \bCE^{\cI_k}$ for $\cI_k := (a-\eta_k, b+ \eta_k)$ so that as $k \to \infty$
	\begin{gather}	
		\label{eq:lemma_convergence_mukjk}
		(\bfmu^k,\bfnu^k) 
		\to
		 (\bfmu, \bfnu) 
			\;
		\text{weakly in } 
			\cM_+\big(\cI\times \Td \big) \times \cM^d \big(\cI\times \Td\big), \\
		\label{eq:lemma_smooth_densities}
			\frac{\dd \bfmu^k}{\dd \Leb^{d+1}} \in \cC_b^\infty\big(\cI_k\times \T^d\big), 
				\quad 
			\frac{\dd \bfnu^k}{\dd \Leb^{d+1}} 
				\in
			\cC_b^\infty\big(\cI_k\times \T^d;\R^d\big),
	\end{gather}
	and such that the following action bound holds true:
	\begin{align}	\label{eq:dense_in_energy_smooth}
		\limsup_{k \to\infty} 
		\bA_\hom^{\cI_k}(\bfmu^k,\bfnu^k)
		\leq 
		\bA_\hom^\cI(\bfmu,\bfnu).
	\end{align}
	Moreover, for any given $k\in \N$ we have the inclusion
	\begin{align}	\label{eq:compact_lemma_smooth}
		\Big\{ 
		\Big( 
		\frac{\dd\bfmu^k}{\dd \Leb^{d+1}}(t,x),\frac{\dd \bfnu^k}{\dd \Leb^{d+1}}(t,x) 
		\Big) 
		: (t,x)\in \cI_k\times \T^d 
		\Big\} 
		\Subset
		(\Dom f_\hom)^\circ.
	\end{align}
\end{proposition}

\begin{proof}
	Without loss of generality we can assume $f_\hom \geq 0$, if not we simply consider $g(\rho,j) = f_\hom(\rho,j) + C\rho + C$ for $C \in \R_+$ as in Lemma \ref{lemma: action lsc}.
	 For simplicity, we also assume $\cI:= (0,T)$, the extension to a generic interval is straightforward.
	
	Fix $(\bfmu,\bfnu)$ $\in \bCE^T$ with $\bA_\hom(\bfmu, \bfnu)<\infty$.
	
	\medskip
	\noindent 
	\textit{Step 1: regularisation}. \ 
	The first step is to regularise in time and space. 
	To do so, we consider two sequences of smooth mollifiers $\phi_1^k: \R \to \R_+$, $\phi_2^k: \Td \to \R$ for $k \in \N$ of integral $1$, where $\supp \phi_1^k = [-\alpha_k, \alpha_k]$, $\supp \phi_2^k =B_{\frac1k}(0)\subset \Td$ with $\alpha_k \to 0$ as $k \to \infty$ to be  suitably chosen. We then set $\phi^k:\R \times \Td \to \R_+$ as $\phi^k(t,x) := \phi_1^k(t)\phi_2^k(x)$.
	
	We define space-time regular solutions to the continuity equation as 

	\begin{align*}
		(\tilde \bfmu^k, \tilde \bfnu^k ) 
		&:= \phi^k * (\bfmu, \bfnu) 
		\in \bCE^{(\alpha_k,T-\alpha_k)}
		, \\
		(\hat\bfmu^k, \hat \bfnu^k )	
		&:=\Big(
		\frac{T+4\eta_k}{T}\text S^{\eta_k}[\tilde \bfmu^k], \text S^{\eta_k}[\tilde  \bfnu^k] 
		\Big) \in \bCE^{\cI_k},
	\end{align*}
	where $\cI_k:= S^{\eta_k}\big((\alpha_k, T-\alpha_k)\big)$. Note that the mollified measures are defined only  We choose $\alpha_k := \frac{T\eta_k}{T+4\eta_k}$, so that $\cI_k = (-\eta_k,T+\eta_k)$.
	
	Finally, for $(\rho^\circ,j^\circ)$ as given in \eqref{eq:def_r0j0}, we define
	\begin{gather}
		\begin{gathered}
			(\bfmu^k, \bfnu^k):= (1-\delta_k)  (\hat \bfmu^k, \hat \bfnu^k) + \delta_k (\rho^\circ, j^\circ) \Leb^{d+1} 
			\in \bCE^{\cI_k}	,
		\end{gathered}
	\end{gather}
	for some suitable choice of $\eta_k, \delta_k \to 0$.

	\medskip
	\noindent
	\textit{Step 2: Properties of the regularised measures.} \ 
	First of all, we observe that $(\bfmu^k,\bfnu^k) \ll \Leb^{d+1}$ with smooth densities for every $k \in \N$, so that \eqref{eq:lemma_smooth_densities} is satisfied. Secondly, the convergence \eqref{eq:lemma_convergence_mukjk} easily follows
	by the properties of the mollifiers and the fact that $S^\eta \to \id$ uniformly in $(0,T)$ as $\eta \to 0$. 
	
	Moreover, we note that for $t>0$, using that $\mu_t(\Td)$ is constant on $(0,T)$ one gets
	\begin{align}	\label{eq:proof_uniform_tildemuj}
		\begin{aligned}
			\sup_{t \in (\alpha_k, T-\alpha_k)} \Big\| \frac{ \dd \tilde \mu_t^k }{ \dd x } \Big\|_\infty 
			&\leq \| \phi_2^k \|_\infty\bfmu \big((0,T)\times \Td \big) =: C_k < +\infty, \\
			\Big\| \frac{ \dd \tilde \bfnu^k }{ \dd \Leb^{d+1} } \Big\|_\infty 
			&\leq \| \phi^k \|_\infty  |\bfnu| \big((0,T)\times \Td \big)	< \infty	,
		\end{aligned}
	\end{align}
	and thanks to \eqref{eq:densities_mu-j_eta} an analogous uniform estimate holds true for $(\hat \bfmu^k, \hat \bfnu^k)$ too. We can then apply Lemma \ref{lemma:prop_domf} and find convex compact sets $K_k \subset (\Dom f_\hom )^\circ$ such that $ \displaystyle
	\Big\{\Big( 
	\frac{\dd  \bfmu^k}{\dd \Leb^{d+1}}(\cdot), \frac{\dd  \bfnu^k}{\dd \Leb^{d+1}}(\cdot) 
	\Big)\Big\} \subset K_k$, so that \eqref{eq:compact_lemma_smooth} follows. 
	
	\smallskip
	Additionally, pick $\theta>0$ such that $B^\circ:=B((\rho^\circ, j^\circ), \theta) \subset (\Dom f_\hom)^\circ$. From \eqref{eq:densities_mu-j_eta}, if one sets $S_k:= S^{\eta_k}$, we see that
	\begin{align}	\label{eq:stability_convexcomb}
		\Big( \frac{\dd \bfmu^k}{\dd \Leb^{d+1}},\frac{\dd \bfnu^k}{\dd \Leb^{d+1}}\Big) (t,x)=(1-\delta_k) \Big( \frac{\dd \tilde \bfmu^{k}}{\dd \Leb^{d+1}},\frac{\dd \tilde \bfnu^{k}}{\dd \Leb^{d+1}}\Big) (S_k^{-1}(t),x)
		+\delta_k (\tilde \rho_t^k(x), j^\circ)
	\end{align}
	for $t \in \cI_k$ and $x \in \Td$, where the functions $\tilde \rho ^k$ are given by  
	\begin{align*}
		\tilde \rho_t^k(x) := \rho^\circ +\frac{1-\delta_k}{\delta_k} 2\eta_k \frac{\dd \tilde \bfmu^{k}}{\dd \Leb^{d+1}}(S_k^{-1}(t),x).
	\end{align*}
	We choose $\delta_k$ such that $\theta \delta_k > 2\eta_k C_k$ and from \eqref{eq:proof_uniform_tildemuj} we get that
	\begin{align}	\label{eq:rhotilde_in_B0}
		(\tilde \rho_t^k(x),j^\circ) \in B^\circ 
		, \quad
		\forall t \in \cI_k, \; x \in \Td, \; k \in \N.
	\end{align}
	
	For example we can pick $\eta_k:= (4kC_k)^{-1}$ and $\theta\delta_k = k^{-1}$, both going to zero when $k \to +\infty$.

	\medskip
	\noindent
	\textit{Step 3: action estimation}. \
	As the next step we show that
	\begin{align}	\label{eq:contractivity_convolution_phik}
		\bA_\hom^{(\alpha_k,T-\alpha_k)} \big( \tilde \bfmu^k, \tilde \bfnu^k \big) \leq \bA_\hom^T (\bfmu, \bfnu ), \quad \forall  k \in \N.
	\end{align}
	
	One can prove \eqref{eq:contractivity_convolution_phik} using e.g. the fact \cite{buttazzo1991} that for every interval $\cI$ the action $\bA_\hom^\cI$ is the relaxation of the functional
	\begin{align*}
		%	\bar E_\hom^\alpha: 
		(\bfmu, \bfnu) \mapsto 
		\begin{cases} \displaystyle
			\int_{\cI\times \T^d}f_\hom\left( \frac{\dd \bfmu}{\dd \Leb^{d+1}}, \frac{\dd \bfnu}{\dd \Leb^{d+1}} \right) \dd \Leb^{d+1}, &\text{ if }( \bfmu, \bfnu) \ll \dd \Leb^{d+1}, \\
			+\infty, &\text{otherwise}, 
		\end{cases}
	\end{align*}
	for which \eqref{eq:contractivity_convolution_phik} follows from the convexity and nonnegativity of $f_\hom$ and the properties of the mollifiers $\phi^k$. 
	
	We shall then estimate the action of $(\bfmu^k, \bfnu^k)$. 
	From \eqref{eq:stability_convexcomb} and \eqref{eq:rhotilde_in_B0}, using the convexity of $f_\hom$ and the definition of the map $S^\eta$, we obtain 
	\begin{align*}
		\bA_\hom^{\cI_k}(\bfmu^k,\bfnu^k) &- (1+2\eta_k) \delta_k \sup_{B^\circ} f_\hom \\
		&\leq (1-\delta_k)  \int\limits_{\cI_k \times \Td} f_\hom \Big(\frac{\dd \tilde \bfmu^k}{\dd \Leb^{d+1}}(S_k^{-1}(t),x), \frac{\dd \tilde \bfnu^k}{\dd \Leb^{d+1}}(S_k^{-1}(t),x)  \Big)  \dd \Leb^{d+1}  \\
		&\leq  (1-\delta_k)(1+4\eta_k) \bA_\hom^{(\alpha_k, T- \alpha_k)}(\tilde \bfmu^k, \tilde \bfnu^k)  
		\leq (1-\delta_k) (1+4\eta_k) \bA_\hom^T(\bfmu, \bfnu), 
	\end{align*}
	where in the last inequality we used \eqref{eq:contractivity_convolution_phik}. Taking the limsup in $k \to \infty$
	\begin{align}
		\limsup_{k \to +\infty} \bA_\hom^{\cI_k}(\bfmu^k,\bfnu^k) \leq \bA_\hom^T (\bfmu, \bfnu)
	\end{align}
	which concludes the proof of \eqref{eq:dense_in_energy_smooth}. 
\end{proof}

Now we are ready to prove the limsup inequality \eqref{eq:limsup} in Theorem \ref{thm:main}.

\begin{proof}[Proof of Theorem \ref{thm:main} (upper bound)]
Fix $\bfmu \in \cM_+\big(\cI \times \Td\big)$. 
By definition of $\bA_\hom^\cI(\bfmu)$, 
it suffices to prove that 
for every $\bfnu \in \cM^d(\cI \times \Td)$ such that 
	$(\bfmu,\bfnu) \in \bCE^T$ and $\bA_\hom^\cI(\bfmu,\bfnu) < +\infty$, 
we can find $\bfm^\eps: \overline{\cI} \to \R_+^{\cX_\eps}$ 	
such that 
	$\iota_\eps \bfm^\eps \to \bfmu$ 
	weakly in $\cM_+(\cI \times \Td)$
	and $\limsup_\eps \cA_\eps^\cI(\bfm^\eps) 
		\leq \bA_\hom^\cI(\bfmu,\bfnu)$. 

For any such $(\bfmu, \bfnu)$, we apply Proposition \ref{prop:density} and find a smooth sequence $(\bfmu^k,\bfnu^k)_k \in \bCE^{\cI(k)}$ where $\cI(k) = (-\eta_k, T+ \eta_k)$, where $\eta_k \to 0$ and such that \eqref{eq:dense_in_energy_smooth} and \eqref{eq:compact_lemma_smooth} hold with $(\bfmu^k, \bfnu^k) \to (\bfmu,\bfnu)$ weakly in $\cM_+(\cI \times \Td) \times \cM^d(\cI \times \Td)$ as $ k \to +\infty$. In particular
\begin{align}	\label{eq:proof_bound_muk_masses}
	\sup_{k \in \N}
	\sup_{t \in \cI} \mu_t^k(\Td) = \sup_{k \in \N} \mu_0^k(\Td) <\infty.
\end{align}

Hence we can apply Proposition \ref{prop:limsup_smooth} and find 
	$\bfm^{\eps,k} \in \cM_+(\overline{\cI} \times \Td)$ 
such that 
	$\iota_\eps \bfm^{\eps,k} \to \bfmu^k$
weakly in $\cM_+(\overline{\cI} \times \Td)$ and for each $k \in \N$,
\begin{align}	\label{eq:proof_diagunal_smooth_recovery}
		\limsup_{\eps \to 0} 
			\cA_\eps^\cI (\bfm^{\eps,k}) 
			\leq \bA_\hom^{\cI(k)}(\bfmu^k, \bfnu^k) 
			+ C \eta_k |\cI|\big(\mu_0^k(\Td) + 1\big).
\end{align}
We conclude by extracting a subsequence $\bfm^\eps:= \bfm^{\eps,k(\eps)}$ such that $\iota_\eps \bfm^\eps \to \bfmu$ weakly in $\cM_+(\cI \times \Td)$ as $\eps \to 0$ and from \eqref{eq:proof_bound_muk_masses}, \eqref{eq:proof_diagunal_smooth_recovery}, \eqref{eq:dense_in_energy_smooth} we have
\begin{align*}
	\limsup_{\eps \to 0} \cA_\eps^\cI (\bfm^\eps) \leq \bA_\hom^\cI(\bfmu,\bfnu), 
\end{align*}
which concludes the proof.
\end{proof}

\section{Analysis of the cell problem}\label{sec:examples}

In the final section of this work, 
we discuss some properties of the limit functional $\bA_\hom$  
and analyse examples where explicit computations can be performed. 
For $\rho \in \R_+$ and $j \in \R^d$, 
recall that
\begin{equation*}
	\begin{aligned}	
		f_{\hom}(\rho,j) 
		& := 
		\inf\big\{ 
		F(m,J) 
		\ : \ 
		(m,J) \in \Rep(\rho,j) 
		\big\}, 
	\end{aligned}
\end{equation*}
where $\Rep(\rho,j)$ denotes the set of representatives of $(\rho,j)$, i.e., all $\Z^d$-periodic functions 
$m \in \R_+^\cX$ and $J \in \R_a^\cE$ 
satisfying
\begin{align*}
	\sum_{x \in \cX^Q} m(x) = \rho 
	, 
	\quad 
	\Eff(J) = \frac12 \sum_{(x,y) \in \cE^Q} J(x,y) (y_\sz - x_\sz) = j
	, 
	\quad 
	\text{ and } \quad
	\dive J \equiv 0.
\end{align*}

%\subsubsection{Invariance by rescaling}
\subsection{Invariance under rescaling}
We start with an invariance property of the cell-problem. Fix a $\Z^d$-periodic graph $(\cX, \cE)$ as defined in Assumption \ref{ass:XE}. 
For fixed $\eps > 0$ with $\eps \in \frac{1}{\N}$, 
we consider the rescaled $\Z^d$-periodic graph 
	$(\tilde \cX, \tilde \cE)$ 
obtained by zooming out by a factor $\frac1\eps$, so that each unit cube contains $(\frac1\eps)^d$ copies of $\XQ$.
Slightly abusing notation, we will identify the corresponding set $\tilde \V$ with the points in $\T_\eps^d$.

Let $\tilde F$ be the analogue of $F$ on $(\tilde \cX, \tilde \cE)$, and let $\tilde f_\hom$ be the corresponding limit density.
In view of our convergence result, 
the cell-formula must be \textit{invariant} under rescaling, 
namely $f_\hom = \tilde f_\hom$. 
We will verify this identity using a direct argument that crucially uses the convexity of $F$.

One inequality follows from the natural inclusion of representatives
\begin{align}	\label{eq:inclusions}
	\Rep(\rho) \hookrightarrow \eps^d \widetilde \Rep(\rho), 
	\quad
	\Rep(j) \hookrightarrow \eps^{d-1} \widetilde \Rep(j),
\end{align}
which is obtained as 
	$\tilde m:=\eps^d (\tau_\eps^0)^{-1}(m)$ 
and 
	$\tilde J := \eps^{d-1} (\tau_\eps^0)^{-1} (J)$ 
for every $(m,J) \in \Rep(\rho,j)$. Here we note that the inverse of $\tau_\eps^0$ is well-defined on $\Z^d$-periodic maps.
In particular we have
\begin{align*}
	\sum_{x \in {\tilde X}^Q} \tilde m(x) 
	= \sum_{x \in \XQ} m(x) = \rho 	
	, \quad 
	\Eff(\tilde J) = \Eff(J)
	, \tand
	\tilde F(\tilde m, \tilde J) = F(m,J),
\end{align*}
which implies that 
	$f_\hom \geq \tilde f_\hom$.

The opposite inequality is where the convexity of $F$ comes into play. 
Pick $(\tilde m, \tilde J) \in \widetilde \Rep(\rho,j)$. 
A first attempt to define a couple in $\Rep(\rho,j)$ would be to consider the inverse map of what we did in  \eqref{eq:inclusions},
but the resulting maps would not be $\Z^d$-periodic (but only $\frac1{\eps}\Z^d$-periodic). 
What we can do is to consider a convex combination of such values. 
Precisely, we define 
\begin{align*}
	m(x) := \eps^d \sum_{z \in \Z_\eps^d} \frac{\tau_\eps^z \tilde m(x)}{\eps^d}
	\tand 
	J(x,y) := \eps^d \sum_{z \in \Z_\eps^d}
	\frac{\tau_\eps^z \tilde J(x,y)}{\eps^{d-1}} 
\end{align*}
for all $(x,y) \in \XQ$.
The linearity of the constraints implies that $(m,J) \in \Rep(\rho,j)$. Moreover, using the convexity of $F$ we obtain
\begin{align*}
	F(m,J)	
	= 
	F
	\bigg( \eps^d
	\sum_{z \in \Z_\eps^d}
	\bigg(
	\frac{\tau_\eps^z \tilde m}
	{\eps^d}
	,
	\frac{\tau_\eps^z \tilde J}
	{\eps^{d-1}}
	\bigg)
	\bigg)
	\leq
	\sum_{z \in \Z_\eps^d}
	\eps^d
	F\bigg(
	\frac{\tau_\eps^z m}
	{\eps^d}
	,
	\frac{\tau_\eps^z J}
	{\eps^{d-1}}
	\bigg)
	= \tilde F(\tilde m, \tilde J),
\end{align*}
which in particular proves that $f_\hom \leq \tilde f_\hom$.

\subsection{The simplest case: $V= \{v\}$ and nearest-neighbor interaction.}
	\label{sec:NNI}
The easiest example we can consider is the one where the set $V$ consists of only one element $v \in V$. 
In other words, we focus on the case when $\cX \simeq \Z^d$ and thus $\cX_\eps \simeq \T_\eps^d$. 
We then consider the graph structure defined via the \textit{nearest-neighbor interaction}, 
meaning that $\cE$ consists of the elements of 
	$(x,y) \in \Z^d \times \Z^d$ 
such that $|x-y|_\infty = 1$.

In this setting, $\XQ \simeq V$ consists of only one element
and 
$\EQ  \simeq 
\left\{  
(v, v \pm e_i) \, : \, i=1, \dots, d
\right\}
$ has cardinality $2 d$. 
In particular, for every $\rho \in \R_+$ and $j \in \R^d$, the set $\Rep(\rho, j)$ consists of only one element $(\underline m, \underline J)$ given by
\begin{align*}
	\underline m (x) = \rho
	, \quad 
	\underline J(v, v \pm e_i) = \pm j_i
	, \quad 
	\text{ for all } (x,y) \in \cE
	\text{ \ and \ }
	i = 1, \dots, d.
\end{align*} 
Consequently, the homogenised energy density is given by  		$f_\hom(\rho,j) = F(\underline m, \underline J)$. 

In the special case where $F$ is edge-based (see Remark \ref{rem:examples}) with edge-energies 
	$\{ F_{\pm i} \}$ for $i = 1, \ldots, d$, 
we have 
\begin{align*}
	F(m, J ) 
		& = \sum_{i = 1}^d	
				F_i\big( m(0), m(e_i), J(0, e_i) \big)
			+	F_{-i}\big( m(0), m(- e_i), J(0, - e_i) \big),
\quad \text{and}			
	\\
	f_\hom(\rho,j) 
	& = \sum_{i = 1}^d	
	F_i\big( \rho, \rho, j_i \big)
+	F_{-i}\big( \rho, \rho, -j_i \big)
		\text{ for all } \quad 
		\rho \in \R_+ , \; j \in \R^d.
\end{align*}
The even more special case of the discretised $p$-Wasserstein distance corresponds to $F_i(\rho_1, \rho_2, j) = \frac{|j|^p}{2\Lambda(\rho_1,\rho_2)^{p-1}}$, where the mean $\Lambda$ is a mean as in  \eqref{eq: Wp}.
We then obtain
\begin{align*}
	f_\hom(\rho,j) = \frac{|j|_p^p}{\rho^{p-1}},
\end{align*}
for $\rho \in \R_+$ and $j \in \R^d$,
which corresponds to the $p$-Wasserstein distance induced by the $\ell_p$-distance $| \cdot |_p$ on the underlying space $\T^d$.
The case $p = 2$ corresponds to the framework studied in \cite{GiMa13}.

As we will discuss in Section \ref{subsec:fv}, this result can also be cast in the more general framework of isotropic finite-volume partitions of $\T^d$.

\subsection{Embedded graphs}
In this section, we shall use an equivalent \textit{geometric} definition of the effective flux. 
We can indeed formulate an interesting expression for $f_\hom$ in the case where $(\cX,\cE)$ is an embedded $\Z^d$-periodic graph in $\T^d$, in the sense of Remark \ref{rem:embedded_graph}. 
We thus choose $\V$ to be a subset of $[0,1)^d$ and use the identification $(z, v) \equiv z + v$, so that $\cX$ can be identified with a $\Z^d$-periodic subset of $\R^d$.

Let us define 
\begin{align*}
	\Effg(J) 
	:=
	\frac12
	\sum_{(x,y) \in \EQ} 
	J(x,y) 
	\big( 
	y - x
	\big).
\end{align*}
Note that we simply replaced $y_\sz - x_\sz \in \Z$ by $y - x \in \R^d$ in the definition of $\Eff(J)$. 
Remarkably, the following result shows that $\Eff(J) = \Effg(J)$ for any periodic and divergence-free vector field $J$. In particular, $\Effg(J)$ does \emph{not} depend on the choice of the embedding into $\T^d$. 
As a consequence, one can equivalently define $\Rep(j)$, and hence the homogenised energy density $f_\hom(\rho, j)$, in terms of $\Effg(J)$ instead of $\Eff(J)$.  

\begin{proposition}
	\label{prop:embedded-rep}	
	For every periodic and divergence-free vector field $J \in \R_a^\cE$ we have $\Eff(J) = \Effg(J)$.
\end{proposition}

\begin{proof}
	Fix $x_0 \in \cX^Q$. For a positive (small enough) parameter $t > 0$ and a vector $v \in \R^d$, consider the modified embedded $\Z^d$-periodic graph $(\cX(t),\cE(t))$ in $\Td$ obtained from $\cX$ by \textit{shifting} the nodes $x_0 + \Z^d$ by $t v \in \Td$, i.e., we consider the shifted node $x_0(t) := x_0 + t v$ instead of $x_0$ (and with it, the associated edges). 
	Fix a divergence-free and $\Z^d$-periodic discrete vector field 
		$J \in \R_a^\cE \simeq \R_a^{\cE(t)}$  
	and consider, for $t > 0$, the corresponding effective flux
	\begin{align*}
		\Effg(t,J) 
		:=
		\frac12
		\sum_{(x,y) \in \EQ(t)} 
		J(x,y) 
		\big( 
		y - x
		\big).
	\end{align*}
	We claim that 
	$\frac{\ddd}{\ddd t} \Effg(t,J) 
	= 0$. 
	Indeed, by construction we have
	\begin{align*}
		2\frac{\ddd}{\ddd t} \Effg(t,J)	
		&= 
		- \sum_{y \sim x_0} J(x_0,y) v 
		+
		\sum_{z \in \Z^d} 
		\sum_{\substack{x \in \XQ \\ x \sim x_0 + z}}
		J(x, x_0 + z) v
		\\
		&\stackrel{J \text{ per.}}{=} - \dive J(x_0) v + 
		\sum_{z \in \Z^d}
		\sum_{\substack{x \in \XQ\\ x - z \sim x_0}}
		J(x - z, x_0) v 
		\\
		&= - \dive J(x_0) v + \sum_{x' \sim x_0} J(x',x_0) v 
		\\&= - \dive J(x_0) v + \sum_{x' \sim x_0} J(x_0,x') v 
		= -2 \dive J(x_0) v.
	\end{align*}
	Since $J$ is divergence-free, this proves the claim. 
	In particular, $t \mapsto \Effg(t,J)$ is constant, hence the value of $\Effg$ does not depend on the location of the embedded points.  This also implies the sought equality $\Eff(J) = \Effg(J)$, since $\Eff(J)$ corresponds to the limiting case where all the elements of $V$ ``collapse'' into a single point of $[0,1)^d$. 
\end{proof}

\subsection{Periodic finite-volume partitions}
\label{subsec:fv}

The next class of examples are the graph structures associated with $\Z^d$-periodic \textit{finite-volume partitions} (FVPs) $\cT$ of $\R^d$. 
We refer to \cite{Eymard-Gallouet-Herbin:2000} for a general treatment.

\begin{figure}[h!]
	\includegraphics[scale=0.35]{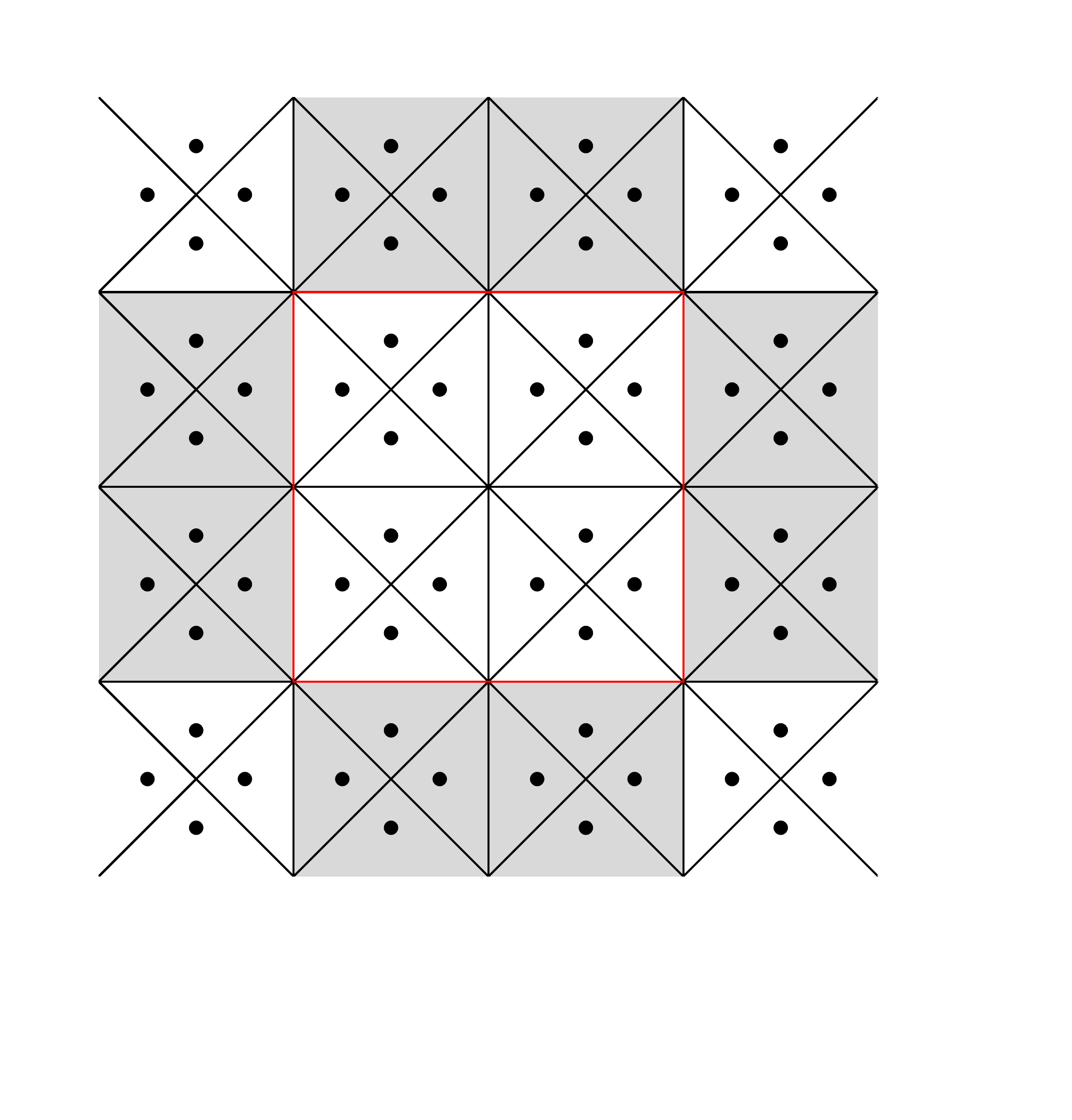}
	\vspace{-1cm}
	\caption{A $\Z^2$-periodic finite-volume partition of $\R^2$. 
	The unit cube $[0,1]^2 \subseteq \R^2$ is shown in red.}
	\label{fig:FVperiodic}
\end{figure}

\begin{definition}[$\Z^d$-periodic finite-volume partition]
	Consider a countable, locally finite, $\Z^d$-periodic family of points $\cX \subseteq \R^d$ together with a family of nonempty open  bounded convex polytopes $K_x \subseteq \R^d$ for $x \in \cX$, such that $K_{x+z} = K_x + z$ for all $x \in \cX$ and $z \in \Z^d$. 
	We call 
	\begin{align*}
		\cT:= \Big\{ (x, K_x) \suchthat x \in \cX \Big\}
	\end{align*}
	a \emph{$\Z^d$-periodic finite-volume partition} of $\R^d$ if
	\begin{enumerate}
		\item $\bigcup_{x\in \cX} \overline{K_x} = \R^d$;
		\item $K_x \cap K_y = \emptyset$ whenever $x\neq y \in \cX$;
		\item $y-x \perp \partial K_x \cap \partial K_y$ whenever  $\Haus^{d-1}(\partial K_x \cap  \partial K_y) > 0$.
	\end{enumerate}
	We define a graph structure on $\cX$ by declaring those pairs $(x,y)\in \cX \times \cX$ with $\Haus^{d-1}(\partial K_x \cap  \partial K_y) > 0$ to be nearest neighbors.
\end{definition}

It is not difficult to see that the graph $(\cX, \cE)$ is connected, $\Z^d$-periodic, and locally finite, even if $x \notin K_x$.
Throughout this section we use the following  notation for $x, y \in \cX$:
\begin{align*}
	|K_x|  & := \Leb^d(K_x), &
	d_{xy} & := |y-x|, \\
	s_{xy} & := \Haus^{d-1}(\partial K_x \cap \partial K_y), &
	n_{xy} & := \frac{y-x}{d_{xy}} \in \cS^{d-1}.
\end{align*}

In the finite-volume framework, we are interested in transport distances with a nonlinear mobility. 
These distances were introduced in \cite{Dolbeault-Nazaret-Savare:2009} as natural generalisations of the 2-Wasserstein metric.
We thus fix a concave upper-semicontinuous function 
	$\fm : \R_+ \times \R^d \to \R_+$
and consider the energy density functional
\begin{align}
	\label{eq:f-formula}
	f(\rho, j) := 
	\begin{cases}
		\frac{|j|_2^2}{\fm(\rho)} 
			& \text{ if } \fm(\rho) > 0,\\
		+\infty 
			& \text{ if }\fm(\rho) = 0 \text{ and } j \neq 0,\\
		0 
			& \text{ if } \fm(\rho) = 0 \text{ and } j = 0.
	\end{cases}
\end{align}

To discretise this energy density, we fix for every edge $(x,y) \in \cE$ 
an \emph{admissible version} of $\fm:\R_+\to \R_+$,
i.e., 
a nonnegative concave upper-semicontinuous function
$\fm_{xy} : \R_+ \times \R_+ \to \R_+$ 
satisfying
	$\fm_{xy}(\rho,\rho) = \fm(\rho)$ for all $\rho\in \R_+$ and $(x,y)\in \cE$.
We always assume that 
	$\fm_{xy}(\rho_1,\rho_2) 
	= \fm_{yx}(\rho_2,\rho_1)$
for all $\rho_1, \rho_2 \in \R_+$.  
It is easy to check that $F$ satisfies the superlinear growth condition \ref{eq: superlinear}.
Furthermore, concavity of $\fm_{xy}$ implies convexity of $F$.\footnote{Concavity of $\fm_{xy}$ is not necessary for convexity of $F$. If $\fm_{xy}$ is not concave, a local version of the supergradient can be substituted into \eqref{eq: isotropy}. For readability we restrict ourselves to the concave case.}
We then consider the edge-based cost defined by
\begin{align}	\label{eq:energies_FV}
	F(m,J) := \frac12 \sum_{(x,y) \in \EQ}  
	\frac{d_{xy}}{s_{xy}}
	\frac{J(x,y)^2}{\fm_{xy} \Big( 
		\frac{m(x)}{|K_x|}, \frac{m(y)}{|K_y|} \Big)}, 
\end{align}
Consistent with \eqref{eq:f-formula}, we use the convention that
\begin{align}
	\label{eq:division-by-zero}
	\frac{J(x,y)^2}{\fm_{xy} 
		\Big( 
		\frac{m(x)}{|K_x|}, \frac{m(y)}{|K_y|} \Big) }	
		= 
		\begin{cases}
			+ \infty & 
			\text{ if } 
			\fm_{xy} \Big( 
				\frac{m(x)}{|K_x|}, 
				\frac{m(y)}{|K_y|} \Big)
				= 0 
			\text{ and } 
			J(x,y) \neq 0, \\
			0 & 
				\text{ if } 
				\fm_{xy} \Big( 
					\frac{m(x)}{|K_x|}, 
					\frac{m(y)}{|K_y|} \Big)
					= 0 
				\text{ and } 
				J(x,y) = 0.
		\end{cases}
\end{align}
It is now natural to ask whether the discrete action functionals associated to $F$ converge to the continuous action funtional associated to $f$: is it true that $f_\hom = f$? 

In the linear case where $\fm(\rho) = \rho$, which corresponds to the $2$-Wasserstein metric, this question has been extensively studied in \cite{GlKoMa18} for a large class of (not necessarily periodic) meshes.
The main result in \cite{GlKoMa18} asserts that the limit of the discrete transport distances $\cW_\eps$ (in the Gromov-Hausdorff sense) as $\eps \to 0$ coincides with the $2$-Wasserstein distance $\bW_2$ on $\cP(\Td)$ if an asymptotic local isotropy condition is satisfied. 
Moreover, it is shown that this convergence fails to hold if the isotropy condition fails to hold (in a sufficiently strong sense).

For periodic finite-volume partitions we show here  that these results are direct consequences of Theorem \ref{thm:main}.
In particular, the following result contains a necessary and condition on a periodic finite-volume partition that ensures that $f_\hom = f$.

\begin{proposition}\label{prop: finite volume}
	Consider a $\Z^d$-periodic finite-volume partition of $\R^d$, 
	and let $F$ and $f$ be as in 
	\eqref{eq:f-formula}
	and
	\eqref{eq:energies_FV} respectively.
	The following assertions hold:
	\begin{enumerate}[(i)]
		\item $f_\hom(\rho,j) \leq f(\rho,j)$ for all $\rho \in \R_+$ and $j\in \R^d$. 
		\item Suppose that for every $\rho\in \R_+$ and $j\in \R^d$ there is a family of vectors 
		$(p^{xy})_{(x,y) \in \cE} \subseteq \R^2$ such that
		\begin{align}
		& p^{xy} = (p^{xy}_1,p^{xy}_2)
			\in 
		\partial^+ \fm_{xy}(\rho,\rho)
		&& \text{for all }
		(x,y) \in \cE, \text{ and }\\
				\label{eq: isotropy}
				& \frac{1}{|K_x|}
				\sum_{y \sim x}
				(p^{xy}_1 + p^{yx}_2) 
				d_{xy} s_{xy} 
				(n_{xy} \cdot j)^2
				&& \text{is independent of }
				x \in \cX.
		\end{align} 
		Then: $f_\hom = f$.

		\item Suppose that all $\fm_{xy}$ are differentiable in a neigbourhood of the diagonal in $(0,\infty)^2$. Then $f_\hom = f$ if and only if
		\begin{equation}\label{eq: isometry}
			\sum_{y \sim x} 
				\frac{\partial_1\fm_{xy}(\rho,\rho)}
					 	{\fm'(\rho) } 
				d_{xy} s_{xy} 
				n_{xy} \otimes n_{xy} 
			= 
			|K_x| \id \quad 
			\text{ for all } 
			x \in \cX \text{ and } 
			\rho > 0.
		\end{equation}
	\end{enumerate}
\end{proposition}

Before we prove Proposition \ref{prop: finite volume}, we first show an elementary identity for finite-volume partitions; see also \cite[Lemma 5.4]{GlKoMa18} for a similar result in a non-periodic setting.

\begin{lemma}\label{lemma: fv identity}
	Let $\cT$ be a $\Z^d$-periodic finite-volume partition of $\R^d$. Then
	\begin{equation}\label{eq: fv identity}
		\frac12 \sum_{(x,y)\in \EQ} d_{xy} s_{xy} n_{xy} \otimes n_{xy} =  \id.
	\end{equation}
\end{lemma}

\begin{proof}
	For $v\in \R^d\setminus \{0\}$
	and $(x,y)\in \cE$, consider the open bounded convex polytope
	\[
	C_{xy} := 
		\big\{
			 z \in (\partial K_x \cap \partial K_y) + \R v\,:\,z \cdot v \in 
			 \big(\conv(x\cdot v, y\cdot v)\big)^\circ 
		\big\}.
	\]
	Note that $C_{xy} = C_{yx}$. We claim that the family $\{C_{xy}\,:\,(x,y)\in \cE\}$ forms a partition of $\R^d$ up to a set of Lebesgue measure zero.
	To see this, fix a point $z \in \R^d$ and consider the function $X: \R \to \cX$ defined by 
		$X(t) = x$ if $z + t v \in K_x$.
	If $v$ is not orthogonal to any of the finitely many $n_{xy}$, 
	then $X(t)$ is well-defined 
	up to a countable set 
		$N \subset \R$. 
	By Fubini's theorem, it follows that 
		$\Leb^d\big(
				\R^d \setminus 
				\bigcup_{(x,y)\in \cE} C_{xy}
				\big) 
				= 0$.
	
	If $t \in N$ and $X(t^-) = x$, $X(t^+) = y$, then $(y-x)\cdot v 	
		= 
		d_{xy} n_{xy} \cdot v > 0$. 
	This shows that 
		$t\mapsto v\cdot X(t)$ 
	is nondecreasing and that $z$ is in at most one parallelepiped.

	On the other hand, we have 
	\[
	\Leb^d(C_{xy}) = d_{xy} s_{xy} \left(n_{xy} \cdot \frac{v}{|v|}\right)^2.
	\] 
	Then we have
	\begin{align*}
		1 
		& = 
		\frac12 \sum_{x\in \cX} \sum_{y\sim x} \Leb^d(C_{xy} \cap [0,1)^d)
		=  
		\frac12 \sum_{x\in \XQ} \sum_{y\sim x} \Leb^d(C_{xy})
		\\ & =  
		\frac12\sum_{(x,y)\in \EQ} 
			d_{xy} s_{xy} 
			\bigg(n_{xy} \cdot \frac{v}{|v|}\bigg)^2
		= \frac{v}{|v|} \cdot \bigg(\frac12 \sum_{(x,y)\in \EQ}  d_{xy} s_{xy} n_{xy} \otimes n_{xy}\bigg) \frac{v}{|v|}.
	\end{align*}
	Since this identity holds for almost every $v \in \R^d$, \eqref{eq: fv identity} holds by polarization.
\end{proof}

\begin{proof}[Proof of Proposition \ref{prop: finite volume}]
	\emph{(i)}: \ We construct a competitor 	
		$(m^\star,J^\star)$ 
	to the cell problem \eqref{eq: fhom} 
	for $\rho \in \R_+$ and $j \in \R^d$.
	Define 
	\begin{align}
	\label{eq:candidates}	
		 m^\star(x) := |K_x|\rho
	\tand 
		J^\star(x,y) := s_{xy} (j \cdot n_{xy}).
	\end{align} 
	We claim that 
		$(m^\star, J^\star)\in \Rep(\rho,j)$.
	Indeed, the periodicity of $\cT$ yields
	\[
	\sum_{x\in \XQ} m^\star(x) 
		= \rho \sum_{x\in \XQ} |K_x| 
		= \rho \Leb^d([0,1)^d)
		= \rho, 
	\]
	which shows that $m^\star \in \Rep(\rho)$.
	To show that $J^\star \in \Rep(j)$, 
	we use the divergence theorem
	to obtain, for $x \in \cX$,
	\[
	\dive J^\star(x) 
		= \sum_{y\sim x} J^\star(x,y) 
		= \sum_{y\sim x} s_{xy} (j \cdot n_{xy})
		= \int_{\partial K_x} 
				j \cdot n 
			\dd \Haus^{d-1} =  0.
	\]
	Moreover, using Proposition \ref{prop:embedded-rep} and Lemma \ref{lemma: fv identity}
	we find 
	\begin{align*}
		\Eff(J^\star) 
		& = \frac12 \sum_{(x,y)\in \EQ}  
			J^\star(x,y) (y_\sz - x_\sz) 
		= \frac12 \sum_{(x,y)\in \EQ}  
			J^\star(x,y) (y - x)
		\\
		& = \frac12 \sum_{(x,y)\in \EQ} 	
			s_{xy} (j \cdot n_{xy}) (y-x)
		= \frac12 \sum_{(x,y)\in \EQ} 
			d_{xy} s_{xy} 
			(n_{xy} \otimes n_{xy}) j
		= j,
	\end{align*}
	which proves that $J^\star \in \Rep(j)$. 
	Therefore, using that $\fm_{xy}$ is an admissible version of $\fm$, another application of Lemma \ref{lemma: fv identity} yields (taking \eqref{eq:division-by-zero} into account),
	\begin{align*}
		f_\hom(\rho,j) 
			 & \leq F(m^\star, J^\star) 
			 = \frac12 \sum_{(x,y)\in \EQ} 
			 	\frac{d_{xy}}{s_{xy}} 
				 \frac{J^\star(x,y)^2}{\fm_{xy}(\rho,\rho)}
		\\ & =   \frac{1}{\fm(\rho)}j \cdot 
			\left(\frac12 \sum_{(x,y)\in \EQ} 
			 d_{xy} s_{xy} 
			 n_{xy} \otimes n_{xy} \right) j 
		= \frac{|j|^2}{\fm(\rho)}
		= f(\rho, j),
	\end{align*} 
	which proves \emph{(i)}.
	
	\medskip
	\emph{(iii)}: \ 
	Suppose first that condition \eqref{eq: isometry} holds.
	We will show that $(m^\star, J^\star)$ is a critical point of $F$.
	Take $(\tilde m, \tilde J) \in \Rep(0,0)$ and define, 
	for $\eps > 0$ sufficiently small,
	\begin{align*}
		m_\eps := m^\star + \eps \tilde m 	
			\tand
		J_\eps := J^\star + \eps \tilde J.
	\end{align*}
	Then: 
	\begin{align*}
		\partial_\eps\big|_{\eps = 0}
			F(m^\star,J_\eps)
		& = 	
		\frac12 
		\partial_\eps\big|_{\eps = 0}
		\sum_{(x,y)\in \EQ} 
			 	\frac{d_{xy}}{s_{xy}} 
				 \frac{J^\star(x,y)^2}
				 {\fm_{xy}(\rho,\rho)}
		= 	
		\frac{1}{\fm(\rho)}
		\sum_{(x,y)\in \EQ} 
				\frac{d_{xy}}{s_{xy}} 
				  J^\star(x,y)
				  		\tilde J(x,y)
	\\&	= 	
		\frac{1}{\fm(\rho)}
		\sum_{(x,y)\in \EQ} 
				  j \cdot (y-x)
				  		\tilde J(x,y)
		= 	\frac{1}{\fm(\rho)}
			j \cdot \Eff(\tilde J) 
		= 0.					  				  
	\end{align*}
	Furthermore, using the symmetry $\fm_{xy}(a,b) = \fm_{yx}(b,a) = $ for $a, b \geq 0$, we obtain
	\begin{equation}
	\begin{aligned}
		\label{eq:m-derivative}
	\partial_\eps\big|_{\eps = 0}
		F(m_\eps, J^\star)
	& = 	
	- \frac12 
	\sum_{(x,y)\in \EQ} 
			 \frac{d_{xy}}{s_{xy}} 
			 \frac{J^\star(x,y)^2}
			 {\fm(\rho)^2}
		\bigg(
			\frac{\tilde m(x)}{|K_x|}
				\partial_1 \fm_{xy}(\rho, \rho)
			+
			\frac{\tilde m(y)}{|K_y|}
				\partial_2 \fm_{xy}(\rho, \rho)
		\bigg)
	\\ & = 	
	-
	\sum_{(x,y)\in \EQ} 
			\frac{d_{xy}}{s_{xy}} 
			\frac{J^\star(x,y)^2}
			{\fm(\rho)^2}
			\frac{\tilde m(x)}{|K_x|}
				\partial_1 \fm_{xy}(\rho, \rho)
	\\& =				
		- \frac{\fm'(\rho)}{\fm^2(\rho)} |j|^2
		\sum_{x \in \XQ}
		b_x(\rho, j) 
			\frac{ \tilde m(x) }{ |K_x| }
			,			
		\end{aligned}	
	\end{equation}
	where we write
	$b_{x}(\rho, j) := 
	\sum_{(x,y \in \EQ} 
		\frac{\partial_1 \fm_{xy}(\rho, \rho)}
			 {\fm'(\rho)} 
			d_{xy} s_{xy} \frac{(n_{xy} \cdot j)^2}{|j|^2}$, 
	so that the condition \eqref{eq: isometry} reads as 
		$b_x(\rho, j) = |K_x|$ 
	for all $\rho > 0$, $j \in \R^d$, and $x \in \XQ$. 
	Hence, if this condition holds, 
	we obtain, since $\tilde m(x) \in \Rep(0)$, 
	\begin{align*}
		\partial_\eps\big|_{\eps = 0}
		F(m_\eps, J^\star)
		 = - \frac{\fm'(\rho) }{\fm^2(\rho)} |j|^2 \sum_{x \in \XQ} \tilde m(x) = 0.
	\end{align*}
	Adding the identities above, we conclude that
	$
		\frac{\ddd}{\ddd \eps}\big|_{\eps = 0}
		F(m_\eps, J_\eps)
		= 0
	$
	whenever \eqref{eq: isometry} holds.
	Therefore, $(m^\star, J^\star)$ is a critical point of $F$ in $\Rep(\rho, j)$.
	By convexity of $F$, it is a minimiser. 
	Consequently, using Lemma \ref{lemma: fv identity}, we obtain
	\begin{align*}
		f_\hom(\rho, j)
		= F(m^\star, J^\star)
		= \frac{1}{2 \fm(\rho)} 
			\sum_{(x,y)\in \EQ} 
			d_{xy}{s_{xy}}
			(j \cdot n_{xy})^2
		= \frac{|j|^2}{\fm(\rho)}
		= f(\rho, j),
	\end{align*}
	which is the desired identity.
	\smallskip
	
	To prove the converse, 
	we assume that 
		\eqref{eq: isometry}
	does not hold, i.e., 
	we have
		$b_{\bar x}(\rho, j) \neq |K_{\bar x}|$
	for some 
		$\rho > 0$, $j \in \R^d$, and $\bar x \in \cX$.	 
	On the other hand, we claim that 
	\begin{align*}
		\sum_{x \in \XQ} b_x(\rho, j) = 1.
	\end{align*}
	To see this, observe first that, 
	by definition of admissibility of $\fm_{xy}$ 
	and 
	the symmetry assumption 
	$\fm_{xy}(a, b) = \fm_{yx}(b, a)$,
	we have	
	\begin{align*}
		\fm'(\rho)
		= \partial_\eps\big|_{\eps = 0}
			\fm(\rho + \eps) 
		= \partial_\eps\big|_{\eps = 0}
			\fm_{xy}(\rho + \eps,\rho + \eps) 
		& = \partial_1 \fm_{xy}(\rho, \rho) 
			+ \partial_2 \fm_{xy}(\rho, \rho)
		\\& = \partial_1 \fm_{xy}(\rho, \rho) 
			+ \partial_1 \fm_{yx}(\rho, \rho).
	\end{align*}
	Using this identity, the periodicity of $m$ and $J$, 
	and the identity \eqref{eq: fv identity} we obtain
	\begin{align*}
		\sum_{x \in \XQ} b_x(\rho, j)
		& = \sum_{(x,y \in \EQ} 
		\frac{\partial_1 \fm_{xy}(\rho, \rho)}
			 {\fm'(\rho)} 
			d_{xy} s_{xy} \frac{(n_{xy} \cdot j)^2}{|j|^2}
		\\& = \frac12 \sum_{(x,y \in \EQ} 
			\frac{\partial_1 \fm_{xy}(\rho, \rho)
				+ \partial_1 \fm_{yx}(\rho, \rho)}
				 {\fm'(\rho)} 
				d_{xy} s_{xy} \frac{(n_{xy} \cdot j)^2}{|j|^2}
		\\& = \frac12 \sum_{(x,y \in \EQ} 
					d_{xy} s_{xy} \frac{(n_{xy} \cdot j)^2}{|j|^2}
		= 1,
	\end{align*}	
	which proves the claim.

	We thus infer that $b_x(\rho, j) / |K_x|$ is non-constant in $x$. (If it were, the identity $\sum_x |K_x| = 1 = \sum_x b_x(\rho, j)$ would imply that $b_x(\rho, j) = |K_x|$ for all $x$. But we assume that this doesn't hold for $x = \bar x$.) 
	Consequently, there exists a $\Z^d$-periodic function $\tilde m : \cX \to \R$ with $\sum_{x \in \XQ} \tilde m(x) = 0$ such that
	\begin{align*}
		\sum_{x \in \XQ}
			b_x(\rho, j) \frac{\tilde m(x)}{|K_x|} \neq 0.
	\end{align*}
	As before, we consider $(m^\star, J^\star) \in \Rep(\rho, j)$ defined by \eqref{eq:candidates}.
	In view of \eqref{eq:m-derivative}, we infer that 
		$(m^\star, J^\star)$
	is \emph{not} a critical point of $F$ in $\Rep(\rho, j)$.
	As $(m^\star, J^\star)$ is a relatively interior point of $\Rep(\rho, j)$, it cannot be a minimiser, hence
		$f_\hom(\rho,j) 
			< F(m^\star, J^\star) 
			= f(\rho, j)$.
	
	\medskip
	\emph{(ii)}: \ 
	We construct an element of the subgradient $(p_m,p_J)\in \partial^- F(m,J)$ with $\langle p_m, dm \rangle = \langle p_J, dJ \rangle = 0$ for all $dm\in \Rep(0)$, $dJ\in \Rep(0)$.
	
	We set
	\[
	p_m(x) := \sum_{y\sim x}\frac{J(x,y)^2}{|K_x|\fm^2(\rho)}  \frac{d_{xy}}{s_{xy}}(p^{xy}_1 + p^{yx}_2)
	\]
	and check by a simple calculation involving the chosen supergradients $p^{xy}$ that $F(m+dm,J) - F(m,J) \geq \langle p_m, dm \rangle$ for all $dm\in \R^\cX$ periodic. The isotropy condition \eqref{eq: isotropy} implies that $p_m$ is independent of $x$ and thus $\langle p_m, dm \rangle = 0$ for all $dm\in \Rep(0)$.
	
	Since $F$ is differentiable in $J$, we have to choose $p_J := \partial_J F(m,J)$. By the same calculation as in $(3)$ we see that $\langle p_J, dJ \rangle  = 0$ for all $dJ\in \Rep(0)$.
	
	To see that $(m,J)$ is indeed a local (and thus global) minimiser of $F$ in $\Rep(\rho,j)$, we introduce a parameter $\eps >0$ and show that
	\begin{equation}\label{eq: directional derivative}
		\liminf_{\eps \searrow 0} \frac1\eps \left( F(m+\eps dm, J+\eps dJ) - F(m,J) \right) \geq 0
	\end{equation}
	for all $dm\in \Rep(0)$ and $dJ\in \Rep(0)$.
	
	To see this, we expand the difference
	\begin{align*}
		&\frac1\eps \left( F(m+\eps dm, J+\eps dJ) - F(m,J) \right)\\
		= & \frac1\eps \left( F(m+\eps dm, J+\eps dJ) - F(m + \eps dm,J) \right)  + \underbrace{\frac1\eps \left(F(m + \eps dm, J) - F(m,J)  \right)}_{\geq 0}\\
		\geq & \langle \partial_J F(m,J) + o(1), dJ \rangle \to_{\eps \to 0} 0,
	\end{align*}
	where we used that $(m,J) \mapsto \partial_J F(m,J)$ is continuous. Because $F$ is convex, \eqref{eq: directional derivative} implies that $(m,J)$ is a minimiser of $F$ in $\Rep(\rho,j)$.

\end{proof}

\begin{remark}
	Given a concave mobility $\fm:\R_+ \to \R_+$, a popular admissible version is to take $\fm_{xy}(a,b) := \fm(\lambda_{xy} a + (1-\lambda_{xy}) b)$, with weights $\lambda_{xy}\in [0,1]$. If $\fm$ is differentiable, this means that $\partial_1 \fm_{xy}(\rho,\rho) = \lambda_{xy} \fm'(\rho)$. As a result, for certain finite-volume partitions we have to choose the weights $\lambda_{xy}$ to satisfy \eqref{eq: isometry}.
	
	Of particular importance is the $\bW_2$ case $\fm(\rho) = \rho$, which was treated in \cite{GlKoMa18} and \cite{gladbach2020}. Here an admissible version $\fm_{xy}$ is called an admissible mean. For differentiable $\fm_{xy}$, condition \eqref{eq: isometry} reduces to
	\[
	\sum_{y\sim x} \partial_1 \fm(\rho,\rho) d_{xy}s_{xy} n_{xy} \otimes n_{xy} = |K_x| \id.
	\]
	We note that condition \eqref{eq: isometry} cannot be satisfied for a large class of finite-volume partitions, although the square partition fulfills it with $\partial_1 \fm(\rho,\rho) = 1/2$.
	The condition also holds for some other partitions that are not $\Z^d$-periodic, such as the equilateral triangular and hexagonal partitions; see \cite{GlKoMa18}.

	If we allow ourselves to use nonsmooth admissible versions of $\fm$, it makes sense to use \ $\fm_{xy}(a,b) := \fm(\min(a,b))$, as this choice guarantees the largest possible supergradient $\partial^+ \fm_{xy} = \partial^+ \fm \{(\lambda, 1-\lambda)\,:\,\lambda\in [0,1] \}$ along the diagonal, making it more likely that $f_\hom(\rho,j) = \frac{|j|_2^2}{\fm(\rho)}$.
	
\end{remark}

\begin{example}
	Let us consider the triangulation given in Figure \ref{fig:FVperiodic}, where each unit square consists of four triangles: north, south, west, and east.
	We now show that \eqref{eq: isometry} cannot be satisfied here, but \eqref{eq: isotropy} is satisfied for the particular nonsmooth choice $\fm_{xy}(\rho_1,\rho_2)=\min(\rho_1,\rho_2)$.
	
	For the smooth case we assume that $\fm_{xy}(\rho,\rho)=\rho$ and define $\lambda_{xy}=\partial_1\fm_{xy}$ and $\lambda_{yx}=\partial_2\fm_{xy}$. Note that by the chain rule $\lambda_{xy}+\lambda_{yx}=1$. Let 
	\begin{align*}
		A_x:=\sum_{y\sim x} \lambda_{xy} d_{xy}s_{xy}n_{xy}\otimes n_{xy}.
	\end{align*}
	For $x_N$ in the north triangle and $x_S$ in the south triangle we obtain that 
	\begin{align*}
		e_2\cdot (A_{N}+A_{S})e_2=\frac12+\frac18(\lambda_{SE}+\lambda_{NE}+\lambda_{SW}+\lambda_{NW})
	\end{align*}
	since $d_{NW}s_{NW}=\frac14$, $d_{NS}s_{NS}=\frac12$, $n_{NS}=e_2$ and $n_{NE}=(\frac1{\sqrt{2}},-\frac1{\sqrt{2}})^T$
	Similarly we obtain for $x_W$ in the west and $x_E$ in the east triangle
	\begin{align*}
		e_1\cdot (A_{W}+A_{E})e_1=\frac12+\frac18(\lambda_{ES}+\lambda_{EN}+\lambda_{WS}+\lambda_{WN}). 
	\end{align*}
	Inserting the last two equalities into \eqref{eq: isometry} we find that $e_2\cdot A_xe_2 = e_1 \cdot A_x e_1 = \frac14$ for all $x\in \{S,E,N,W\}$, i.e. that
	\begin{align*}
		\lambda_{SE}+\lambda_{NE}+\lambda_{SW}+\lambda_{NW}=\lambda_{ES}+\lambda_{EN}+\lambda_{WS}+\lambda_{WN}=0.
	\end{align*}
	But this is a contradiction to $\lambda_{xy}+\lambda_{yx}=1$. In particular there exists no $\fm_{xy}$ satisfying \eqref{eq: isometry}.
	
	For the nonsmooth case note that the supergradient for $\fm_{xy}(\rho_1,\rho_2)=\min(\rho_1,\rho_2)$ is given by 
	\begin{align*}
		\partial^+\fm_{xy}(\rho,\rho)=\{(\lambda,1-\lambda) : \lambda\in[0,1]\}.
	\end{align*}
	For $\rho\in\mathbb R_+$ and $j\in\mathbb R^d$ we set 
	\begin{align*}
		p^{NS}=p^{SN}=p^{EW}=p^{WE}=&\left(\frac12,\frac12\right)\in\partial^+\fm_{xy}(\rho,\rho)\\
		p^{NE}=p^{NW}=p^{SE}=p^{SW}=&\left(\frac{j_1^2}{|j|_2^2},\frac{j_2^2}{|j|^2_2}\right)\in\partial^+\fm_{xy}(\rho,\rho)\\
		p^{EN}=p^{WN}=p^{ES}=p^{WS}=&\left(\frac{j_2^2}{|j|_2^2},\frac{j_1^2}{|j|^2_2}\right)\in\partial^+\fm_{xy}(\rho,\rho).
	\end{align*}
	We need to show that $a_{x,j}:=\frac1{|K_x|}\sum_{y\sim x}(p_1^{xy}+p_2^{yx})d_{xy}s_{xy}(n_{xy}\cdot j)^2$ is independent of $x$. For $x$ in the north or the south triangle we find
	\begin{align*}
		a_{S,j}= a_{N,j}=&4\left(\frac12 j_2^2+\frac28 \frac{2j_1^2}{|j|_2^2}\left(\frac{(j_1-j_2)^2}{2}+\frac{(j_1+j_2)^2}{2}\right)\right)\\
		=&4\left(\frac12 j_2^2+\frac12 \frac{j_1^2}{|j|_2^2}|j|_2^2\right)=2|j|_2^2.
	\end{align*}
	Similarly for x in the west or east triangle we obtain
	\begin{align*}
		a_{E,j}=a_{W,j} =&4\left(\frac12 j_1^2+\frac28 \frac{2j_2^2}{|j|_2^2}\left(\frac{(j_2-j_1)^2}{2}+\frac{(j_1+j_2)^2}{2}\right)\right)\\
		=&4\left(\frac12 j_1^2+\frac12 \frac{j_2^2}{|j|_2^2}|j|_2^2\right)=2|j|_2^2.
	\end{align*}
	Consequently, this is independent of $x$ and \eqref{eq: isotropy} holds.
\end{example}

\appendix

\section{The Kantorovich--Rubinstein metric on signed measures}
\label{sec:KR}

We collect some facts on the Kantorovich--Rubinstein metric that are used in the paper. 
We refer to \cite[Section 8.10(viii)]{Bogachev} for more details.

Let $(X,d)$ be a metric space.
Let $\cM(X)$ denote the space of finite signed Borel measures on $X$.
For $\mu \in \cM(X)$, let 
	$\mu^+, \mu^- \in \cM_+(X)$ 
	be the positive and negative parts, respectively.
Let $|\mu| = \mu^+ + \mu^-$ be its variation, and 
	$\|\mu\|_\TV := |\mu|(X)$ be its total variation.

\begin{definition}[Weak and vague convergence]
	Let $\mu, \mu_n \in \cM(X)$ for $n = 1, 2, \ldots$.
	\begin{enumerate}[$(i)$]
		\item 
		We say that	$\mu_n \to \mu$ weakly in $\cM(X)$ if
		$
			\int_{X} \psi \dd \mu_n 
		\to 
			\int_{X} \psi \dd \mu 
		$ 
		for every $\psi \in \cC_{\rm b}(X)$.
		\item 
		We say that $\mu_n \to \mu$ vaguely in $\cM(X)$ if
		$
			\int_{X} \psi \dd \mu_n 
		\to 
			\int_{X} \psi \dd \mu 
		$ 
		for every $\psi \in \cC_{\rm c}(X)$.
	\end{enumerate}
\end{definition}

If $(X,d)$ is compact, $\cM(X)$ is a Banach space endowed with the norm $\|\mu\|_\TV$.
By the Riesz–Markov theorem, it is the dual space of the Banach space $\cC(X)$ of all continuous functions $\psi : X \to \R$ endowed with the supremum norm 
	$\|\psi\|_\infty = \sup_{x \in X} |\psi(x)|$.

For $\psi : X \to \R$ let 
	$\Lip(\psi) := \sup_{x \neq y}
		\frac{|\psi(x) - \psi(y)|}{d(x,y)}
	$
be its Lipschitz constant. 

\begin{definition}
	Let $(X,d)$ be a compact metric space. 
	The \emph{Kantorovich--Rubinstein norm} on $\cM(X)$ is defined by 
	\begin{align}
		\label{eq:KR}
		\| \mu \|_{\KR(X)}
			:= 
		\sup\bigg\{ \int_X 
					\psi \dd \mu
				\ : \
				\psi \in \cC(X), \ 
				\| \psi \|_\infty \leq 1, \
				\Lip(\psi) \leq 1
			\bigg\}.
	\end{align} 
\end{definition}	

In non-trivial situations (i.e., when $X$ contains an infinite convergent sequence), the norms $\|\cdot\|_\KR$ and $\| \cdot \|_\TV$ are not equivalent. 
Thus, by the open mapping theorem, $(\cM(X), \|\cdot \|_\KR)$ is not a complete space.

	A closely related norm on $\cM(X)$ that is often considered is
	\begin{align*}
		\| \mu \|_{\tKR(X)}
			:= 
		 |\mu(X)|
		 +
		\sup\bigg\{ \int_{X} 
					\psi \dd \mu
				\ : \
				\psi \in \cC(X), \ 
				\psi(x_0) = 0, \
				\Lip(\psi) \leq 1
			\bigg\},
	\end{align*} 
	for some fixed $x_0 \in X$;
	see \cite[Section 8.10(viii)]{Bogachev}. 
	The next result shows that these norms are equivalent.
	
	\begin{proposition}
		Let $(X,d)$ be a compact metric space. 
		For $\mu \in \cM(X)$ we have
		\begin{align*}
			\| \mu \|_{\KR(X)}
			\leq
			\| \mu \|_{\tKR(X)}
			\leq
			c_X \| \mu \|_{\KR(X)},
		\end{align*}
		where $c_X < \infty$ depends only on $\diam(X)$.
	\end{proposition}

	\begin{proof}
		We start with the first inequality.
		Let $\psi \in \cC(X)$ with 
			$\| \psi \|_\infty \leq 1$ and
			$\Lip(\psi) \leq 1$. 
		Define $\phi := \psi - \psi(x_0)$, so that
		$\phi(x_0) = 0$ and $\Lip(\phi) = \Lip(\psi) \leq 1$.
		Then
		\begin{align*}
			\int \psi \dd \mu
			= \int \psi(x_0) + \phi \dd \mu
			= \psi(x_0) \mu(X) + \int \phi \dd \mu
			\leq |\mu(X)| + \int \phi \dd \mu
			\leq \| \mu \|_{\tKR}.
		\end{align*}
		Taking the supremum over $\psi$ yields the desired bound.

		\smallskip

		Let us now prove the second inequality.
		Set $\Delta := 1 \vee \diam(X)$.
		Take $\psi \in \cC(X)$ with 
			$\psi(x_0) = 0$ and $\Lip(\psi) \leq 1$.
		Then 
			$|\psi(x)| 
			= |\psi(x) - \psi(x_0)| \leq d(x,x_0) \leq \diam(X) \leq \Delta$ for all $x \in X$,
		so that
			$\| \frac{\psi}{\Delta}\|_\infty \leq 1$ and 
			$\Lip(\frac{\psi}{\Delta}) \leq 1$.
		We obtain
		\begin{align*}
			\int \psi \dd \mu 
			= \Delta \int \frac{\psi}{\Delta} \dd \mu 
			\leq \Delta \| \mu \|_{\KR}.
		\end{align*}
		Moreover, $|\mu(X)| \leq \| \mu\|_{\KR}$ as can be seen by taking $\psi = \pm 1$ in \eqref{eq:KR}
		It follows that 
		\begin{align*}
			\| \mu \|_{\tKR}
			\leq (1 + \Delta) 
			 \| \mu \|_{\KR},
		\end{align*}
		as desired.
	\end{proof}

	\begin{proposition}[Relation to $\bW_1$]
		\label{rem:W1}
	Let $(X,d)$ be a compact metric space. 	
	If $\mu_1, \mu_2 \in \cM_+(X)$ 
	are nonnegative measures of equal total mass, 
	we have 
		$\| \mu_1 - \mu_2 \|_{\tKR}
			 = 
		\bW_1(\mu_1, \mu_2)
		$.
	\end{proposition}

	\begin{proof}
		This follows from the Kantorovich duality for the distance $\bW_1$. 
	\end{proof}

	On the subset of \emph{nonnegative} measures, the $\KR$-norm induces the weak$^*$ topology:

	\begin{proposition}[Relation to weak$^*$-convergence]
	\label{prop:KR-weakstar}
	Let $(X,d)$ be a compact metric space. 	
	For
		$\mu_n, \mu \in \cM_+(X)$ 
	we have
	\begin{align*}
		\mu_n \to \mu \, \text{weakly}
	\quad \text{if and only if} \quad
		\| \mu_n - \mu \|_{\KR} \to 0.
	\end{align*}
	\end{proposition}

	\begin{proof}
		See \cite[Theorem 8.3.2]{Bogachev}.
	\end{proof}
	
	\begin{remark}[Testing against smooth functions]
		If $X = \T^d$,
		the space of $\cC^1$ functions $\psi$ with 
			$\Lip(\psi) \leq 1$ 
		is dense in the set of Lipschitz functions with $\Lip(\psi) \leq 1$; see, e.g., \cite[Proposition A.5]{Schmitzer-Wirth:2019}.
		Consequently, 
		\begin{align}
			\label{eq:smooth_case}
			\| \mu \|_{\KR(X)}
			= 
		\sup\bigg\{ \int_X 
					\psi \dd \mu
				\ : \
				\psi \in \cC^1(\T^d), \ 
				\| \psi \|_\infty \leq 1, \
				\|\nabla \psi\|_\infty \leq 1
			\bigg\}.
		\end{align}
	\end{remark}		

	\begin{remark}
		The identity \eqref{eq:smooth_case} shows that  $\|\cdot\|_{\KR}$ is the dual norm of the separable Banach space $\cC^1(Q)$. 
		The dual space of $\cC^1(Q)$ is a strict superset of the finite Borel measures.
	\end{remark}	
	
	\section{Norms on curves in the space of measures}
	\label{sec:norms}

	We work with curves of bounded variation taking values in the space $\cM_+(\T^d)$.

	\begin{definition}[Curves of bounded variation]
	The space
		$\BVKR$
	consists of all curves of measures 	
		$\bfmu : \cI \to \cM_+(\Td)$ 
	such that the $\BV$-seminorm
	\begin{align}
		\label{eq:BV-def}
		\|\bfmu \|_{\BVKR} 
		:= \sup\left\{ 	
			\int_\cI 
				\int_{\T^d}
					\partial_t \phi_t 
				\dd \mu_t		
			\dd t 
			\ : \ 
			\phi \in \cC_c^1\big(\cI; \cC^1(\Td)\big), 
			\ 
			\max_{t\in \cI}
			\|\phi\|_{\cC^1(\Td)} 
			\leq 1 
			\right\}
	\end{align}
	is finite.
	\end{definition}
	
	\begin{remark}
		The space 
			$\BVKR$
		is a (non-closed) subset of the space 
			$\BV(\cI; X^*)$,
		where $X$ is the separable Banach space 
			$\cC^1(\T^d)$.
		We refer to 
			\cite[Section 2]{Heida-Patterson-Renger:2019} 
		for the equivalence of several definitions of 
		$\BV\big(\cI; X^* \big)$.
	\end{remark}

	\begin{definition}
		The space 
			$W^{1,1}_\KR(\cI; \cM_+(\T^d))$
		consists of all curves 
			$(\mu_t)_{t \in \cI}$ 
		in the Banach space-valued Sobolev space 
			$W^{1,1}\big(\cI; (\cC^1(\T^d))^*\big)$
		such that 
			$\mu_t \in \cM_+(\T^d)$ 
		for a.e.\ $t \in \cI$.
	\end{definition}

\section{Domain property of convex functions}

\begin{lemma}[Domain properties of convex functions]
	\label{lemma:prop_domf}
	Let $f : \R^n \to \R \cup \{ + \infty\}$ be convex, 
	and let 
		$x^\circ \in \Dom(f)^\circ$. 
	For every $\lambda \in (0, 1)$ and every bounded  set 
		$K \subseteq \Dom(f)$, 
	there exists a compact convex set 
		$K_{\lambda} \subseteq \Dom(f)^\circ$ 
	such that 
	\begin{align*}
		(1-\lambda) K + \lambda x^\circ 
			\subseteq
		K_{\lambda}.
	\end{align*}
	\end{lemma}
	
	\begin{proof}
		Let $K \subseteq \Dom(f)$ be  bounded  
		and $\lambda \in (0,1)$.
		Since $x^\circ \in \Dom(f)^\circ$, 
		we can pick $r > 0$ such that 
			$B(x^\circ, r) \subseteq \Dom(f)^\circ$.
		Fix  $y \in \bar K$  and set 
			$y_\lambda := (1-\lambda) y + \lambda x^\circ$.
		We claim that 
			$B( y_\lambda, \lambda r) \subseteq \Dom(f)^\circ$.
		
		To prove the claim, it suffices to show that
		$B( y_\lambda, \lambda r) \subseteq \Dom(f)$, since $B( y_\lambda, \lambda r)$ is open. 
		Take $z \in B( y_\lambda, \lambda r)$  and pick a sequence $(y_n)_n \subset K$ such that $y_n \to y$. Observe that 
			$z = (1-\lambda) y_n + \lambda \tilde x_n$
		with $\tilde x_n \in B(x^\circ, r)$ if $n$ is large enough (indeed, 
			$\tilde x_n - x^\circ 
				= \frac{1}{\lambda}(z - y_\lambda)+ \frac{1-\lambda}{\lambda}(y-y_n)
			$ and $|z - y_\lambda| < \lambda r$
		).
		Since $y_n, \tilde x_n \in \Dom(f)$, 
		the claim follows by convexity of $f$.

		We now define			
		\begin{align*}
			C_{\lambda} 
				:= \bigcup_{y \in K} 
						B\Big(y_\lambda, \frac{\lambda r}{3} \Big)
		\tand 
			K_\lambda
				:= \Conv( \overline{ C_\lambda } ).
		\end{align*}
		By construction, $K_{\lambda}$ is convex, bounded, and closed, thus compact. 
		Let us show that 
			$K_{\lambda} 	
				\subseteq \Dom(f)^\circ$. 
		
		By convexity of $f$, it suffices to show that 
			$\overline{C_\lambda} 
				\subseteq \Dom(f)^\circ$. 
		Pick $z \in \overline{C_\lambda}$
		and $\{z_n\}_n \subseteq C_\lambda$ such that 
			$z_n \to z$.
		Then there exists $y_n \in K$ such that 
			$z_n \in 
				B\big(
					(y_n)_\lambda, \frac{\lambda r}{3} 
				\big)
			$. 
		Passing to a subsequence, we may assume that 
			$y_n \to \bar y$ for some $\bar y \in \bar K$ 
		and $z_n \in 
				B\big(
					{\bar y}_\lambda, \frac{\lambda r}{2} 
				\big)
			$ for $n \geq \bar{n} \in \N$. 
		Taking the limit as $n \to +\infty$ we infer that
			$z \in \overline{
						B\big(
							{\bar y}_\lambda, 
							\frac{\lambda r}{2} 
						\big)}$.
		Since 
			$B\big(
				{\bar y}_\lambda, 
				\lambda r
			\big) 
			\subseteq \Dom(f)^\circ$, 
		it follows that $z \in \Dom(f)^\circ$.
	\end{proof}

	\section{Notation}
	For the convenience of the reader we collect some notation used in this paper.
	
	\begin{tabular}{ll}
		$A^\circ$ & topological interior of a set $A$
		\\
		$\Dom(F)$ & domain of a functional $F$: $\Dom(F) = \{x \in \cX : F(x) < \infty \}$.		\\
	$\cI$ & bounded open time interval.  
		\\
	$\cM^d(A)$ 
	& 
	the space of finite $\R^d$-valued Radon measures on $A$.
		\\
	$\cM_+(A)$ 
	& 
	the space of finite (positive) Radon measures on $A$. 
		\\
	$\cX^Q$
	&
	the set of all $x \in \cX$ with $x_\sz = 0$.
		\\
	$\cE^Q$
	&
	the set of all $(x,y) \in \cE$ with $x_\sz = 0$.
		\\
	$\R_a^\cE$ 
	& 
	the set of anti-symmetric real functions on $\cE$.
		\\
	$\T_\eps^d$, $\Z_\eps^d$
	&
	the discrete torus of mesh size $\eps>0$:
		$\T_\eps^d = ( \eps \Z/\Z )^d = \eps \Z_\eps^d$.
		\\
	$\Eff(J)$
	&
	the effective flux of $J$: 
		$\Eff(J) = \frac12 \sum_{(x,y) \in \cE^Q} J(x,y) (y_\sz - x_\sz)$.
		\\
	$\Rep(\rho)$
	&
	the set of representatives of $\rho \in \R_+$, i.e, all $m \in \R_+^\cX$ s.t. $\sum_{x \in \cX^Q} m(x) =\rho$.
	\\
	$\Rep(j)$
	&
	the set of representatives of $j \in \R^d$, i.e, all $J \in \R_a^\cE$ divergence-free and s.t.
		\\
			&$\frac12\sum_{(x,y) \in \cX^Q} J(x,y) (y_\sz - x_\sz) = j$.
	\\
	$\Rep(\rho,j)$
	&
	the set of representatives of $\rho \in \R_+$, $j \in \R^d$: $\Rep(\rho,j) = \Rep(\rho) \times \Rep(j)$.
	\\
	$Q_\eps^z$
	&
	the cube of size $\eps>0$ centered in $\eps z \in \Td$: for $z \in \Z_\eps^d$, $Q_\eps^z := [0,\eps)^d + \eps z$.
		\\
	$S^{\bar z}$
	& 
	shift operator: $S_\eps^{\bar z} : \cX \to \cX$, \ 
	$
		S_\eps^{\bar z}(x) = (\bar z + z, v)$
	for
		$x = (z,v) \in \cX$.
		\\
%	$S^{\bar z}$ 
	& 
	shift operator: $S_\eps^{\bar z}: \cE \to \cE$, \ 
	$
		S_\eps^{\bar z}(x,y) := 
			\big(
				S_\eps^{\bar z}(x), 
				S_\eps^{\bar z}(y) 
			\big)	
	$
	for $(x,y) \in \cE_\eps$\\
	$\sigma^z$ & 
	$\sigma_\eps^{\bar z} \psi: \cX_\eps \to \R,
	\quad 
	(\sigma_\eps^{\bar z} \psi)(x) := \psi(S_\eps^{\bar z}(x)) 
	\quad 
	\text{for} \; x \in \cX_\eps.$
	\\
		&
	$\sigma_\eps^{\bar z}J: \cE_\eps \to \R, \quad
		(\sigma_\eps^{\bar z} J)(x,y) := J(S_\eps^{\bar z}(x,y)) 
		\quad 
		\text{for} \; (x,y) \in \cE_\eps.
	$\\
	$T_\eps^{\bar z}$ 
		& 
	rescaling operator: $T_\eps^{\bar z} : \cX \to \cX_\eps$: 
	$
		T_\eps^{\bar z}(x) = (\eps(\bar z + z), v)$
	for
		$x = (z,v) \in \cX$.\\
	$\tau_\eps^z$ & 
		$\tau_\eps^{\bar z} \psi 
		: \cX \to \R, 
		\quad
		\big(\tau_\eps^{\bar z} \psi \big)(x)
		:=
		\psi\big( T_\eps^{\bar z}(x) \big)
		\quad \text{ for } 
		x \in \cX$.
	\\
		&
		$\tau_\eps^{\bar z} J : \cE \to \R, \quad  
		\big( \tau_\eps^{\bar z} J \big)(x,y)
		:=
		J
		\big( 
		T_\eps^{\bar z}(x),
		T_\eps^{\bar z}(y)
		\big)
		\quad \text{ for } 
		(x,y) \in \cE$.
	\\
	$\cCE$
	&
	discrete continuity equation: $(\bfm,\bfJ) \in \cCE$ iff
	$
		\partial_t m_t + \dive J = 0
	$ on $(\cX,\cE)$.
	\\
	$\bCE$
	&
	continuous continuity equation: $(\bfmu,\bfnu) \in \bCE$ iff
	$
	\partial_t \mu_t + \nabla \cdot \bfnu = 0
	$ on $\Td$.
	\\
	$\BV$
	& 
	more precisely $\BVKR$: the space of time-dependent curves of 
		\\
			&(positive) measures with bounded variation with respect to the $\KR$ norm
			\\
				&(Kantorovich-Rubenstein) on $\cM_+(\Td)$.
	\\
	$W^{1.1}$
	&
	more precisely $W^{1,1}_\KR(\cI; \cM_+(\T^d))$: the space of time-dependent curves of 
		\\
			&(positive) measures belonging to the Banach space $W^{1,1}\big(\cI; (\cC^1(\T^d))^*\big)$.
	\\
	$\tP_\eps\mu, \tP_\eps \nu$
	&
	discretisation of $\mu \in \cM_+(\Td)$, $\nu \in \cM^d(\Td)$: for $z \in \Z_\eps^d$, ($\tP_\eps \mu(z), \tP_\eps \nu(z)) \in$
		\\
			&$ \R_+ \times \R^d$, given by $\tP_\eps \mu(z) = \mu(Q_\eps^z)$, $\tP_\eps \nu(z) = 
			\big(
				(\nu \cdot e_i)(\partial Q_\eps^z\cap \partial 
				Q_\eps^{z+ e_i})
			\big)_i
			$.
	\end{tabular}

		\medskip
	In the paper we use some standard terminology from graph theory. 
	Let $(\cX, \cE)$ be a locally finite graph. 
	A \emph{discrete vector field} is an anti-symmetric function $J : \cE \to \R$.
	Its \emph{discrete divergence} is the function $\dive J : \cX \to \R$ defined by	
	\begin{align}
	\label{eq:divergence}
		\dive J(x) 
		:= \sum_{y \sim x} J(x,y). 
	\end{align}
	We say that $J$ is \textit{divergence-free} if $\dive J = 0$.

\subsection*{Acknowledgements}

J.M. gratefully acknowledges support by the European Research Council (ERC) under the European Union's Horizon 2020 research and innovation programme (grant agreement No 716117).
J.M and L.P. also acknowledge support from the Austrian Science Fund (FWF), grants No F65 and W1245. 
E.K. gratefully acknowledges support by the German Research Foundation through the Hausdorff Center for Mathematics and the Collaborative Research Center 1060.
P.G. is partially funded by the Deutsche Forschungsgemeinschaft (DFG, German Research Foundation) -- 350398276.

\bibliographystyle{my_alpha}
\bibliography{Multiple-dimensions}{}
\end{document}